%% file: RelativeTest-3-26-17.tex
\documentclass[11pt]{amsart}
\usepackage{calc,amssymb,amsthm,amsmath}
\usepackage{mathtools}
\RequirePackage[dvipsnames,usenames]{xcolor}
\usepackage{hyperref}
\usepackage{tikz-cd}

\hypersetup{
bookmarks,
bookmarksdepth=3,
bookmarksopen,
bookmarksnumbered,
pdfstartview=FitH,
colorlinks,backref,hyperindex,
linkcolor=Sepia,
anchorcolor=BurntOrange,
citecolor=MidnightBlue,
citecolor=OliveGreen,
filecolor=BlueViolet,
menucolor=Yellow,
urlcolor=OliveGreen
}
\usepackage{alltt}
\usepackage{multicol}
\usepackage{etex}
\usepackage{xspace}
\usepackage{rotating}
\newcommand{\ie}{{\itshape ie} }
\usepackage{mabliautoref}
\usepackage{colonequals}
\input{kmacros3.sty}
\usepackage{stmaryrd}
\usepackage{tikz}
\usepackage[left=1.0in,top=1.0in,right=1.0in,bottom=1.0in]{geometry}

\usepackage{verbatim}
\usepackage{enumerate}
\theoremstyle{theorem}

\newtheorem*{theoremA}{Theorem A}
\newtheorem*{theoremB}{Theorem B}
\newtheorem*{theoremC}{Theorem C}
\newtheorem*{theoremD}{Theorem D}
\newtheorem*{theoremE}{Theorem E}
\theoremstyle{remark}
\newtheorem*{*setting}{Setting}

\theoremstyle{remark}

\usepackage{graphicx}

\usepackage[all,cmtip]{xy}

\newcommand{\mysubsection}[1]{\setcounter{subsection}{\value{theorem}} \subsection{#1} \addtocounter{theorem}{1}}

\renewcommand{\O}{\mathcal{O}}
\renewcommand{\ie}{{\itshape i.e.}}
\renewcommand{\cf}{{\itshape cf.\xspace\xspace}}

\renewcommand{\phi}{\varphi}

\renewcommand{\theta}{\vartheta}

\renewcommand{\epsilon}{\varepsilon}

\renewcommand{\to}[1][]{\xrightarrow{\ #1\ }}
\newcommand{\onto}[1][]{\protect{\xrightarrow{\ #1\ }\hspace{-0.8em}\rightarrow}}

\renewcommand{\theenumi}{\alph{enumi}}
\begin{document}

\title {$F$-singularities in families}

\author{Zsolt Patakfalvi }
\address{EPFL, SB MATHGEOM CAG, MA B3 635 (B\^atiment MA), Station 8, CH-1015 Lausanne}
\email{zsolt.patakfalvi@epfl.ch}

\author{Karl Schwede}
\address{Department of Mathematics\\ University of Utah\\ Salt Lake City\\ UT 84112}
\email{schwede@math.utah.edu}

\author{Wenliang Zhang}
\address{Department of Mathematics\\ University of Illinois at Chicago\\Chicago\\  IL 60607}
\email{wlzhang@uic.edu}

\begin{abstract}
We study the behavior of test ideals and $F$-singularities in families. In particular, we obtain generic (and non-generic) restriction theorems for test ideals and non-$F$-pure ideals, which imply for example openness of most of the $F$-singularity classes when the relative canonical sheaf is $\bQ$-Cartier. Additionally, we study the global behavior of certain canonical linear systems (induced by Frobenius) associated to adjoint line bundles, in families. As a consequence, we obtain some positivity results for pushforwards of some adjoint line bundles and for certain subsheaves of these.
%\textbf{Keywords: }
\end{abstract}

\thanks{The second author was partially supported by a Fellowship from the Sloan Foundation as well as NSF grant DMS \#1064485, NSF FRG Grant DMS \#1265261/1501115 and NSF CAREER Grant DMS \#1252860/1501102.}
\thanks{The third author was partially supported by the NSF grants DMS \#1068946/\#1247354 and \#1606414}
\thanks{The material is based upon work supported by the National Science Foundation under Grant No. 0932078 000, while the second and third authors were in residence at the Mathematical Science Research Institute (MSRI) in Berkeley, California, during the spring semester 2013.}
\maketitle
\numberwithin{equation}{theorem}
\tableofcontents

%\todo{ {\bf Zsolt:} Can we answer J\'anos's question with slc replaced by sharply F-pure? The question is: $f : (X , \Delta) \to V $ is slc family, i.e., $V$ is a smooth curve and $(X,\Delta + X_s)$ is slc for every $s \in V$,  is then any pullback of $f$ to another smooth curve is also an slc family?}
\section{Introduction}

%\todo{ {\bf Zsolt:} Karl, can you put in references for the next sentence?}

In \cite{KunzCharacterizationsOfRegularLocalRings} Kunz proved that a scheme over a field of characteristic $p>0$ is regular if and only if the Frobenius (endo)morphism is flat, which initiated the study of singularities using the Frobenius morphism. Classes of singularities defined via the Frobenius morphism are referred to usually as $F$-singularities. Since Kunz's work, the study of $F$-singularities has become an active research area, \cf \cite{HochsterRobertsFrobeniusLocalCohomology,GotoWatanabeTheStructureOfOneDimensionalFPureRings,FedderFPureRat,MehtaRamanathanFrobeniusSplittingAndCohomologyVanishing,RamananRamanathanProjectiveNormality,SrinivasFPureType,SmithFRatImpliesRat,HaraRatImpliesFRat,MehtaSrinivasRatImpliesFRat,HaraYoshidaGeneralizationOfTightClosure,TakagiInterpretationOfMultiplierIdeals}. However, the methods of $F$-singularities have been widely applied to global geometry over a field of positive characteristic only recently \cite{SchwedeACanonicalLinearSystem,
MustataNonNefLocusPositiveChar,HaconSingularitiesOfThetaDivisors,MustataSchwedeSeshadri,ZhangPluriCanonicalMapOfVarietiesOfMaximalAlbaneseDimension,
CasciniHaconMustataSchwedeNumDimPseudoEffective,PatakfalviSemipositivity,HaconXuThreeDimensionalMinimalModel,TanakaTraceOfFrobenius}, at least outside of special classes of varieties \cite{MehtaRamanathanFrobeniusSplittingAndCohomologyVanishing,BrionKumarFrobeniusSplitting}. A large field within global geometry is moduli theory, which requires understanding how varieties behave in flat families.  In this direction, there have not been any positive results on the behavior of test ideals in families.  We fill this gap.

Previously $F$-rational, Cohen-Macaulay $F$-injective and (Gorenstein) $F$-pure singularities have been studied in such a context \cite{HashimotoCMFinjectiveHoms,ShimomotoZhangOnTheLocalizationProblem,HashimotoFPureHoms}.  Additionally, \cite[Example 4.7]{MustataYoshidaTestIdealVsMultiplierIdeals} showed that the test ideal $\tau$ does not satisfy the generic restriction theorem, at least as stated for multiplier ideals \cite[Theorem 9.5.35]{LazarsfeldPositivity2}.  We develop tools tackling this issue, obtaining generic (and non-generic)
restriction theorems for test ideals, see Theorem A below.

%\todo{ {\bf Zsolt: } Should I make the above sentence more concrete? My problem is that I do not know which are the questions that were not known.}

In a flat family $f : X \to V$ the absolute Frobenius on $X$ does not restrict to the Frobenius morphism on each fiber. Thus we systematically study the relative Frobenius morphism\footnote{Also called the Radu-Andr\'e morphism, especially in commutative algebra.} $X' \to X \times_V V'$ where $X' \to X$ and $V' \to V$ are the Frobenius morphisms of $X$ and $V$, respectively. Its fibers over $V$ are morphisms between thickenings  of the fibers of $f$. However, its fibers over $V'$ are exactly Frobenius morphisms of fibers of $f$ (at least over points with perfect residue field). %In summary, the article aims to systematically study the relative Frobenius morphism from the point of view of $F$-singularity theory.

We begin by stating our main result in the local setting.  We denote by $X^n$ and $V^n$ the domains of $n$-iterated Frobenii of $X$ and $V$ respectively.  Let $\tau$ and $\sigma$ denote the test ideal\footnote{Technically, we are working with what has recently become known as the \emph{big test ideal} \cite{HochsterFoundations}.} \cite{HochsterHunekeFRegularityTestElementsBaseChange} and non-sharply-$F$-pure ideal \cite{FujinoSchwedeTakagiSupplements,BlickleBockleCartierModulesFiniteness}, respectively.  These ideals play key roles in the theory of $F$-singularities, the test ideal $\tau$ being the unique smallest non-zero ideal fixed by $p^{-e}$-linear maps and $\sigma$ being the unique largest.  We define relative versions of these two ideals by using the relative instead of absolute Frobenius.
The iterated relative Frobenius targets different spaces $X \times_V V^{n}$, and so we actually define a sequence ideals, one on each of these spaces.  We explain the setup.

Suppose that $f : X \to V$ is a finite-type flat, equidimensional, reduced and S2 and G1 morphism of Noetherian $F$-finite schemes with $V$ integral.
Additionally suppose $\Delta$ is a $\bQ$-divisor on $X$ satisfying suitable conditions which allow it to be restricted to fibers and such that $K_{X/V} + \Delta$ is $\bQ$-Cartier with index not divisible by $p > 0$ so that $(p^e - 1) \left(K_{X/V} + \Delta \right)$ is Cartier (see \autoref{rem:relative_canonical} for a precise statement).  Then for each integer $n > 0$, divisible by $e$, we define ideals $\sigma_n(X/V, \Delta)$ and $\tau_n(X/V, \Delta) \subseteq  \sO_{X \times_V V^n}$ called the \emph{relative non $F$-pure} and  \emph{relative test ideals} (respectively).

Our main local theorem is that these ideals restrict to absolute non-$F$-pure and absolute test ideals on \emph{all} of the geometric fibers, and can be used to prove \emph{generic} restriction theorems for the usual absolute test ideals on $X \times_V V^n$ at least if $V$ is regular.

\begin{theoremA} \textnormal{(\autoref{thm.UniversalNForAllClosedFibers}, \autoref{cor.SigmaRestrictsToAllPoints}, \autoref{cor.RestrictionOfRelativeTauToFibers}, \autoref{cor.RestrictionOfTauForDivisors})}
\label{thm.IntroductionRestrictionOfTauSigma}
With notation as above, there exists an $N > 0$ such that for every \emph{perfect point}\footnote{By definition, a morphism $\Spec K \to V$ from a perfect field $K$, see \autoref{def.PerfectPoint}.} $s \in V$, and every $n \geq N$
\[
\sigma_{ne}(X/V, \Delta)\cdot \O_{X_{s^{ne}}} =  \sigma(X_s, \Delta|_{X_s})
\]
and
\[
\tau_{ne}(X/V, \Delta) \cdot \O_{X_{s^{ne}}} =  \tau(X_s, \Delta|_{X_s})
\]
where $X_{s^{ne}}$ is the fiber of $X \times_V V^{ne} \to V^{ne}$ over $s^{ne} \in V^{ne}$, which is isomorphic to $X_s$ since $k(s)$ is perfect.  Additionally, both $\sigma_n$ and $\tau_{ne}$ map surjectively onto their arbitrary base changes.

Furthermore, if $V$ is regular and $N$ is sufficiently large, then for all perfect points $s \in V$ we have that the absolute non-$F$-pure ideal restricts to all of the fibers  for $n \geq N$
\[
\sigma(X \times_V V^{ne}, \Delta \times_V V^{ne})\cdot \O_{X_{s^{ne}}} =  \sigma(X_s, \Delta|_{X_s})
\]
and at least for an open dense set of the base $U \subseteq V$ that the same holds for the absolute test ideal
\[
\tau(X\times_V V^{ne}, \Delta\times_V V^{ne})\cdot \O_{X_{s^{ne}}} = \tau(X_s, \Delta|_{X_s})
\]
for all perfect points $s \in U$.
\end{theoremA}

Additionally, we show that over a dense open subset $U$ of $V$ with $W = f^{-1}(U)$,  $\tau_{ne}(X/V, \Delta)|_W$ coincides with the absolute test ideal $\tau(X \times_V V^{ne}, \Delta \times_V V^{ne})|_W$ and likewise $\sigma_{ne}(X/V, \Delta)|_W = \sigma(X \times_V V^{ne}, \Delta \times_V V^{ne})|_W$ in \autoref{thm.RelativeTauVsAbsoluteOverOpen} and \autoref{thm.RelativeSigmaVsAbsoluteOverOpen} respectively.

The key method that allows to prove this result (at least stated for $\sigma_n$) is that if $m > n$, then $\Image\big(\sigma_{ne}(X/V, \Delta) \otimes_{\O_{X \times V^{ne}}} \O_{X \times V^{me}} \to \O_{X \times V^{me}}\big) \supseteq \sigma_{me}(X/V, \Delta)$ and there is a dense open set of the base $V$ over which equality holds for all sufficiently large $n$, this is \autoref{prop.StabilizingSigma}.  This stabilization result should be viewed as a relative version of \cite[Proposition 1.11]{HartshorneSpeiserLocalCohomologyInCharacteristicP}, \cite[Proposition 4.4]{LyubeznikFModulesApplicationsToLocalCohomology}, and \cite[Remark 13.6]{Gabber.tStruc}.

\begin{remark}
If one replaces the divisor $\Delta$ with the language of principal Cartier algebras, then the previous result still holds without the technical assumptions about divisors from \autoref{rem:relative_canonical}.
\end{remark}

\begin{remark}
There are a number of subtle issues in the statement above that we are suppressing.  In particular, $\tau_n(X/V, \Delta)$ depends on the choice of some ideal $I$ contained within the test ideal of every fiber, and shown to exist in \autoref{prop.ExistRelativeTestElements}.
\end{remark}

%\todo{ {\bf Zsolt:} It is not clear in which context we would like to state the previous theorem.\\
%{\bf Karl:} What about this?\\
%{\bf Wenliang:} In the statement, do we want to say `fix $N\gg 0$', or rather, `there exists $N\gg 0$ such that for every $n\geq N$...'?\\
%{\bf Karl:}  I reworded this.}

A particularly important related question is that of deformation of sharply $F$-pure singularities in flat families with $\bQ$-Cartier relative canonical divisors.
%\todo{{\bf Wenliang:} Where was this `deformation question'?}
This would be important for a positive characteristic construction of the moduli of stable varieties, also known as the KSBA compactification. In characteristic zero, this is the moduli space given by the log-minimal model program. It classifies log-canonical models, hence birational equivalence classes of varieties of general type, and furthermore it contains some nodal varieties for the compactification. There is a conjectural framework of constructing this moduli space \cite{KollarProjectivityOfCompleteModuli}. One of the main ingredients in this framework is to prove that log-canonical singularity deforms in flat families with $\bQ$-Cartier relative log-canonical divisor. An important step in this direction in positive characteristic is the corresponding statement for sharply $F$-pure singularities. It is also an important ingredient in an upcoming
paper of the first author where he is planning to address the question of the existence of an algebraic space structure on the space of sharply $F$-pure stable varieties.  In this paper, we handle the deformation of $F$-pure and $F$-regular singularities.  Indeed, openness of  $F$-pure and $F$-regular singularities is a direct consequence of Theorem A above.

\begin{theoremB}\textnormal{(Deformation of $F$-pure and $F$-regular singularities: \autoref{cor.OpennessOfSharpFPurityForDivisors}, \autoref{cor.OpennessOfFregDivisors})}
Suppose $f : X \to V$ and $\Delta$ is as in Theorem A and additionally assume that $f$ is proper.  If $s \in V$ is a perfect point and $(X_s, \Delta|_{X_s})$ is sharply $F$-pure (respectively, strongly $F$-regular\footnote{In which case, you can even remove the index not divisible by $p$ assumption from $K_{X/V} + \Delta$, see \autoref{cor.OpennessOfFregDivisorsBadIndex}.}) then there exists an open set $U \subseteq V$ such that for all $u \in U$ that $(X_u, \Delta|_{X_u})$ is also sharply $F$-pure (respectively, strongly $F$-regular).
\end{theoremB}

We also build relative test submodules and non-$F$-injective submodules of $\omega_{X \times_V V^{ne}/V^{ne}}$ and prove restriction theorems like Theorem A for them \autoref{cor.RestrictionTheoremForSigmaOmega} and \autoref{cor.RestrictionTheoremForTauOmega}. As a consequence in \autoref{thm.OpenSetNakayama} we reprove a result of M.~Hashimoto \cite{HashimotoCMFinjectiveHoms}, deformation for Cohen-Macaulay $F$-injectivity and $F$-rationality.

Furthermore we apply our setup to global questions. One of the reasons for the recent global applications of $F$-singularity theorem is the lifting theorem shown by the second author in \cite[Proposition 5.3]{SchwedeACanonicalLinearSystem}. This theorem can be used to replace some of the lifting arguments that use Kodaira vanishing in characteristic zero. One of the fundamental ideas in \cite{SchwedeACanonicalLinearSystem} is to try to lift only a big enough set of sections of adjoint bundles instead of all the sections. This canonical set of sections for a pair $(X, \Delta)$ with $(1-p^e)(K_X + \Delta)$ Cartier and for a line bundle $M$ is defined as \cite[Definition 4.1]{SchwedeACanonicalLinearSystem}
\begin{equation*}
S^0 \big(X, \sigma(X, \Delta) \otimes M \big) := \bigcap_{m >0}  \im \Big( H^0 \big(X, F^{me}_* \sO_{X}((1-p^{me})(K_{X} + \Delta) )\otimes M   \big)  \to H^0\big(X, M\big) \Big) .
\end{equation*}
First, we investigate questions about how  this canonical space of sections behaves in families: semicontinuity, stabilization of the intersection, etc.

\begin{theoremC}
Let $f : X \to V $ and $\Delta$ be as in Theorem A, with $f$ projective and $V$ regular. Further, suppose that $M$ is a line bundle on $X$ such that $M - K_{X/V} - \Delta$ is $f$-ample (where by $M$ we actually mean the Cartier divisor corresponding to $M$). Then the following statements hold.
\begin{itemize}
\item[(a)] {\textnormal{[\autoref{cor:S^0_uniform_stabilization}, \autoref{ex:S_f_*_base_change_not_isomorphism}, \autoref{ex:S_f_*_base_change_not_isomorphism_arbitrary_power}, \autoref{ex.NotLowerSemiContEither}]}, \cf \cite[Example 5.5]{HaraRatImpliesFRat}, \cite[Theorem 8.3]{TanakaTraceOfFrobenius}}.  The function
\begin{equation}
\label{eq:thmC}
s \mapsto \dim_{k(s)}S^0(X_{s},\sigma(X_{s},\Delta_s) \otimes M_s)
\end{equation}
is not semicontinuous (here $s \in V$ is a perfect point) in either direction, however, there is a dense open subset $U \in V$, such that the function  \autoref{eq:thmC}
%$s \mapsto \dim_{k(s)}(S^0(X_{s},\sigma(X_{s},\phi_s) \otimes M_s)$
is constant on $U$.
\item[(b)] {\textnormal{[\autoref{prop:S^0_uniform_stabilization}]}}. There exists $n>0$ so that for all integers $m \geq n$ and perfect points  $s \in V$,
\begin{multline*}
\qquad \qquad \im \left( H^0\left(X_{s}, F^{me}_* \sO_{X_s}((1-p^{me})(K_{X_s} + \Delta_s)\right) \otimes M_{s}  )  \to H^0(X_{s}, M_{s}) \right)
\\ = S^0(X_{s}, \sigma(X_{s}, \Delta_{s}) \otimes M_{s}).
\end{multline*}
\item[(c)] \label{itm:thmC_maximal_open} {\textnormal{[\autoref{prop:S_0_maximal_is_open}]}}.  If there is a perfect point $s_0 \in V$ such that
\begin{equation*}
 H^0(X_{s_0},M_{s_0})=S^0(X_{s_0}, \sigma(X_{s_0}, \Delta_{s_0}) \otimes M_{s_0}),
\end{equation*}
then there is an open neighborhood $U$ of $s_0$, such that
 $f_* M|_U$ is locally free and compatible with base change and
\begin{equation*}
H^0(X_{s},M_{s})=S^0(X_{s}, \sigma(X_{s}, \Delta_{s}) \otimes M_{s})
\end{equation*}
for every perfect point $s \in U$. In particular, the function \autoref{eq:thmC}
%$\dim_{k(s)} S^0(X_{s}, \sigma(X_{s}, \Delta_{s}) \otimes M_{s})$
is constant for $s \in U$.
\end{itemize}
\end{theoremC}

We would like to also mention the following natural question left open by Theorem C.

\begin{question*}
Can one remove the $f$-ampleness assumption from the  statements of Theorem C?
\end{question*}

By Theorem C, it does make sense to talk about general value of $\dim_{k(s)} S^0(X_{s},\sigma(X_{s},\Delta_s) \otimes M_s)$ in the  following theorem.

\begin{theoremD}
With assumptions as in Theorem C such that $V$ is projective over a perfect field and $(p^e - 1)(K_X + \Delta)$ is Cartier.  Then for every $n \gg0$ there is a subsheaf $S^0_{\Delta,ne} f_* (M)$  of $(F_V^{ne})^*(f_* M) $, for which the following holds.
\begin{itemize}
 \item[(a)] \label{itm:thmD_rank} {\textnormal{[\autoref{cor:S^0_uniform_stabilization}]}}.  The rank of $S^0_{\Delta,ne} f_* (M)$ is the general value of $\dim_{k(s)} S^0(X_{s},\sigma(X_{s},\Delta_s) \otimes M_s)$.
\item[(b)] {\textnormal{[\autoref{prop:global_generation_of_S_f_*}]}}.  If $M -K_{X/V} - \Delta$ is ample, then $S^0_{\Delta,ne} f_* (M)$ is globally generated for every $n \gg 0$.
\item[(c)] \label{itm:thmD_weak_pos} {\textnormal{[\autoref{thm:global_generation_of_S_f_*}]}}.  If $M -K_{X/V} - \Delta$ is nef, then $S^0_{\Delta,ne} f_* (M)$ is weakly-positive for $n \gg 0$.
\item[(d)] \label{itm:thmD_pushforward} {\textnormal{[\autoref{cor:S_f_if_big_power_of_rel_ample}]}}. If $M=Q^l \otimes P$ where $Q$ is $f$-ample, then for all $l \gg0$ and nef line bundle $P$, we have that $S^0_{\Delta,ne} f_* (M)$ is contained in $(f_{V^{ne}} ) _* (\sigma_{ne}(X/V, \Delta) \otimes M_{V^{ne}}) $ as subsheaves of $\left( f_{V^{ne}} \right)_* M_{V^{ne}}$, and furthermore, these two subsheaves are generically equal.
\end{itemize}

% In the situation of \autoref{notation:S_f_*}, if $V$ is projective over a perfect field $k$ and
% $L \otimes M^{p^e - 1}$ is ample, then $S^0_{\phi^n} f_*(M)$ is globally generated for every $n \gg 0$.
%
% In the situation of \autoref{notation:S_f_*}, if $V$ is projective and
% $L \otimes M^{p^e -1}$ is a nef and $f$-ample, then $S^0_{\phi^n} f_*(M)$ is weakly positive for $n \gg 0$.
%
In fact points (a) and (d) are true without the projectivity assumption on $V$.
\end{theoremD}

Further, note that if  $V$ in Theorem D is a curve, $M - K_{X/V} -\Delta$ is $f$-ample and nef  and there is a $s \in V$, such that  $H^0(X_s, M_s)= S^0(X_s, \sigma(X_s, \Delta_s))$, then points \autoref{itm:thmD_rank} and \autoref{itm:thmD_weak_pos} of Theorem D with point \eqref{itm:thmC_maximal_open} of Theorem $C$ imply that $f_* \sO_X(M)$ is a nef vector bundle. This strengthens  \cite[Proposition 3.6]{PatakfalviSemipositivity}, it removes the cohomology vanishing assumption made there. Of course our statements are considerably stronger. For example with the same assumptions except replacing $H^0(X_s, M_s)= S^0(X_s, \sigma(X_s, \Delta_s))$ with $\rk f_*( \sigma (X,\Delta) \otimes M)$ equals the general value of $S^0(X_s, \sigma(X_s, \Delta_s))$, we obtain that $f_* (\sigma(X,\Delta) \otimes M)$ being nef. In fact, the following even stronger statement can be made.

\begin{theoremE}
{\textnormal{[\autoref{cor:semi_positivity_of_S_0_f}]}} With assumptions as in Theorem C such that $V$ is projective over a prefect field, assume that $M -K_{X/V} - \Delta$ is nef and $f$-ample and that $\rk S^0f_* (\sigma(X,\Delta) \otimes M)$ equals the general value of $H^0(X_s, \sigma(X_s, \Delta_s) \otimes M_s)$. Then $S^0f_* (\sigma(X,\Delta) \otimes M)$ is weakly positive. In particular, if $V$ is a smooth curve then it is a nef vector bundle.
\end{theoremE}

Further, we show how $S^0_{\Delta, ne} f_* (M)$ relates to the other similar notion $S^0 f_*( \sigma(X, \Delta) \otimes M)$ introduced in \cite[Definition 2.14]{HaconXuThreeDimensionalMinimalModel}. In particular we obtain that $S^0 f_*( \sigma(X, \Delta) \otimes M)$ does not restrict to $S^0(X_{s}, \sigma(X_{s}, \Delta_{s}) \otimes M_{s})$ in general for general $s \in S$. Intuitively, though $S^0 f_*( \sigma(X, \Delta) \otimes M)$ is a pushforward, it captures the global geometry of $(X,\Delta)$ rather then the geometry of the fibers. In positive characteristic these two can differ considerably, essentially because the function field of $V$ is not perfect. Of course the relative and absolute $S^0 f_*$ are related.  We study these similarities and differences in \autoref{sec.RelationBetweenRelativeAbsoluteS0F*}.

\subsection{Organization} In Section 2, we set up notation that we will follow throughout the paper, explain the interplay between $p^{-e}$-linear maps and $\mathbb{Q}$-Cartier divisors, and discuss the behaviors of such maps and divisors under base changes. In Sections 3, 4, and 5,  we introduce the relative non-$F$-pure ideals, the relative test ideals and the relative test submodules, respectively. Applications to $F$-singularities in families are discussed in these 3 sections. Section 6 is devoted to the behaviors of $S^0$ (introduced in \cite[Definition 4.1]{SchwedeACanonicalLinearSystem}) under base changes. Some semi-positivity results are also proved in this section. Finally, in the Appendix, we collect some statements that we can not find proper references for their generality but are needed in our paper.

\subsection{Acknowledgements}

The authors began working on this project while attending a workshop at the American Institute of Mathematics (AIM) titled ``ACC for minimal log discrepancies and termination of flips'', May 14th -- 18th 2012, organized by Tommaso de Fernex and Christopher Hacon.  We would like to thank the organizers and the American Institute of Mathematics.  Part of the work was done when the third author visited the Department of Mathematics at Pennsylvania State University; he would like to thank them for their hospitality. Additionally, the authors would like to thank Lawrence Ein, Christopher Hacon, J\'anos Koll\'ar, Joseph Lipman, Mircea \mustata{}, and Kevin Tucker for useful discussions.  Especially, we would like to thank Brian Conrad for pointing out compatibility of trace with base change.  We would also like to thank Hiromu Tanaka for comments on a previous draft.  Additionally, the authors of the paper would like to thank Leo Alonso,  Damian R\"ossler, Graham Leuschke and Blup for providing answers and
commentary to the following questions
{\tt http://mathoverflow.net/questions/120982/}
and {\tt http://mathoverflow.net/questions/120625/}.

\section{Notation and setup}

Throughout this paper, all schemes are Noetherian and all maps of schemes are separated. We fix the following notation, which is in effect for the entire paper. In particular, for simplicity we do not state it in every statement even though it is assumed.

\begin{notation}
\label{notation:basic}
Suppose that $f : X \to V$ is a flat, equidimensional and geometrically reduced map\footnote{Meaning $X_T = X \times_V T$ is reduced for all $T \to V$ with $T$ integral.} of finite type from a scheme $X$ to an excellent integral scheme $V$ of equal characteristic $p > 0$ with a dualizing complex.  We write $F^e = F^e_V : V = V^e \to V$ as the absolute $e$-iterated Frobenius on $V$, form the base change $f_{V^e} : X_{V^e} = X \times_V V^e \to %\xrightarrow{h_e}
V^e$ and define $F^e_{X^e/V^e} : X^e \to X_{V^e}$ to be the $e$-iterated relative Frobenius.
Furthermore, we often assume that $V$ is $F$-finite in which case we automatically assume that Frobenius $F_V : V \to V$ satisfies the following identity\footnote{For some discussion of this identity, which always holds for varieties or schemes of essentially finite type over a local ring with a dualizing complex, see \cite[Section 2]{BlickleSchwedeTakagiZhang}.}, $(F_V)^! \omega_V^{\mydot} \qis \omega_V^{\mydot}$.
If we say that $V$ is a variety, it is always of finite type over a perfect field $k$.

Because we will be considering numerous different sheaves on the same topological space $X = X^e = X \times_V V^e = \text{etc.}$, but with respect to different schemes, we will adopt the following somewhat nonstandard notation.  We use these because otherwise writing numerous $\left( F^e_{X^e/V^e} \right)_*$ and $\left( F^e_{X^e/V^e} \right)^{-1}$ operations, which do nothing to the underlying space, is confusing.   We would need notations for projections of the form $X^e \times_{V^e} V^{e+d} \to V^{e+d}, X^e$ in as well as more general relative Frobenii $F^{e}_{(X^{e})_{V^{e+d}}/V^{e+d}} : (X^e)_{V^{e+d}} := X^e \times_{V^e} V^{e+d} \to X_{V^{e+d}}$, and maps $X^{e} \times_{V^{e+d+c}} \to X^e \times_{V^{e + d}}$.  By using simply modules, this becomes more transparent.
\begin{itemize}
\item[(1)]  We will use $R$ to denote $\O_X$.
\item[(2)]  We will use $A$ to denote $f^{-1} \O_V$.
\item[(3)]  We will use $A^{1/p^e}$ to denote $f^{-1} (F^e_V)_* \O_V$.  We note that this is not an abuse of notation since $V$ is integral.
\item[(4)]  We will use $R^{1/p^e}$ to denote $(F^e_X)_* R$.  This is a slight abuse of notation if $X$ is not reduced.
\item[(5)]  We will use $R_{A^{1/p^e}}$ to denote $R \otimes_A A^{1/p^e}$.
\item[(6)]  We will use $\left( R_{A^{1/p^e}} \right)^{1/p^d}$ to denote $R^{1/p^d} \otimes_{A^{1/p^d}} A^{1/p^{e+d}}$.  This may be a slight abuse of notation if $X_{V^e}$ is not reduced.
\item[(7)]  Given $R$-module $M$, we will use $M^{1/p^e}$ to denote $F^e_* M$.  This will generally not cause any confusion because typically $M$ will be locally free or even a line bundle.
\item[(8)]  We use $\omega_R$ to denote $\omega_X$ and $\omega_A$ to denote $f^{-1} \omega_V$.
\item[(9)]  Etc.
\end{itemize}
\end{notation}

Some of the main results of the paper concern restriction to fibers. These statements pertain only to a special set of fibers of $f$, the fibers over perfect points (see definition below). For example, if $V$ is a curve over an algebraically closed field, then often we restrict to fibers over all closed points and the perfect closure of the generic point of $V$, but not the generic point itself.

\begin{definition}[Perfect points]
\label{def.PerfectPoint}
A \emph{perfect point of $V$, $s \in V$} is a morphism from the spectrum of a perfect field $k(s)$ to $V$, $s = \Spec k(s) \to V$.  It can also be viewed as a choice of a point $v \in V$ and a field extension $k(s) = K \supseteq k(v)$ such that $K$ is perfect.  Finally, a \emph{neighborhood of a perfect point $s \in V$} is simply a neighborhood of the image $v$ of $s$.
\end{definition}

The fact that $f : X \to V$ is of finite type implies that the relative Frobenius is a finite map.  In some cases, this will allow us to avoid assuming that $V$ is $F$-finite.

\begin{lemma}
\label{finiteness of relative Frobenius}
With notation as above, since $f : X \to V$ is of finite type, the relative Frobenius map $f : X^e \to X_{V^e}$ is a finite map.
\end{lemma}
\begin{proof}
We work locally on $V$ and $X$.  It is sufficient to show that $R^{1/p^e}$ is a finite $R_{A^{1/p^e}}$-module.  Write $S = A[x_1, \dots, x_n]$ and $R = S/I$.  We first observe that $S_{A^{1/p^e}} \to S^{1/p^e}$ is a finite map.  However, $x_1^{i_1} \cdots x_{n}^{i_n}$ for $0 \leq i_j < p^e$ clearly form a generating set for $S^{1/p^e}$ over $S_{A^{1/p^e}}$.

Now, tensor the map $S_{A^{1/p^e}} \to S^{1/p^e}$ with $\otimes_S (S/I) = \otimes_S R$.  We obtain
\[
R_{A^{1/p^e}} = R \otimes_A A^{1/p^e} = (S/I) \otimes_S S_{A^{1/p^e}} \to (S/I) \otimes_S S^{1/p^e} = (S/I^{[p^e]})^{1/p^e} \to (S/I)^{1/p^e}
\]
where the final map is the canonical surjection of rings.  The map is finite since each part is.
\end{proof}

We also recall the result of Radu and Andr\'e.

\begin{theorem} \cite{RaduUneClasseDAnneaux,AndreHomomorphismsRegulariers}
\label{thm.RaduAndre}
Suppose that $f : X \to V$ is a flat map of Noetherian schemes.  Then $f$ has geometrically regular fibers if and only if the relative Frobenius $R_{A^{1/p^e}} \cong A^{1/p^e} \tensor_A R \to R^{1/p^e}$ is flat.
\end{theorem}

We immediately obtain the following corollary which will also be useful.
\begin{corollary}
\label{cor.SmoothMapImpliesSurjective}
Suppose that $f : X \to V$ is as in \autoref{notation:basic} and additionally has geometrically regular fibers.  Then for any $R_{A^{1/p^e}}$-module $M$, the natural evaluation-at-1 map below surjects
\[
\sHom_{R_{A^{1/p^e}}}(R^{1/p^e}, M) \to M
\]
\end{corollary}
\begin{proof}
$R^{1/p^e}$ is a locally free $R_{A^{1/p^e}}$-module, and the result follows.
\end{proof}

Now we state the object which we will study for the majority of the paper.

\begin{definition}[$\phi$]  From here on, we will fix a line bundle $L$ on $X \cong X^e$ and we will also fix a (possibly zero) $R_{A^{1/p^e}}$-linear map $\phi : L^{1/p^e} \to R_{A^{1/p^e}}$.
\end{definition}

\input{LocalStuff.tex}

\input{GlobalStuff.tex}

\input{Appendix.tex}

\bibliographystyle{skalpha}
\bibliography{MainBib}

\end{document}

%% file: LocalStuff.tex
% \todo{{\bf Karl:} Brendan and Sandor assume that everything is of finite type over a field $k = \bar{k}$.  We should be careful, and maybe include what we are using in the appendix stated explicitly in our generality.\\
% {\bf Wenliang:} I don't see where the assumption on $k$ is used in [\S3,HK] upto Corollary 3.8.}

\begin{remark}[Reflexive sheaves for G1 and S2 Morphisms]
\label{rem.G1S2Morphisms}
Suppose that $f : X \to V$ is G1 + S2.
Note that there exists an open set $\iota : U \hookrightarrow X$ such that $X \setminus U$ has codimension $\geq 2$ along each fiber and that $f|_U$ is a Gorenstein morphism. %Furthermore, note that by \autoref{prop:relative_canonical_reflexive}, $\omega_{X/V}$ is reflexive. Hence by \autoref{prop:pushforward_relative_S_2_or_reflexive}, $\iota_* \omega_{X/V}|_U \cong \omega_{X/V}$.

% so that $\omega_{X/V}|_U$ is locally free and hence reflexive.  Note then that $i_* (\omega_{X/V}|_U)$ is reflexive by \autoref{prop:extension_of_reflexive_reflexive} since $\omega_{X/V}|_U$ is locally free.  Furthermore, $i_* (\omega_{X/V}|_U) = (\omega_{X/V})^{**}$ by \autoref{prop:relative_S_2_morphism_reflexive_has_extension_property}.

Finally, suppose that $M$ is any rank-1 reflexive $R$-module which is locally free on a set $U$ as above.  Then $M^{1/p^e}$ is not only reflexive as an $R^{1/p^e}$-module, we claim it is also reflexive as an $R_{A^{1/p^e}}$-module.  Since $i_* M|_U = M$ already by
\autoref{prop:pushforward_relative_S_2_or_reflexive} because $M$ is reflexive, it is sufficient to replace $X$ by $U$.  Thus $\omega_{X/V}$ and $M$ are both locally free as $R$-modules and $\omega_{X_{V^e}}$ is locally free as an $R_{A^{1/p^e}}$-module.  We work locally so as to trivialize all these modules.  Then
\[
\begin{array}{rcl}
 M^{1/p^e} & \cong & \omega_{X/V}^{1/p^e}\\
&  =  & (F^{e}_{X^e/V^e})_* \omega_{X/V} \\
 & \cong & \sHom_{\O_{X_{V^e}}}( (F^{e}_{X^e/V^e})_* \O_X, \omega_{X_V^e/V^e}) \\
 & \cong & \sHom_{R_{A^{1/p^e}}}(R^{1/p^e}, R_{A^{1/p^e}} )
\end{array}
\]
which is clearly reflexive (the second isomorphism follows from Grothendieck duality for the finite relative Frobenius map).

Conversely, if $M$ is any $R$-module which is locally free on a set $U$ and reflexive as an $R_{A^{1/p^e}}$-module, then it is also reflexive as an $R^{1/p^e}$-module.  To see this, note that $M|_U$ is reflexive as an $R$-module, and because $M$ is reflexive as an $R_{A^{1/p^e}}$-module it satisfies $i_* M|_U = M$.
\end{remark}

\begin{definition}[$\phi$ versus divisors]
\label{def:D_phi_Delta_phi}
Observe the following identifications:
\begin{equation}
\label{eq.HomToSheaf1}
\begin{array}{rc}
 & \sHom_{R_{A^{1/p^e}}}\left(L^{1/p^e}, R_{A^{1/p^e}} \right) \\
\cong & \underbrace{\Big(\sHom_{R_{A^{1/p^e}}} \left(L^{1/p^e} \otimes_{R_{A^{1/p^e}}} (\omega_{R_{A^{1/p^e}}/A^{1/p^e}} ), \omega_{R_{A^{1/p^e}}/A^{1/p^e}}\right)\Big)^{**}}_{\text{This clearly holds on $U$, note both sheaves are reflexive and use \autoref{cor:relative_S_2_morphism_reflexive_has_extension_property}.} }\\
= & \underbrace{\Big(\sHom_{\O_{X_{V^e}}}\left((F^e_{X^e/V^e})_* \Big( L \otimes ((F^e_{X^e/V^e})^* \omega_{X_{V^e}/V^e})\Big), \omega_{X_{V^e}/V^e} \right)\Big)^{**}}_{\text{This just rewrites the previous line using different notation.}}\\
\cong & \underbrace{\Big((F^e_{X^e/V^e} )_* \sHom_{\O_{X^e}} \left(L \otimes ((F^e_{X^e/V^e})^* \omega_{X_{V^e}/V^e}), \omega_{X^e/V^e} \right)\Big)^{**}}_{\textrm{Grothendieck duality for a finite map}}\\
\cong & \underbrace{(F^e_{X^e/V^e} )_* \left(L^{-1}\otimes \omega_{X^e/V^e} \otimes (F^e_{X^e/V^e})^* \omega_{X_{V^e}/V^e}^{-1} \right)^{**}}_{\textrm{By \autoref{rem.G1S2Morphisms} and \autoref{cor:relative_S_2_morphism_reflexive_has_extension_property} we may take reflexive hull as $(F^e_{X^e/V^e})_* \O_{X^e}$-modules.}}\\
\end{array}
\end{equation}
% By the discussion of \autoref{rem.G1S2Morphisms}, it does not matter whether we reflexify the above entries as $R^{1/p^e}$-modules or as $R_{A^{1/p^e}}$-modules because all of the sheaves are free of rank-one as $R^{1/p^e}$-modules.
Now, observe that $\omega_{X/V}$ is compatible with base change up to reflexification.  In particular, if our base change is the Frobenius $F^e_{V} : V^e \to V$, then writing $\pi_{V^e} : X_{V^e} = X \times_V V^e \to X$ as the projection, we have $\omega_{X_{V^e}/V^e} \cong (\pi_{V^e}^* \omega_{X/V})^{**}$ since both sheaves are reflexive and they certainly agree outside the non-relatively Cohen-Macaulay locus (which is of relative codimension at least $2$) by \autoref{cor:relative_S_2_morphism_reflexive_has_extension_property}.  In particular
\[
\big((F^e_{X^e/V^e})^*\omega_{X_{V^e}/V^e}\big)^{**} \cong \big((F^e_{X^e/V^e})^*\pi_{V^e}^* \omega_{X/V})\big)^{**} = \big((F^e_{X})^* \omega_{X/V}\big))^{**} = \big(\omega_{X/V}^{p^e}\big)^{**}.
\]
Plugging this into \autoref{eq.HomToSheaf1} we obtain
\begin{equation}
\label{eq.HomToSheaf2}
\begin{array}{rc}
 & \sHom_{R_{A^{1/p^e}}}\left(L^{1/p^e}, R_{A^{1/p^e}} \right) \\
\cong & (F^e_{X^e/V^e} )_* \left(L^{-1}\otimes \omega_{X^e/V^e}^{1-p^e} \right)^{**}\\
\end{array}
\end{equation}

If additionally $X$ is absolutely (instead of relatively) G1 and S2 (for example, if $V$ is regular), then any choice of nondegenerate\footnote{Nonzero on any irreducible component.} $\phi$ induces a non-zero, effective Weil divisorial sheaf $D_{\phi}$ such that $\O_X(D_{\phi}) \cong \Big(L^{-1}\otimes \omega_{X^e/V^e}^{1-p^e}\Big)^{**}$ by \cite{HartshorneGeneralizedDivisorsOnGorensteinSchemes}.  We would like to generalize this to the case that $X \to V$ is \emph{relatively} G1 and S2.

\begin{definition}
\label{def.RelativelyDivisorial}
We say that $\phi$ is \emph{relatively divisorial} if $\phi$ locally generates $\sHom_{R_{A^{1/p^e}}}(L^{1/p^e}, R_{A^{1/p^e}})$ as an $R^{1/p^e}$-module at
\begin{enumerate}
\item \label{itm:D_phi_Delta_phi:generic_point_of_fiber} the generic points of each fiber and
\item \label{itm:D_phi_Delta_phi:singular_point_of_fiber} at the generic point of every codimension-1 singular point of every geometric fiber.
\end{enumerate}
\end{definition}
In this case, by removing a set of relative codimension 2 so that $f$ is relatively Gorenstein, we see that $\phi \cdot R^{1/p^e} \subseteq \sHom_{R_{A^{1/p^e}}    }(L^{1/p^e}, R_{A^{1/p^e}})$ is a rank-1 free submodule of an invertible $R^{1/p^e}$-module. To be able to associate a divisor to this submodule in a sensible way, we should show that it is the trivial (full) submodule at every singular codimension one point. Indeed, let $\xi$ be a singular codimension 1 point. Then one of the following cases hold.
\begin{itemize}
 \item $f(\xi)$ is a codimension 1 point. In this case, $\xi$ is a general point of the fibers over $f(\xi)$, hence assumption \autoref{itm:D_phi_Delta_phi:generic_point_of_fiber} guarantees that $\phi$ generates a full submodule at $\xi$.
\item $f(\xi)$ is the general point of $V$. In this case $\xi$ is a codimension 1 point of the fiber over $f(\xi)$, and it is not in the smooth locus of $f$. Therefore, assumption \autoref{itm:D_phi_Delta_phi:singular_point_of_fiber} shows that again $\phi$ generates a full submodule at $\xi$.
\end{itemize}
Therefore, the submodule $\phi \cdot R^{1/p^e}$ determines a  Cartier divisor and also an honest Weil divisor on the original $X$.  We denote this divisor $D_{\phi}$ as well.
\end{definition}

\begin{definition}[$\phi$ as a divisor]
  \label{def.PhiAsDivisor} If $\phi$ is relatively divisorial, we set $\Delta_{\phi}$ to be the $\bQ$-divisor $\frac{1}{p^e - 1} D_{\phi}$.  This makes sense because $D_{\phi}$ is trivial along the codimension-1 components of the singular locus of $X$ and so we avoid the pathological issues which occur for $\bQ$-divisors on non-normal spaces.
  \end{definition}

  We now explain how to recover $\phi$ from a $\bQ$-divisor.

\begin{remark}[Obtaining $\phi$ from divisors]
\label{rem:relative_canonical}
We work under the conventions of \autoref{def.PhiAsDivisor}.
Untangling \autoref{def:D_phi_Delta_phi} yields a method to obtain $\phi$ from a divisor $\Delta \geq 0$ (which then coincides with $\Delta_{\phi}$) under the following assumptions:
\begin{enumerate}
\item $\Delta = {1 \over m} D$ for some Weil divisor $D$ where $p \not|\,\; m$.
\item $D$ is a Weil divisor on $X$ which is Cartier in relative codimension 1.
\item $\Supp D$ does not contain any general point or any singular codimension one point of any geometric fiber.
\item $(p^e - 1)/m \in \bZ$ and $\big(\omega_{X/V}^{1-p^e} \otimes \O_X ( (1-p^e) \Delta) \big)^{**} = L$ is a line bundle.
\end{enumerate}
In such a case, the integral divisor $(p^e - 1)\Delta$ (which is well defined since $\Supp \Delta$ does not contain the codimension-1 components of the non-regular locus of $X$, \cf \cite[Pages 171--173]{KollarFlipsAndAbundance} or \cite[Section 2.2]{MillerSchwedeSLCvFP}) induces an inclusion
\begin{equation*}
(F^e_{X^e/V^e} )_* L \hookrightarrow  (F^e_{X^e/V^e} )_* \big((\omega_{X/V}^{1-p^e})^{**}\big) .
\end{equation*}
This composed with the natural Grothendieck trace map
\begin{equation*}
(F^e_{X^e/V^e} )_* \big((\omega_{X/V}^{1-p^e})^{**}\big) \to \sO_{X_{V^e}}
\end{equation*}
yields a map $\phi$.  It is easy to see that conditions (a)--(d) above guarantee that $\phi$ is relatively divisorial.   Furthermore, we also have that $\Delta = \Delta_{\phi}$.
\end{remark}

For future reference we make the following definition.

\begin{definition}
\label{def:pair}
In the situation of \autoref{notation:basic}, \emph{$(X,\Delta)$ is a pair}, if $f$ is G1 + S2 and $\Delta$ satisfies the assumptions of \autoref{rem:relative_canonical}.
\end{definition}

\mysubsection{Composing maps}
\label{subsec.ComposingMaps}
Given $\phi$ as in \autoref{def:D_phi_Delta_phi}, we can compose $\phi : L^{1/p^e} \to R_{A^{1/p^e}}$ with itself (after twisting) similar to \cite[Section 4]{BlickleSchwedeSurveyPMinusE} or \cite{SchwedeFAdjunction} and thus obtain new maps
\[
\phi^2 \in \Hom_{R_{A^{1/p^{2e}}}} \left( \left(L^{(p^e + 1)}\right)^{1/p^{2e}}, R_{A^{1/p^{2e}}}\right)
\]
and more generally
\[
\phi^n \in \Hom_{R_{A^{1/p^{ne}}}}\left( \left(L^{p^{ne} - 1 \over p^e - 1}\right)^{1/p^{ne}}, R_{A^{1/p^{ne}}}\right).
\]
We explain this construction.

Begin by tensoring $\phi$ by $L$ over $R$, and then taking $1/p^e$th roots we obtain:
\begin{equation}
\label{eq.Twist1}
\left(L^{1 + p^e} \right)^{1/p^{2e}} = \left(L^{1/p^{e}} \otimes_{R} L \right)^{1/p^e} \to \left( (R \otimes_A A^{1/p^e}) \otimes_R L \right)^{1/p^e} = L^{1/p^e} \otimes_{A^{1/p^e}} A^{1/p^{2e}}
\end{equation}
On the other hand, we can also tensor $\phi$ by $A^{1/p^{2e}}$ over $A^{1/p^e}$ to obtain:
\begin{equation}
\label{eq.Twist2}
 L^{1/p^e} \otimes_{A^{1/p^e}} A^{1/p^{2e}} \to R_{A^{1/p^{2e}}}
\end{equation}
By composing \eqref{eq.Twist1} and \eqref{eq.Twist2} we obtain the desired map $\phi^2$.  We now define $\phi^n$ inductively as follows.  Given $\phi^{n-1} : \left(L^{p^{(n-1)e} - 1 \over p^e - 1}\right)^{1/p^{(n-1)e}} \to R_{A^{1/p^{(n-1)e}}}$ tensor with $L$ over $R$ and then take $p^e$th roots which yields
\[
{\def\arraystretch{2.6}
\begin{array}{rl}
& \left(L^{{p^{ne} - 1 \over p^e - 1}}\right)^{1/p^{ne}} \\
= & \left(\left(L^{{p^{(n-1)e} - 1 \over p^e - 1} + p^{(n-1)e}}\right)^{1/p^{(n-1)e}}\right)^{1/p^e} \\
= & \left(L \otimes_R \left(L^{p^{(n-1)e} - 1 \over p^e - 1}\right)^{1/p^{(n-1)e}} \right)^{1/p^e} \\
\to & \left(L \otimes_R  R_{A^{1/p^{(n-1)e}}} \right)^{1/p^e}\\
%= & \left(L \otimes_A A^{1/p^{(n-1)e}}\right)^{1/p^e}\\
= & L^{1/p^e} \otimes_{A^{1/p^e}} A^{1/p^{ne}}.
\end{array}
}
\]
We then apply $\phi$ to the first term in the final tensor product to obtain:
\[
\left(L^{{p^{ne} - 1 \over p^e - 1}}\right)^{1/p^{ne}} \to L^{1/p^e} \otimes_{A^{1/p^e}} A^{1/p^{ne}} \xrightarrow{\phi \tensor \ldots} \left(R \otimes_A A^{1/p^e} \right) \otimes_{A^{1/p^e}} A^{1/p^{ne}} \cong R_{A^{1/p^{ne}}}
\]
which we take as the official definition of $\phi^n$.  On the other hand, for every $0<m<n$ one can look at the following composition of $\phi^m$ and $\phi^{n-m}$
\begin{equation}
\label{eq.phimAndphin}
\begin{matrix}
\gamma : \left(L^{\frac{p^{ne}- 1}{p^e -1}} \right)^{\frac{1}{p^{ne}}}
\cong \left(L^{\frac{ p^{me} - 1}{p^e-1}} \otimes_R L^{\frac{ p^{ne}- p^{me}}{p^e -1}} \right)^{\frac{1}{p^{ne}}}
\cong \left(\left( L^{\frac{ p^{me} - 1 }{p^e -1}} \right)^{\frac{1}{p^{me}}} \otimes_{R} L^{\frac{  p^{(n-m)e}- 1}{p^e -1}} \right)^{\frac{1}{p^{(n-m)e}}}
\vspace{20pt}\\
\vspace{20pt}
\xrightarrow{\left( \phi^m \otimes \id_{L^{\cdots}} \right)}^{\frac{1}{p^{(n-m)e}}}
\left( \left( R \otimes_A A^{\frac{1}{p^{me}}} \right) \otimes_{R} L^{\frac{  p^{(n-m)e}- 1}{p^e -1}} \right)^{\frac{1}{p^{(n-m)e}}}
\cong \left(   L^{\frac{  p^{(n-m)e}- 1}{p^e -1}} \right)^{\frac{1}{p^{(n-m)e}}}  \otimes_{A^{\frac{1}{p^{(n-m)e}}}} A^{\frac{1}{p^{ne}}}\\
\xrightarrow{\phi^{n-m} \otimes_{A^{\cdots}} A^{\cdots}}  \left( R \otimes_A A^{\frac{1}{p^{(n-m)e}}} \right) \otimes_{A^{\frac{1}{p^{(n-m)e}}}} A^{\frac{1}{p^{ne}}} \cong R \otimes_A A^{\frac{1}{p^{ne}}}
\end{matrix}
\end{equation}
In particular, taking $m = 1$ gives an a priori different map which we could also define as $\phi^n$.
We now explain why this map is actually equal to the \emph{official} $\phi^n$.
\begin{lemma}
With notation as above, $\gamma = \phi^n$.
\end{lemma}
\begin{proof}
The statement is local, and so we may suppose that $L = R$.  With the (obscuring) powers of $L$ removed, $\phi^n$ is described as the following composition:
\[
\begin{array}{rll}
& R^{1/p^{ne}} & \\
\xrightarrow{\phi^{1/p^{(n-1)e}}} & R^{1/p^{(n-1)e}} \otimes_{A^{1/p^{(n-1)e}}} A^{1/p^{ne}}\\
\xrightarrow{\phi^{1/p^{(n-2)e}} \otimes \ldots} & \Big(R^{1/p^{(n-2)e}} \otimes_{A^{1/p^{(n-2)e}}} A^{1/p^{(n-1)e}}\Big) \otimes_{A^{1/p^{(n-1)e}}} A^{1/p^{ne}}\\
\xrightarrow{\phi^{1/p^{(n-3)e}} \otimes \ldots} & \cdots\\
\cdots & \cdots\\
\xrightarrow{\phi^{1/p^e} \otimes \ldots} & R^{1/p^{e}} \otimes_{A^{1/p^e}} \ldots \otimes_{A^{1/p^{(n-1)e}}} A^{1/p^{ne}}\\
\xrightarrow{\phi \otimes \ldots} & R \otimes_{A} \ldots \otimes_{A^{1/p^{(n-1)e}}} A^{1/p^{ne}}.
\end{array}
\]
The first $m$ entries in the composition make up $(\phi^m \otimes \ldots)^{\frac{1}{p^{(n-m)e}}}$ in \autoref{eq.phimAndphin} and the last $n-m$ entries clearly yield $\phi^{n-m} \otimes_{A^{1/p^{(n-m)e}}} A^{1/p^n}$ as desired.  The result is then obvious.
\end{proof}

\begin{lemma}
\label{lem:composition}
With notation as in \autoref{def:D_phi_Delta_phi} and additionally that $\phi$ is relatively divisorial, then $\phi^n$ is relatively divisorial and $\Delta_{\phi} = \Delta_{\phi^n}$  for every integer $n \geq 1$.
\end{lemma}
\begin{proof}
For showing any of the two statements we may remove the non relatively Gorenstein locus. That is, by possibly further restricting $R$ we may assume that $R$ is relatively Gorenstein over $V$ and that $L$ is trivial  (since our map $f:X\to V$ is geometrically G1+S2). Hence by \autoref{rem.G1S2Morphisms},
\[
\Hom_{R\otimes_{A}A^{1/p^{ie}}}(R^{1/p^{ie}},R\otimes_{A}A^{1/p^{ie}})\cong R^{1/p^{ie}}\ {\rm for\ all\ }i.
\]
From \cite[Appendix F]{KunzKahlerDifferentials}, \cite[Lemma 3.9]{SchwedeFAdjunction}, or composition of Grothendieck trace, we have
% \begin{multline}
% \label{eq:composition:tensor}
% \Hom_{\left(R^{1/p^{(n-m)e}} \right)_{A^{1/p^{ne}}}} \left( R^{1/p^{ne}}, \left(R^{1/p^{(n-m)e}} \right)_{A^{1/p^{e}}}   \right) %
% %
% \\ \otimes_{\left(R^{1/p^{(n-m)e}} \right)_{A^{1/p^{ne}}}}
% %
%  \Hom_{R_{A^{1/p^{ne}}}} \left( \left(R^{1/p^{(n-m)e}} \right)_{A^{1/p^{ne}}}, R_{A^{1/p^{ne}}}   \right)
% %
% \\ \cong
% %
% \Hom_{R_{A^{1/p^{ne}}}} \left( R^{1/p^{ne}},  R_{A^{1/p^{ne}}}   \right)
% \end{multline}
\begin{multline}
\label{eq:composition:tensor}
\left(\begin{array}{rl}& \Hom_{R^{1/p^{e}} \otimes_{A^{1/p^e}}A^{1/p^{ne}}} \left( R^{1/p^{ne}}, R^{1/p^{e}} \otimes_{A^{1/p^e}}A^{1/p^{ne}}  \right) \\
\otimes_{R^{1/p^{e}} \otimes_{A^{1/p^e}}A^{1/p^{ne}}} & \Hom_{R_{A^{1/p^{ne}}}} \left( R^{1/p^{e}} \otimes_{A^{1/p^e}}A^{1/p^{ne}}, R_{A^{1/p^{ne}}}   \right)\end{array}\right)
%\Hom_{R^{1/p^{e}} \otimes_{A^{1/p^e}}A^{1/p^{ne}}} \left( R^{1/p^{ne}}, R^{1/p^{e}} \otimes_{A^{1/p^e}}A^{1/p^{ne}}  \right) %
%
%\\ \otimes_{R^{1/p^{e}} \otimes_{A^{1/p^e}}A^{1/p^{ne}}}
%
% \Hom_{R_{A^{1/p^{ne}}}} \left( R^{1/p^{e}} \otimes_{A^{1/p^e}}A^{1/p^{ne}}, R_{A^{1/p^{ne}}}   \right)
%
\\ \cong
\Hom_{R_{A^{1/p^{ne}}}} \left( R^{1/p^{ne}},  R_{A^{1/p^{ne}}}   \right)
\end{multline}
via the homomorphism induced by composition.
Let $\theta_e$ be the $R^{1/p^e}$-module generator of $$\Hom_{R\otimes_{A}A^{1/p^e}}(R^{1/p^e},R\otimes_{A}A^{1/p^e})$$ for each $e$ and let $\theta^n_e$ be defined as $\varphi^n$, but $\varphi$ replaced by $\theta$.  Then by using \autoref{eq:composition:tensor} iteratively, $\theta^n_e=\theta_{ne}$ up to multiplication by a unit.

Further let $r \in R$ such that $\varphi(-)=\theta_e \left(r^{1/p^e} \cdot - \right)$. Then it is easy to verify that
\begin{equation}
\label{eq:composition:powers_of_r}
 \varphi^n =\theta_{ie} \left( \left(r^\frac{p^{ne} -1}{p^e -1}\right)^{1/p^{ne}} \cdot  \_ \right)
\end{equation}
Then we see that if $\phi$ was generating at a point $P \in X$, or equivalently $r$ is a unit at $P$, then so is $\phi^n$. This shows that $\phi^n$ is relatively divisorial. Furthermore \autoref{eq:composition:powers_of_r} shows   that $\Delta_{\phi^i} = \Delta_{\phi}$.
\end{proof}

\mysubsection{Base change of $\phi$}
\label{subsec.BaseChangeOfPhi}

Suppose that we are given $g : T \to V$ any morphism of schemes such that $T$ is also excellent, integral and has a dualizing complex.  For example, we could set $T$ to be a closed point of $V$ and let $g$ be the inclusion.  Alternately, we could let $g$ be a regular alteration over some closed subscheme of $V$.  We list the following maps:
 \begin{equation}
 \label{eq.DefOfBasicMaps}
 \begin{array}{rccr}
 p_1 :&  X \times_V T & \to X & \text{ (the projection)}.  \\
 (p_1)^{1/p^i} :& X^i \times_{V^i} T^i & \to X^i & \text{ (the projection for any $i$)}.\\
 q_{i} : & X \times_{V} T^i & \to X \times_V V^i & \text{ (base change)}\\
 p_1=q_0 : & X \times_V T & \to X.\\  %All of these maps are pictured in the diagram below:
 \end{array}
 \end{equation}
 These are pictured below.
\[
\xymatrix{
X^i \times_{V^i} T^i  \ar[d]_{F^i_{(X_T)^i/T^i}} \ar[r]^-{p_1^{1/p^i}} & X^i \ar[d]^{F^i_{X^i/V^i}} \\
X \times_V T^i \ar[r]_{q_i} & X \times_V V^i \\
}
\]

Notice that given $\phi : L^{1/p^e} \to R_{A^{1/p^e}} = \O_{X_{V^e}}$, we can form  $(q_e)^* \phi$ which we denote by
\[
\phi_T : L_T^{1/p^e} \cong q_e^* L^{1/p^e} \to q_e^* \O_{X_{V^e}} = \O_{X_{T^e}}.
\]
We explain the isomorphism $L_T^{1/p^e} \cong  q_e^* L^{1/p^e}$ briefly.
Working locally, let $V = \Spec A$, $T = \Spec B$ and $X = \Spec R$.  Then the map $\phi_T$ is identified with the map:
\[
L^{1/p^e} \otimes_{R_{A^{1/p^e}}} R_{B^{1/p^e}} \cong L^{1/p^e} \otimes_{A^{1/p^e}} B^{1/p^e} \to ((R \otimes_A A^{1/p^e}) \otimes_{A^{1/p^e}} B^{1/p^e}) \cong R \otimes_A B^{1/p^e}.
\]
The isomorphism in the definition of $\phi_T$ is now immediate.

The next lemma shows that base change of $\phi$ commutes with the self-composition defined in \autoref{subsec.ComposingMaps}.
\begin{lemma}
\label{lem.BaseChangeForPhin}
Suppose that $g : T \to V$ is as above.  Then $(\phi^n)_T = (\phi_T)^n$.
\end{lemma}
\begin{proof}
It is sufficient to prove this in the affine case assuming our isomorphism is sufficiently canonical (which will be clear).
We notice that $(\phi_T)^2$ is the composition:
\[
\begin{array}{rcl}
\left((L^{p^e + 1})^{1/p^e} \otimes_{A^{1/p^e}} B^{1/p^e}\right)^{1/p^e} & \xrightarrow{(\phi_T \otimes_{R_{B^{1/p^e}}} L_T)^{1/p^e}} & L^{1/p^e} \otimes_{A^{1/p^e}} B^{1/p^e} \otimes_{B^{1/p^e}} B^{1/p^{2e}}\\
& \xrightarrow{\phi_T \otimes_{B^{1/p^e}} B^{1/p^{2e}}} & \left((R \otimes_A A^{1/p^e}) \otimes_{A^{1/p^e}} B^{1/p^e}\right) \otimes_{B^{1/p^e}} B^{1/p^{2e}}
\end{array}
\]
Now unraveling the definitions we obtain
\[
\begin{array}{rcl}
(\phi_T)^2 & %
= & \Big(\big(\phi_T \otimes_{R_{B^{1/p^e}}} L_T\big)^{1/p^e}\Big) \circ (\phi_T \otimes_{B^{1/p^e}}B^{1/p^{2e}})\\
& = & \Big(\big( (\phi \otimes_{A^{1/p^e}} B^{1/p^e} ) \otimes_{R_{B^{1/p^e}}} (L \otimes_{A^{1/p^e}} B^{1/p^e} )\big)^{1/p^e}\Big) \circ (\phi \otimes_{A^{1/p^e}} B^{1/p^e} \otimes_{B^{1/p^e}}B^{1/p^{2e}})\\
& = & \Big(\big(\phi \otimes_{R_{A^{1/p^e}}} L \otimes_{A^{1/p^e}} B^{1/p^e} \big)^{1/p^e}\Big) \circ (\phi \otimes_{A^{1/p^e}} B^{1/p^{2e}})\\
& = & \Big(\big(\phi \otimes_{R_{A^{1/p^e}}} L\big)^{1/p^e} \otimes_{A^{1/p^{2e}}} B^{1/p^{2e}} \Big) \circ \big((\phi \otimes_{A^{1/p^e}} A^{1/p^{2e}} ) \otimes_{A^{1/p^{2e}}} B^{1/p^{2e}}\big)\\
& = & \Big(\big(\phi \otimes_{{R_{A^{1/p^e}}}} L\big)^{1/p^e} \circ (\phi \otimes_{A^{1/p^e}} A^{1/p^{2e}} ) \Big)\otimes_{A^{1/p^{2e}}} B^{1/p^{2e}}\\
& = & (\phi^2)_T
\end{array}
\]
as desired.  The general $n$th self composition is similar.
\end{proof}

%Next we

Our next goal is to describe how divisors, corresponding to maps $\varphi$, fare under base change.  We thank Brian Conrad for pointing us in the right direction -- \cite[Theorem 3.61]{ConradGDualityAndBaseChange}.

\begin{lemma}
\label{lem.BaseChangeOfTraceForCMMaps}
Suppose that $f : X \to V$ is a finite type Cohen-Macaulay morphism over $V$ an excellent scheme of characteristic $p > 0$ and that $g : T \to V$ is as above.
Then the Grothendieck-trace map
\[
(F_{X^e/V^e})_* \omega_{X^e/V^e} \cong \sHom_{\O_{X_{V^e}}}((F_{X^e/V^e})_* \O_{X^e}, \omega_{X_{V^e}/V^e}) \to \omega_{X_{V^e}/V^e}
\]
is compatible with base change.
\end{lemma}
\begin{proof}
We clearly have a commutative diagram:
\[
\xymatrix@R=20pt{
q_e^* \sHom_{\O_{X_{V^e}}}((F_{X^e/V^e})_* \O_{X^e}, \omega_{X_{V^e}/V^e}) \ar[d]_{\alpha} \ar[r] & q_e^* \omega_{X_{V^e}/V^e} \ar[d]^{\beta} \\
\sHom_{\O_{X_{T^e}}}((F_{(X_T)^e/T^e})_* \O_{(X_T)^e}, \omega_{X_{T^e}/T^e}) \ar[r]  & \omega_{\omega_{X_{T^e}/T^e}}
}
\]
and by \cite[Theorem 3.6.1]{ConradGDualityAndBaseChange} the map $\beta$ is an isomorphism.  It is sufficient to verify that $\alpha$ is an isomorphism as well.  We work locally on some affine chart on $X$ and hence assume that $X_{V} \subseteq \bA^{N}_{V} =: P_V$ embeds as a closed subscheme since $f$ is of finite type.  The map $\alpha$ can then be identified with
\[
Q_e^* \sExt_{\O_{P_{V^e}}}^{j}((F_{X^e/V^e})_* \O_{X^e}, \omega_{P_{V^e}/V^e}) \to \sExt_{\O_{P_{T^e}}}^j((F_{(X_T)^e/T^e})_* \O_{(X_T)^e}, \omega_{P_{T^e}/T^e})
\]
where $Q_e : P_{T^e} \to P_{V^e}$ is the induced map and $j = N - \dim(X/V)$ (we leave off the pushforward for the inclusion $i : X_V \to P_V$ above).  This in turn can be identified with
\[
Q_e^* \sExt_{\O_{(P_V)^e}}^{j}( \O_{X^e}, \omega_{(P_V)^e/V^e}) \to \sExt_{\O_{(P_T)^e}}^j( \O_{(X_T)^e}, \omega_{(P_T)^e/T^e})
\]
This last map is exactly the bottom row of \cite[Diagram (3.6.1) in Theorem 3.6.1]{ConradGDualityAndBaseChange}, which is an isomorphism, and hence the proof is complete.
\end{proof}

\autoref{lem.BaseChangeOfTraceForCMMaps} above allows us to show if the divisor $D_{\phi}$ or $\Delta_{\phi}$ is trivial, then it stays trivial after base change.

\begin{lemma}[Divisors which are zero stay zero]
\label{lem:zero_divisor_stay_zero_divisor}
Suppose that $f : X \to V$ is a G1 and S2 morphism and that $\phi : L^{1/p^e} \to R_{A^{1/p^e}}$ is as above.  Additionally suppose that $g : T \to V$ is any base change.
Finally suppose that the natural map
\[
\phi \cdot R^{1/p^e} \to \sHom_{R_{A^{1/p^e}}}\left(L^{1/p^e}, R_{A^{1/p^e}} \right) \cong (F^e_{X^e/V^e} )_* \left(L^{-1}\otimes \omega_{X^e/V^e}^{1-p^e} \right)^{**}
\]
is an isomorphism of $R^{1/p^e}$-modules (here we define $\phi \cdot R^{1/p^e}$ to be the $R^{1/p^e}$-submodule of the set $\sHom_{R_{A^{1/p^e}}}\left(L^{1/p^e}, R_{A^{1/p^e}} \right)$ generated by $\phi$).  Then
\[
\phi_T \cdot (R_T)^{1/p^e} \to \sHom_{R_{T^{1/p^e}}}\left(L_T^{1/p^e}, R_{T^{1/p^e}} \right) \cong (F^e_{(X_T)^e/T^e} )_* \left(L_T^{-1}\otimes \omega_{(X_T)^e/T^e}^{1-p^e} \right)^{**}
\]
is also an isomorphism.

In particular, if $\phi$ is relatively divisorial (see \autoref{def.RelativelyDivisorial}), then the following holds:  If $D_{\phi}$ is zero, then so is $D_{\phi_T}$.
\end{lemma}
\begin{proof}
The statement about divisors is trivial since it is easy to see that a divisor being zero corresponds to the map above being an isomorphism.  Thus we merely need to prove the assertion.  However, since $f$ and all the sheaves involved are relatively S2, it suffices to prove the statement off a set of relative codimension 2 by \autoref{cor:relative_S_2_morphism_reflexive_has_extension_property}.  Therefore, by removing a set of codimension 2, we can assume that $f$ is a Gorenstein morphism.  Thus, working locally if needed, $\phi$ can be identified (up to multiplication by a unit in $R^{1/p^e}$) with the Grothendieck trace $(F^e_{X^e/V^e})_* \omega_{X^e/V^e} \to \omega_{X_{V^e}/V^e}$.  Hence by \autoref{lem.BaseChangeOfTraceForCMMaps}, so can $\phi_T$.  The proof is complete.
\end{proof}

We now move on to a discussion of base change with respect to divisors.  First we observe that by \autoref{lem:zero_divisor_stay_zero_divisor}
 if $\phi$ is relatively divisorial, then so is $\phi_T$ for every base change $g : T \to V$.

\begin{definition}
Suppose that for some relatively divisorial $\phi$ we have $\Delta_{\phi}$ is as above and that $g : T \to V$ is a base change.  Then we write $\Delta_{\phi} \talloblong_{X_T}$ (respectively $D_{\phi} \talloblong_{X_T}$) to denote the divisor $\Delta_{\phi_T}$ (respectively $D_{\phi_T}$).
\end{definition}

\begin{lemma}[Pulling back $\Delta_{\phi}$]
\label{lem.PullingBackDelta}
Suppose that $f : X \to V$ is G1 and S2 and that $\phi$ is relatively divisorial.  Then for any $g : T \to V$  we have $\Delta_{\phi} \talloblong_{X_T} = (p_1^{1/p^e})^* \Delta_{\phi}$ (recall $p_1^{1/p^e} : X^e \times_{V^{i}} T^i \to X_i$ from \autoref{eq.DefOfBasicMaps}).
\end{lemma}
 Here, even though $\Delta_{\phi}$ is not necessarily $\bQ$-Cartier, $(p_1^{1/p^e})^* \Delta_{\phi}$ can be defined after removing a set of relative codimension 2 outside of which it is $\bQ$-Cartier (since $D_{\phi}$ is Cartier on such a set).
\begin{proof}
The statement is local on $X$ and can be checked after removing a relative codimension 2 set and so we may assume that $D_{\phi}$ is a Cartier divisor.  Thus we assume that $L$ is trivial on the affine scheme $X = \Spec R$ and that $f$ is a Gorenstein morphism.  However then shrinking the neighborhood further if necessary, the map $\phi : R^{1/p^e} \to R_{A^{1/p^e}}$ can be identified with $s^{1/p^e} \cdot \Tr$ where $\Tr : (F^e_{X^e/V^e})* \omega_{X^e/V^e} \to \omega_{X_{V^e}/V^e}$ is the Grothendieck trace.  The divisor $D_{\phi}$ is then easily seen to be the divisor $\Div_X(s)$. On the other hand, it is clear by \autoref{lem.BaseChangeOfTraceForCMMaps} that then $\phi_T$ is trace on $X_T$ multiplied by $s$ as well.  In particular, it equals $(p_1^{1/p^e})^* D_{\phi}$ as desired.
\end{proof}

\begin{corollary}
\label{cor.RestrictDeltaAndPhiToFiberIsOk}
Suppose that $g : z \to V$ is inclusion of a point.  Suppose that $\phi$ is relatively divisorial and corresponds to $\Delta_{\phi}$.  Then $\Delta_{\phi}|_{X_z} = \Delta_{\phi_z}$.
\end{corollary}
\begin{proof}
The divisor $\Delta_{\phi_z}$ is determined in codimension 1 where $\Delta_{\phi}|_{X_z}$ and $\Delta_{\phi_z}$ agree.  %The result follows.
\end{proof}

\mysubsection{Passing to $V^{\infty}$ and other perfect points}
\label{sec.BaseChangeToPerfectPoints}

Suppose first that $V$ is the spectrum of a perfect field $A$.  Then the map $L^{1/p^e} \to R_{A^{1/p^e}}$ is also a map $L^{1/p^e} \to R$ since $R_{A^{1/p^e}} \cong R$.  In such a case, we often also typically write the map as $\psi$ to help distinguish how we are thinking about it.   We note that maps such as $\psi : L^{1/p^e} \to R$ can be composed with themselves as in \cite[Section 4.1]{BlickleSchwedeSurveyPMinusE} and furthermore this composition $\psi^n$ coincides with $\phi^n$.
For the rest of the subsection, we discuss base change to perfect points.  We first consider the perfection of the generic point of $V$.

Set $V^{\infty}$ to denote the not-necessarily-Noetherian scheme $\Spec \O_V^{1/p^\infty}$ and we set $\eta_{\infty}$ to be the generic point of $V^{\infinity}$ with perfect fraction field $k(V^{\infty})$ (note that this is a perfect point).  We obtain a map $f_{k(V^\infty)} : X_{k(V^\infty)} = X \times_V \Spec k(V^{\infty}) \to \Spec k(V^{\infty})$ now a scheme of finite type over a perfect field.

%\todo{{\bf Zsolt:} Did you assume earlier that $X \to V$ was of finite type? You seem to be using it above. On the other hand, wouldn't essentially of finite type be a better assumption? \\
%{\bf Karl:}  You are right, we didn't assume it.  What do people think, finite type or essentially finite type?  The latter case probably usually reduces to the former though, so I guess finite type is probably as good? }

%Therefore, the relative Frobenius $F^e_{X_{V^\infty}/k(V^{\infty})}$ can be identified with the absolute Frobenius map on $X_{k(V^\infty)}$.
As in \autoref{subsec.BaseChangeOfPhi}, the map $\phi$ then induces a map
\[
\phi_{\infty} : L^{1/p^e} \otimes_{k(V^e)} k(V^{\infty}) \to R_{k(V^{\infty})} := R \otimes_A k \left( A^{1/p^{\infty}} \right).
\]
which can be identified with
\[
\psi_{\infty} : \left(L_{{k(V^{\infty})}} \right)^{1/p^e} \to R_{k(V^{\infty})}
\]
since $\left((F_{X^e/V^e}^e )_* \O_{X^e} \right) \otimes_{V^e} k(V^{\infty}) \cong F^e_* \O_{X_{k(V^{\infty})}}$.
%We set $J$ to be the Jacobian ideal of $X \to V$ and set $I = J \cdot \O_{X}(-D_{\phi}) \subseteq \O_{X}$.  It follows from construction that $I \otimes_{k(V^e)} k(V^{\infty})$ is a nonzero ideal which is contained in the test ideal $\tau(X_{k(V^{\infty})}, \phi_{\infty})$ (this is because the Jacobian ideal is compatible with this base change, as is the divisorial ideal associated to $\phi$).  \todo{Karl: It might not be divisorial, do it more generally...}

More generally, if $s \in V$ is any perfect point, we can induce
\[
\phi_s : L_s^{1/p^e} := L^{1/p^e} \otimes_{A^{1/p^e}} k(s)^{1/p^e} \to R \otimes_{A} k(s)^{1/p^e}.
\]
Since $k(s) = k(s)^{1/p^e}$ is perfect, we can identify this with a $p^{-e}$-linear map:
\[
\psi_s : L_s^{1/p^e} \to R_{s}.
\]
As above, we see that the composition of $\psi_s$ in the sense of \cite[Section 4.1]{BlickleSchwedeSurveyPMinusE} coincides with the composition $\phi$ as in \autoref{subsec.ComposingMaps}.
Finally, we study this process with respect to divisors.

\begin{lemma}
\label{lem.DeltaPhiRestrictedEqualsDeltaPsi}
With notation as above, $\Delta_{\phi}|_{X_s} = \Delta_{\psi_s}$ where $\Delta_{\phi}$ is as in \autoref{def.PhiAsDivisor} and $\Delta_{\psi_s}$ is as in \cite[Section 4.]{BlickleSchwedeSurveyPMinusE}.
\end{lemma}
\begin{proof}
By \autoref{lem.PullingBackDelta}, we simply must show that $\Delta_{\phi_s}$ coincides with $\Delta_{\psi_s}$.  Working locally, and removing a set of relative codimension 2, we may assume that $\Phi \in \Hom_{R_{s^{1/p^e}}}( (R_s)^{1/p^e}, R_{s^{1/p^e}})$ generates the $\Hom$ as an $(R_s)^{1/p^e}$-module and that $\phi(\blank) = \Phi(z^{1/p^e} \cdot \blank)$.  Thus $\Delta_{\phi} = {1 \over p^e - 1} \Div_X(z)$.

We then identify $R_{s^{1/p^e}}$ with $R_s$ and hence $\phi$ with $\psi$.  Likewise, we can identify $\Phi$ with $\Psi$ which now generates $\Hom_{R_{s}}( (R_s)^{1/p^e}, R_{s})$ as an $(R_s)^{1/p^e}$-module.  Hence $\psi(\blank) = \Psi(z^{1/p^e} \cdot \blank)$ and so $\Delta_{\psi} = {1 \over p^e - 1} \Div_X(z)$ as desired.
\end{proof}

\section{Relative non-F-pure ideals}
\label{sec.RelativeNonFPureIdeals}

With notation as above, by the Hartshorne-Speiser-Lyubeznik-Gabber (HSLG-)Theorem \cite[Lemma 13.1]{Gabber.tStruc} \cf \cite{HartshorneSpeiserLocalCohomologyInCharacteristicP,LyubeznikFModulesApplicationsToLocalCohomology,BlickleTestIdealsViaAlgebras}, we know that the chain
\[
 R_{k(V^{\infty})} \supseteq \psi_{\infty}\big((L_{{k(V^{\infty})}})^{1/p^e}\big) \supseteq \psi_{\infty}^2\big((L_{{k(V^{\infty})}}^{p^e + 1})^{1/p^{2e}}\big) \supseteq \cdots \supseteq \psi_{\infty}^n\big((L_{{k(V^{\infty})}}^{p^{ne} - 1 \over p^e - 1})^{1/p^{ne}}\big) \supseteq \cdots
\]
eventually stabilizes.  Say that it stabilizes at
\begin{equation}
\label{eq.N0Definition}
n \geq n_0.
\end{equation}
For the rest of this section, we fix this integer $n_0$, and make the following definition.

\begin{definition}
With notation as above, we define the integer $n_0$ to be the \emph{uniform integer for $\sigma$ over the generic point of $V$}, and in general, it will be denoted by $n_{\sigma(\phi), k(V)}$.  We notice that for any point $\eta \in V$, we can base change $\Spec k(\eta) \to V$ and form a corresponding integer $n_{\sigma(\phi_{\eta}), k(\eta)}$.
\end{definition}

On the other hand, without the passing to $k(V^{\infty})$.  We have the images
\[
\begin{array}{rcll}
\ba_1 & := & \phi^1\big( L^{1/p^{e}} \big) & \subseteq R_{A^{1/p^{e}}},\\
 \ba_2 & := & \phi^2\big( (L^{(p^e + 1)})^{1/p^{2e}} \big) & \subseteq R_{A^{1/p^{2e}}},\\
  \dots \\
  \ba_n & := & \phi^n \big( (L^{p^{ne} - 1 \over p^e - 1})^{1/p^{ne}}\big) & \subseteq R_{A^{1/p^{ne}}}, \\
  \dots
  \end{array}
\]
These are ideals of different rings.  However, we do have the following relation for any $i > j$
\[
\Image\big(\ba_j \otimes_{A^{1/p^j}} A^{1/p^i} \to R_{A^{1/p^i}}\big) \supseteq \ba_i.
\]
This is straightforward so we leave it to the reader to check (note that the image of \eqref{eq.Twist2} contains the image of $\phi^2$).  Additionally, observe that if $A$ was regular, $A^{1/p^e}$ would be flat over $A$ by \cite{KunzCharacterizationsOfRegularLocalRings}, and so we could identify the tensor product $\ba_j \otimes_{A^{1/p^j}} A^{1/p^i}$ with its image in $R_{A^{1/p^i}}$.
Therefore, for any integer $n \geq i$, we set
\[
\ba_{i,n} = \Image\big(\ba_i \tensor_{A^{1/p^i}} A^{1/p^n} \to R_{A^{1/p^n}} \big)
\]
and consider the chain of ideals:
\[
R_{A^{1/p^{ne}}} \supseteq \ba_{1, n} \supseteq \ba_{2,n} \supseteq \dots \supseteq \ba_{n-1,n} \supseteq \ba_{n,n}.
\]

\begin{definition}[$n$th relative non-$F$-pure ideal]
For every integer $n \geq n_0 = n_{\sigma(\phi), k(V)}$, we define the \emph{$n$th limiting relative non-$F$-pure ideal} to be $\ba_{n,n} = \ba_n$.  It is denoted by $\sigma_n(X/V, \phi)$.
\end{definition}

%This might seem unsatisfactory at first, but w
We now obtain a relative version of the Hartshorne-Speiser-Lyubeznik-Gabber theorem.

%\todo{{\bf Zsolt:}
%Since we are not claiming that $U_n$ below equals $V$ for $n \gg 0$, we should give an example, where $U_n$ cannot be chosen to be $V$ for all $n$.\\
%{\bf Karl:} Ok, I'm including an example after this Proposition.  It's kind of a mess right now, but it does show $\sigma_1(X/V, \phi) \otimes_{A^{1/p}} A^{1/p^2} \neq \sigma_2(X/V, \phi)$.}

\begin{proposition}
\label{prop.StabilizingSigma}
Fix notation as above.  Then there exists a nonempty open subset $U \subseteq V$ of the base scheme $V$ satisfying the following for every integer $m \geq n \geq n_0$
\begin{equation}
\label{eq.sigmaNRestricts}
\sigma_m(X/V, \phi)|_{f^{-1}(U)} = \Image\Big(\sigma_{n}(X/V, \phi) \otimes_{A^{1/p^{n e}}} A^{1/p^{m e}} \to R_{A^{1/p^{me}}}\Big) \Big|_{f^{-1}(U)}.
\end{equation}
% Fix notation as above.  For every integer $n \geq n_0 = n_{\sigma(\phi), k(V)}$, there exists a nonempty open subset $U_n \subseteq V$ of the base scheme $V$ satisfying the following condition.  If one sets %$U_n^{ne} \subseteq V^{ne}$ to be the corresponding open set of $V^{ne} \cong V$ with
% $X_n = f^{-1}(U_n)$, then we have that for every $m \geq n$
% \begin{equation}
% \label{eq.sigmaNRestricts}
% \sigma_m(X/V, \phi)|_{X_n} = \Image\Big(\sigma_{n}(X/V, \phi) \otimes_{A^{1/p^{n e}}} A^{1/p^{m e}} \to R_{A^{1/p^{me}}}\Big) \Big|_{X_n}.
% \end{equation}
% Furthermore we may assume that $U_{n_0} \subseteq U_{n_0+1} \subseteq \cdots \subseteq U_{n} \subseteq U_{n+1} \subseteq \cdots$ form an ascending chain of open sets.
\end{proposition}
%\todo{{\bf Karl:} I'm using flatness of Frobenius for regular things lots when I'm doing these tensor products.  I should make it explicit where exactly things are being used.}
\begin{proof}
For any $m \geq n$, we consider the containment:
\[
\ba_{n, m} \supseteq \ba_{m,m}.
\]
Fix $k(V)$ to be the residue field of the generic point $\eta \in V$ and consider the induced containment:
\[
\ba_{n,m} \otimes_A k(V) = \ba_{n,m} \otimes_{A^{1/p^m}} k(V)^{1/p^m} \supseteq \ba_{m,m} \otimes_{A^{1/p^m}} k(V)^{1/p^m} = \ba_{m,m} \otimes_A k(V).
\]
since inverting an element is the same as inverting its $p$th power.
We notice two identifications
\[
\begin{array}{c}
\ba_{n,m} \otimes_{A^{1/p^m}} k(V)^{1/p^{\infty}} = \ba_{n,m} \otimes_{A^{1/p^m}} k(V^{\infty}) = \psi_{\infty}^n\left( \left(L_{{k(V^{\infty})}}^{p^{ne} - 1 \over p^e - 1} \right)^{1/p^{ne}}\right), \\
\ba_{m,m} \otimes_{A^{1/p^m}} k(V)^{1/p^{\infty}} = \ba_{m,m} \otimes_{A^{1/p^m}} k(V^{\infty}) = \psi_{\infty}^m\left( \left(L_{{k(V^{\infty})}}^{p^{me} - 1 \over p^e - 1} \right)^{1/p^{me}}\right).
\end{array}
\]
Thus
\[
\ba_{n,m} \otimes_A k(V)^{1/p^{\infty}} = \ba_{m,m} \otimes_A k(V)^{1/p^{\infty}}
\]
since $m > n \geq n_0$.  It immediately follows that
\[
\ba_{n,m} \otimes_{A^{1/p^m}} k(V)^{1/p^{m}} = \ba_{m,m} \otimes_{A^{1/p^m}} k(V)^{1/p^{m}}
\]
since $k(V)^{1/p^m} \subseteq k(V)^{1/p^\infty}$ is a faithfully flat extension.  But now by generic freeness \cite[Theorem 14.4]{EisenbudCommutativeAlgebraWithAView}, for a \emph{fixed $m$}, \autoref{eq.sigmaNRestricts} follows for some open set $U_{n,m}$ (in the case that $V$ is affine, which we can certainly reduce to, we can invert a single element of $A$ to form $U_{n,m}$).

We now vary $m$.  Choose $U = U_{n,n+1}$ that works for $m = n + 1$ and consider the diagram:
\[
\xymatrix@R=24pt{
\left(L^{p^{(n+2)e} - 1 \over p^e - 1} \right)^{\frac{1}{p^{(n+2)e}}} \ar@/_18pc/[ddddd]_{\phi^{n+2}} \ar@/^16pc/[dddd]^{(\phi^{n+1})^{1/p^e}\otimes \ldots} \\
\left(L^{p^{(n+1)e} - 1 \over p^e - 1}\right)^{\frac{1}{p^{(n+1)e}}} \otimes_{A^{1/p^{(n+1)e}}} A^{1/p^{(n+2)e}} \ar@/_12pc/[dddd]_{\phi^{n+1}} \ar@/^10pc/[ddd]^{(\phi^{n})^{1/p^e} \otimes \ldots} \\
\left(L^{p^{ne} - 1 \over p^e - 1}\right)^{\frac{1}{p^{ne}}} \otimes_{A^{1/p^{ne}}} A^{1/p^{(n+2)e}} \ar@/_8pc/[ddd]_{\phi^{n}} \\
\cdots\\
L^{1/p^e} \otimes_{A^{1/p^{e}}} A^{1/p^{(n+2)e}} \ar[d]^{\phi \otimes \ldots} \\
R_{A^{1/p^{(n+2)e}}}
}
\]
%\todo{\textbf{Zsolt}: Can I replace some of the $1/p^e$ with $\frac{1}{p^e}$? I'm especially thinking about places where there is $ne$ instead of $e$ (as above). \\{\bf Karl:}  Sounds good, go ahead.}
where the maps labeled $(\phi^{n})^{1/p^e}\otimes \ldots$ and $(\phi^{n+1})^{1/p^e}\otimes \ldots$ are induced as in \eqref{eq.Twist1}.  We know that over $U$, $\phi^{n+1}$ and $\phi^n$ have the same image.  Therefore, so do $(\phi^{n})^{1/p^e}\otimes \ldots$ and $(\phi^{n+1})^{1/p^e}\otimes \ldots$ again over $U$ (since tensor is right exact).  But then, composing with $(\phi\otimes \ldots)$ one more time, we know $\phi^{n+1} = (\phi \otimes \ldots) \circ ((\phi^{n})^{1/p^e} \otimes \ldots)$ and $\phi^{n+2} = (\phi \otimes \ldots) \circ ((\phi^{n+1})^{1/p^e}\otimes \ldots)$ also have the same image over $U$.  Thus they also share the image with $\phi^{n}$ over $U$. Hence, if $\ba_{n,n+1}|_U= \ba_{n+1,n+1}|_U$, then $\ba_{n,n+2}|_U=\ba_{n+2,n+2}|_U$. We use this iteratedly to obtain that $\ba_{n,m}|_U = \ba_{m,m}|_U$, which is exactly the statement of the proposition.
\end{proof}

%\todo{ {\bf Zsolt} change two to three below}

We give three examples of these $\sigma_n$ and $U_n$.  In the first example, we show that it is possible that the images $\sigma_n(X/V, \phi)$ never stabilize in the sense of \autoref{prop.StabilizingSigma} on all of $V$ but only over an open set.  We do the same in the second example, but with respect to a more interesting choice of $X$ and $\phi$.  Finally, we give an example where stabilization occurs at $n = 2$ (instead of at $n = 1$).

\begin{example}
Fix $k$ to be an algebraically closed field of characteristic $p > 2$, set $A = k[t]$ and set $R = k[x,t]$ with the obvious map $X \to V$.
Let $\phi : R^{1/p} = k[x^{1/p}, t^{1/p}] \to R_{A^{1/p}} = k[x,t^{1/p}]$ be the composition of the local generator $\beta \in \Hom_{R_{A^{1/p}}}(R^{1/p}, {R_{A^{1/p}}})$ with pre-multiplication by $t^{1\over p}$ (note since $\phi$ is $k[t^{1/p}]$-linear, this is also post-multiplication by $t^{1/p}$, and it corresponds to the divisor $\Delta_{\phi}=\frac{1}{p-1}\{ t =0\}$).   It easily follows that the image of $\phi$ is $\langle t^{1/p} \rangle \subseteq k[x, t^{1/p}]$.  By tensoring with $\otimes_{k[t^{1/p}]} k[t^{1/p^2}]$, we obtain that $\ba_{1,2} = \langle t^{1/p} \rangle = t^{1/p} \cdot k[x, t^{1/p^2}]$.  More generally, we see that $\ba_{1,b} = t^{1/p} \cdot k[x, t^{1/p^n}]$.  Now we compute $\ba_{2,2}$ as the image:
\[
k[x^{1/p^2}, t^{1/p^2}] \xrightarrow{\phi^{1/p}} k[x^{1/p}, t^{1/p^2}] \xrightarrow{\phi \otimes \ldots} k[x, t^{1/p^2}].
\]
The image of $\phi^{1/p}$ is $t^{1/p^2}$ and since the map $\phi \otimes \ldots$ is $k[t^{1/p^2}]$-linear, we see that the composition has image $\langle t^{1/p} \cdot t^{1/p^2} \rangle = \langle t^{(1+p)/p^2} \rangle = \ba_{2,2}$.  In general, we see that
\[
\ba_n = \ba_{n,n} = \langle t^{(1+p+\dots+p^{n-1})/p^n }\rangle
\]
In particular, while we may take $U_i = \bA^{1} \setminus \{ 0 \} = \Spec A \setminus \langle t \rangle$, we see that $\ba_n = \sigma_n(X/V, \phi)$ never stabilizes over all of $V$.
\end{example}

\begin{example}
\label{ex:cusp}
Let $f$ be the morphism $X := \Spec \left( \frac{k[x,y,t]}{(y^2 + x^3+t)} \right) \to V := \Spec ( k[t])$. Let the standard trace map $\phi : (F_{X^1/V^1})_* \omega_{X^1/V^1} \to \sO_{X_{V^1}}$, for which $\Delta_{\phi}=0$. This map can be identified with the descent of the following  map to the quotient
\begin{equation*}
\xymatrix{
k[x,y,t] \ar[rrr]^{\cdot (y^2 + x^3 + t)^{p-1}} & & & k[x,y,t] \ar[r]^{\mathrm{Tr}} & k[x,y,t] \\
}
\end{equation*}
where $\mathrm{Tr}$ is a $k$-linear map, such that
\begin{equation*}
\mathrm{Tr} (x^iy^j t^l)=
\left\{
\begin{matrix}
x^{\frac{i+1 -p}{p}}y^{\frac{j+1-p}{p}} t^l & \textrm{ if $p|i+1$ and $p|j+1$} \\
0 & \textrm{ otherwise} \\
\end{matrix}
\right. .
\end{equation*}
Then $\phi^n$ is given by
\begin{equation*}
\xymatrix{
k[x,y,t] \ar[rrr]^{\cdot (y^2 + x^3 + t)^{p^n -1}} & & & k[x,y,t] \ar[r]^{\mathrm{Tr_n}} & k[x,y,t] \\
}
\end{equation*}
where $\mathrm{Tr_n}$ is a $k$-linear map, such that
\begin{equation}
\label{eq:cusp:tr_n}
\mathrm{Tr_n} (x^iy^j t^l)=
\left\{
\begin{matrix}
x^{\frac{i+1-p^n}{p^n}}y^{\frac{j+1-p^n}{p^n}} t^l & \textrm{ if $p^n | i+1$ and $p^n | j+1$} \\
0 & \textrm{ otherwise} \\
\end{matrix}
\right. .
\end{equation}
Assume that $p$ is a prime such that $p \equiv 1 (\text{mod }6)$ so that $2,3|p^n -1$ for all $n > 0$. Let us try to compute now the following number.
%\begin{multline*}
%d:=\min \{ c | t^c \in \im \phi^n \} =
%%
%\underbrace{\min \{ c|  x^{p^n-1} y^{p^n -1} t^c \in \langle (y^2 + x^3 + t)^{p^n -1} \rangle \}}_{\textrm{by \autoref{eq:cusp:tr_n}}}
%%
%\\ = \min \left\{p^n-1 -a -b \left| a,b \in \bZ,  2a \leq p^n -1, 3b \leq p^n -1, \frac{(p^n-1)!}{a! b! (p^n-1-a-b)!} \neq 0   \right.  \right\}
%\end{multline*}
\[
d:=\min \{ c | t^c \in \im \phi^n \}
\]
To have $t^c \in \im \phi^n$ it is necessary to have a polynomial in the ideal generated by $(y^2 + x^3 + t)^{p^n -1}$ one of the non-zero monomials of which is $x^{p^n-1} y^{p^n -1} t^c$. Therefore, $(y^2 + x^3 + t)^{p^n -1}$ itself has to have a non-zero monomial which divides $x^{p^n-1} y^{p^n -1} t^c$. That is, there have to be integers $0 \leq a \leq \frac{p^n -1}{2}$ and $0 \leq b \leq \frac{p^n -1}{3}$, such that  $c= p^n -1 -a-b$. So, we see that
\begin{equation*}
d \geq p^n-1 - \frac{p^n -1}{2} - \frac{p^n -1}{3} =
\frac{p^n -1 }{6}
\end{equation*}
On the other hand we claim that $\Tr_n(y^2 + x^3 + t)^{p^n -1} = t^{\frac{p^n -1}{6}}$, which will show that $d= \frac{p^n-1}{6}$ and also that $\sigma_n(X/V, \phi)$ does not stabilize. First, note that $x^{\frac{p^n-1}{2}} y^{\frac{p^n-1}{3}} t^{\frac{p^n-1}{6}}$ is the only monomial in the expansion of $(y^2 + x^3 + t)^{p^n-1}$ the image of which via $\Tr_n$ is not zero. Indeed, the expansion can contain only monomials of the form $y^{2a} x^{3b}  t^{p^n-1-a-b}$, where $a, b \geq 0$ are integers. Hence, higher powers ($2p^n-1$, $3p^n -1$, etc) of $x$ and $y$ that do not go to zero by $\Tr_n$ cannot be obtained. Hence, as we stated  $x^{\frac{p^n-1}{2}} y^{\frac{p^n-1}{3}} t^{\frac{p^n-1}{6}}$ is the only interesting monomial that can show up in the expansion. Lastly we have to verify that the monomial it has non-zero coefficient in the expansion. That is,
\begin{equation*}
\frac{(p^n-1)!}{\frac{p^n-1}{2}!\frac{p^n-1}{3}!\frac{p^n-1}{6}!} \neq 0
\end{equation*}
Equivalently, we have to show that the power of $p$ in the prime factorization of the numerator is the same as in the denominator. For an arbitrary number $m$ this number for $m!$ is
\begin{equation*}
\sum_{i=1}^{\infty} \left\lfloor \frac{m}{p^i} \right\rfloor.
\end{equation*}
Now assume that $m = \frac{p^n -1}{r}$, where $r | p-1$. Then the times $p$ divides $m!$ is
\begin{equation*}
\begin{split}
\sum_{i=1}^{\infty} \left\lfloor \frac{m}{p^i} \right\rfloor
& = \sum_{i=1}^{n-1}  \left\lfloor \frac{p^n-1}{rp^i} \right\rfloor
 = \sum_{i=1}^{n-1}  \left\lfloor \sum_{j=0}^{n-1} p^j \frac{p-1}{rp^i} \right\rfloor
 = \sum_{i=1}^{n-1}  \left\lfloor \sum_{j=0}^{n-1} p^{j-i} \frac{p-1}{r} \right\rfloor
\\ & = \sum_{i=1}^{n-1}   \sum_{j=i}^{n-1} p^{j-i} \frac{p-1}{r}
 = \sum_{i=1}^{n-1}   \frac{p^{n-i}-1}{p-1} \frac{p-1}{r}
 = \sum_{i=1}^{n-1}   \frac{p^{n-i}-1}{r}
 = \sum_{i=1}^{n-1}   \frac{p^{i}-1}{r}
\\ & = \sum_{i=0}^{n-1}   \frac{p^{i}-1}{r}
 = \frac{p^n -1}{r}
\end{split}
\end{equation*}
Therefore, we see that $p$ divides $p^n-1$ times both $(p^n-1)!$ and $\frac{p^n-1}{2}!\frac{p^n-1}{3}!\frac{p^n-1}{6}!$, as claimed earlier. Summarizing, $t^{\frac{p^n-1}{6}}$ is the smallest power of $t$ that is in $\sigma_n(X/V, \phi)$. Hence there is no stabilization over all of $V$.

Based on this example, one might ask:
\begin{question}
Is there  a relation between the asymptotics of the $\sigma_n$ over the non-stable locus, and the singularities of those fibers (for example, the $F$-pure threshold)?
\end{question}
\end{example}

We now give an example that stabilizes at the second step (over the entire base).

\begin{example}\textnormal{(\cf \cite[Example 4.7]{MustataYoshidaTestIdealVsMultiplierIdeals})}
\label{MustataYoshidaExample}
Fix $k$ to be an algebraically closed field of characteristic $p > 2$, set $A = k[t]$ and set $R = k[x,t]$ with the obvious map $X \to V$.
Let $\phi : R^{1/p} = k[x^{1/p}, t^{1/p}] \to R_{A^{1/p}} = k[x,t^{1/p}]$ be the composition of the local generator $\beta \in \Hom_{R_{A^{1/p}}}(R^{1/p}, {R_{A^{1/p}}})$ with pre-multiplication by $(x^{p^2} + t )^{1\over p} = f^{1 \over p}$ (which corresponds to $\Delta_{\phi}=\frac{1}{p-1}\{x^{p^2} = t \}$ restricting to $\frac{p^2}{p-1}\{x = \lambda^{1/p^2} \}$ on the fiber over $t=\lambda$).  Note that in $R_{A^{1/p}}$ or $R^{1/p}$, $f^{1\over p}$ can be written as $x^p + t^{1/p}$.  In particular, $f^{1\over p}$ is already an element of $R_{A^{1/p}}$. Therefore, since $\phi$ is $R_{A^{1/p}}$-linear and $\beta$ is clearly surjective, the image of $\phi$ is just $\langle f^{1/p} \rangle = \sigma_1(X/V, \phi) = \left\langle x^p + t^{1/p} \right\rangle$.

%Now, we have the following factorization of the generator of $\Phi \in \Hom_{R_{A^{1/p}}}(R_{A^{1/p}}^{1/p}, R_{A^{1/p}})$.
%\[
%R_{A^{1/p}}^{1/p} \xrightarrow{\alpha} R^{1/p} \xrightarrow{\beta} R_{A^{1/p}}.
%\]
%Consider the image of $\Phi$ pre-multiplied by $f^{1 \over p}$.  That has the same image as $\phi$ since clearly $\Phi(f^{p-1\over p} \cdot x) = \beta(f^{p-1\over p} \cdot \alpha(x))$, and since $\alpha$ is surjective.  On the other hand, since $\Phi$ itself is surjective, we see that the image of $\phi$ is simply $x^p + t^{1/p}$.  Thus $\sigma_1(X/V, \phi) = \langle x^p + t^{1/p} \rangle$.

On the other hand, we now compose $\phi$ with itself as described above.  In this case, $\phi^2 : R^{1/p^2} \to R_{A^{1/p^2}}$ is induced by taking the generator $\beta^2 : R^{1/p^2} \to R_{A^{1/p^2}}$ and pre-multiplying by $f^{p+1 \over p^2}$.  Now, $f^{p+1\over p^2} = \left(x^{1} + t^{1/p^2} \right)^{p+1}$.  This again is already an element of $R_{A^{1/p^2}}$, and so by the same argument as above, we see that
\begin{equation*}
\sigma_2(X/V, \phi) = \left\langle x + t^{1/p^2} \right\rangle^{p+1}= \sigma_1(X/V, \phi) \cdot \left\langle x + t^{1/p^2} \right\rangle = \left\langle f^{1+p \over p^2} \right\rangle.
\end{equation*}

Next we form $\phi^3$.  In this case, it is obtained by taking a generator $R^{1/p^3} \to R_{A^{1/p^3}}$ and pre-multiplying by $f^{1 + p + p^2 \over p^3}$.  However, this time $f^{1+p+p^2 \over p^3}$ is \emph{not} contained in $R_{A^{1/p^3}}$, and so we cannot argue as above. However, $f$ still has a $1/p^2$-root in $R_{A^{1/p^3}}$, and so
\[
\phi^3 \left( \left\langle f^{1+p+p^2 \over p^3} \right\rangle \right) = \phi^3\left(f^{p(1+p)}\left\langle f^{1/p^3} \right\rangle\right) = f^{1+p\over p^2} \phi^3 \left( \left\langle f^{1/p^3} \right\rangle\right).
\]
So we must only compute $\phi^3(\langle f^{1/p^3} \rangle)$.  Now, $\phi^3 : k[x^{1/p^3},t^{1/p^3}] \to k[x, t^{1/p^3}]$ we may take as the map which sends $x^{p^3-1 \over p^3}$ to $1$ and other monomials in $x$ to $0$ (those monomials form a basis for $R^{1/p^3}$ over $R_{A^{1/p^3}}$).  It is thus obvious that $\phi^3 \left(\left\langle f^{1/p^3} \right\rangle\right) = k[x, t^{1/p^3}] = R_{A^{1/p^3}}$.  In particular, \begin{equation*}
\sigma_3(X/V, \phi) = \left\langle f^{1+p \over p^2} \right\rangle = \sigma_2(X/V, \phi) \otimes_{A^{1/p^2}} A^{1/p^3}.                                                                                                                                                                                                                                                                                                                                                                                                                                                                                                                                                                                                                                                                                                                                                                                                                                         \end{equation*}

\end{example}

We now obtain a surprising base change statement.

\begin{proposition}%[Relative non-$F$-pure ideals and base change]
\label{prop.RelativeNonFPureIdealsAndBaseChange}
Suppose that $T \to V$ is a map from an excellent integral scheme with a dualizing complex, then using the notation of \autoref{subsec.BaseChangeOfPhi}
\[
\Image\left((q_{ne})^* \sigma_n(X/V, \phi) \rightarrow \O_{X_{T^{ne}}}\right) =  \sigma_n(X/V, \phi) \cdot \O_{X_{T^{ne}}} = \sigma_n(X_T/T, \phi_T).
\]
Furthermore, if $U $ satisfies condition \eqref{eq.sigmaNRestricts} from Proposition \ref{prop.StabilizingSigma}, then $$W = g^{-1}(U) \subseteq T$$ satisfies the same condition for $\sigma_n(X_T/T, \phi_T)$.
\end{proposition}
\begin{proof}
Indeed, images of maps are compatible with arbitrary base change by the right exactness of tensor.  Thus the first statement follows immediately.
The second statement follows from the first since if the two images $\ba_{n,n}$ and $\ba_{n-1, n}$ are equal, they are also equal after base change.
\end{proof}

\begin{theorem}[Base change for $\sigma_n$]
\label{thm.BaseChangeSigmaGeneral}
There exists an integer $N \geq 0$, such that for all points $s \in V$, $N \geq n_{\sigma(\phi_s), k(s)}$.  In other words, we have both that $\sigma_n(X/V, \phi)\cdot \O_{X_{s^{ne}}} = \sigma_n(X_s/s, \phi_s)$ (which always holds) and also that for all $m \geq n \geq N$ that
\[
\sigma_n(X/V, \phi) \otimes_{V^{n}} k(s)^{1/p^m} = \sigma_m(X_s/s, \phi_s).
\]
\end{theorem}
\begin{proof}
Of course, the statement already holds on $U$ with respect to some integer $N_0$.
Set $V_1' := V \setminus U$ to be the complement with reduced scheme structure and let $$i_1 : V_1 = (V_1')_{\reg} \hookrightarrow V_1'$$ be the regular locus.  We notice that $V_1$ has dimension strictly smaller than $\dim V$.  We consider the base change $X_{V_1} \to V_1$.  Each fiber of $X_{V_1} \to V_1$ is isomorphic to a fiber of $X \to V$.  Then choose an open set $U_1 \subseteq V_1$ for which the statement holds for some integer $N_1$.

Now fix $V_2' = V_1' \setminus U_1$ and $i_2 : V_2 = (V_2')_{\reg} \hookrightarrow V_2'$ to be the regular locus and repeat.
%Again, we note that $V_2$ has dimension strictly smaller than $\dim V_1$.
This process terminates by Noetherian induction.

Setting $N = \max\{N_0, N_1, N_2, \dots\}$ completes the proof.
\end{proof}

%Now suppose that $V$ is a variety over a perfect field $k$.  For every point $s \in V$, notice that since $k$ is perfect and $k(s)$ is a finite extension of $k$, $k(s)$ is also perfect.  In particular, $s = s^e$ and so %$X_{s^e} \cong X_s \times_s s^e \cong X_s$.  Thus we obtain $\phi_s : L_s^{1/p^e} \to R_s$.
%Combining all of the ideas so far, we see that $\sigma_n(X/V, \phi)$ restricts to the ordinary non-$F$-pure ideal $\sigma(X_s, \phi_s)$, defined as in \cite{FujinoSchwedeTakagiSupplements}, on all closed fibers for $n \gg 0$.

Suppose now that $s \in V$ is a perfect point, see \autoref{def.PerfectPoint}.  It follows that $X_s/s$ is a variety over a perfect field and so since $s^e \cong s$, we have $X_{s^e} \cong X_s$.  Thus we can identify $\phi_s : L_s^{1/p^e} \to R_{s^e}$ with a $p^{-e}$-linear map $\psi_s$ as in \autoref{sec.BaseChangeToPerfectPoints}.  Thus under these identifications $\bigcap_{n \geq 0} \sigma_n(X_s/s, \phi) = \sigma(X_s, \psi_s)$.  However, if we have the stabilization $\sigma_n(X_s/s, \phi) \cong \sigma_n(X_s/s, \phi) \otimes_{k(s)^{1/p^{ne}}} k(s)^{1/p^{me}} \cong \sigma_{m}(X_s/s,\phi)$ for every $m \geq n$, the equality $\sigma_n(X_s/s, \phi)  = \sigma(X_s, \psi)$ holds.  Combining this with our previous work we obtain the following corollary.

\begin{corollary}
\label{thm.UniversalNForAllClosedFibers}
With notation as above, there exists an integer $N \geq 0$ such that $$\sigma_n(X/V,\phi)\cdot \O_{X_{s^{ne}}} = \sigma(X_s, \psi_s)$$ for all perfect points $s \in V$ and $n \geq N$.
\end{corollary}

\mysubsection{Iterated non-$F$-pure ideals versus the absolute non-$F$-pure ideal}
\label{subsec.IteratedNonFPureIdealsVsAbsolute}

%Suppose that $S$ is such that $K_S = 0$ (for example, if $S$ is affine).

We now compare the iterated non-$F$-pure ideal with the absolute non-$F$-pure ideal.  First however, we make the following assumption.
\begin{convention}
\label{conv.BaseIsNice}
In this subsection, \autoref{subsec.IteratedNonFPureIdealsVsAbsolute}, we always assume that our base $V$ is an $F$-pure quasi-Gorenstein (i.e., $\omega_V$ is a line bundle) $F$-finite integral scheme.  %Often we additionally assume that $V$ is regular.
\end{convention}
Let $\Psi_V^{e} : F^{e}_* \O_V((1-p^e)K_V) \to \O_V$ be a map generating $\sHom_V(F^{e}_* \O_V((1-p^e)K_V), \O_V)$ as an $F^{e}_* \O_V$-module.  By taking $p^{ne}$th roots, applying $f^{-1}$, and then writing the $A$-module $M := f^{-1}(\O_V( (1-p^e)K_V))$, we obtain:
\[
\left(\Phi^e_A \right)^{1/p^{ne}} :  M^{1/p^{(n+1)e}} =f^{-1} F^{(n+1)e}_* (\O_V( (1-p^e)K_V)) \to f^{-1} F^{ne}_* \O_V = A^{1/p^{ne}}.
\]
We tensor with $L^{1/p^{e}} \otimes_{A^{1/p^{e}}} \blank$ and so obtain:
\begin{equation}
\label{eq.TwistiedPhiA}
\Phi' : L^{1/p^{e}} \otimes_{A^{1/p^{e}}} M^{1/p^{(n+1)e}} \to L^{1/p^{e}} \otimes_{A^{1/p^e}} A^{1/p^{ne}}.
\end{equation}
Set $N = L \otimes_A M^{1/p^{ne}}$. We observe that $N$ is a line bundle on $X_{V^{ne}}$.  In fact, if $q_1 : X_{V^{ne}} = X \times_V V^{ne} \to X$ is the first projection and $q_2 : X_{V^{ne}} \to V^{ne}$ is the second projection, then
\begin{equation*}
\label{eq.NAsPullBack}
N = q_1^* L \otimes_{\O_{X_{V^{ne}}}} q_2^* \O_{V^{ne}}( (1-p^e)K_{V^{ne}})
\end{equation*}
and so it follows for any integer $l > 0$ that
\begin{equation}
\label{eq.NToPower}
N^{l} \cong q_1^* L^l \otimes_{\O_{X_{V^{ne}}}} q_2^* \O_{V^{ne}}( l(1-p^e)K_{V^{ne}}) = L^l \otimes_A \left(M^l \right)^{1/p^{ne}}
\end{equation}
Finally, we compose \autoref{eq.TwistiedPhiA} with $\phi' := \phi \otimes_{A^{1/p^e}} A^{1/p^{ne}}$ to construct:
\begin{equation}
\label{eq.defnGamma}
\begin{array}{rrl}
\gamma :&  & N^{1/p^e}\\
& = & L^{1/p^{e}} \otimes_{A^{1/p^{e}}} M^{1/p^{(n+1)e}} \\
& \xrightarrow{\Phi'}  & L^{1/p^{e}} \otimes_{A^{1/p^e}} A^{1/p^{ne}} \\
& \xrightarrow{\phi'} & R_{A^{1/p^e}} \otimes_{A^{1/p^e}} A^{1/p^{ne}} \\
& = & R \otimes_A A^{1/p^{ne}} \\
& = & R_{A^{1/p^{ne}}}.
\end{array}
\end{equation}
This is a map from an invertible sheaf on $X^e_{V^{(n+1)e}} = \left(X_{V^{ne}}\right)^e$ to the structure sheaf on $X_{V^{ne}}$.   In particular, it is a map such as one studied in \cite[Section 4]{BlickleSchwedeSurveyPMinusE}.
%In case $\phi$ is relatively divisorial, then the total space $X$ is G1 and S2 absolutely (since the base is Gorenstein).  In this case $\gamma$ corresponds to a Weil divisorial sheaf.  Furthermore, in \autoref{lem.ChoiceOfDeltaPhiEqualsDeltaGamma} below we will see that $\Delta_{\gamma}= (\Delta_{\phi})_{V^{ne}}$.

Let us point out an alternate way to construct $\gamma$.  Indeed, take
\[
{\phi}'' := \phi \tensor_{A^{1/p^{e}}} M^{1/p^{(n+1)e}} : L^{1/p^{e}} \tensor_{A^{1/p^{e}}} M^{1/p^{(n+1)e}} \to R_{A^{1/p^e}} \tensor_{A^{1/p^e}} M^{1/p^{(n+1)e}} \cong R \tensor_A M^{1/p^{(n+1)e}}.
\]
We can then compose this with $R_{A^{1/p^e}} \tensor_{A^{1/p^e}} (\Phi_A^e)^{1/p^{ne}} = (\Phi_A^e)'$ to obtain
\[
N^{1/p^e} = L^{1/p^{e}} \tensor_{A^{1/p^{e}}} M^{1/p^{(n+1)e}} \xrightarrow{\phi''} R_{A^{1/p^e}} \tensor_{A^{1/p^e}} M^{1/p^{(n+1)e}} \xrightarrow{(\Phi_A^e)'} R_{A^{1/p^{e}}}\tensor_{A^{1/p^e}} A^{1/p^{ne}} = R_{A^{1/p^{ne}}}.
\]
This composition is easily seen to coincide with $\gamma$.

%\todo{ {\bf Zsolt:} Karl, could you verify the last sentence that I added (about $\Delta_{\phi}$?\\{\bf Karl:}  Yes, that is right.}

We can then compose $\gamma$ with itself $m$-times as in \cite[Section 4]{BlickleSchwedeSurveyPMinusE} or \cite{SchwedeFAdjunction}.  We recall this construction for the benefit of the reader.  To construct $\gamma^2$, we tensor the map
\[
\gamma: N^{1/p^e} \to R\otimes_A A^{1/p^{ne}} = R_{A^{1/p^{ne}}}
\]
with the line bundle $ \otimes_{R_{A^{1/p^{ne}}}} N$, take $p^e$th roots and then compose it with $\gamma$ to obtain:
\[
\gamma^2 : \left(N^{1 + p^e} \right)^{1/p^{2e}} \to N^{1/p^e} \xrightarrow{\gamma} R \otimes_A A^{1/p^{ne}}
\]
Recursively, we can construct:
\[
\begin{array}{rcl}
\gamma^m & : & \Big(N^{p^{me} - 1 \over p^e - 1}\Big)^{1/p^{me}} \\
& = & \Big(L^{p^{me} - 1 \over p^e - 1}\Big)^{1/p^{me}} \otimes_{A^{1/p^{me}}} \Big(M^{p^{me}-1 \over p^e - 1}\Big)^{1/p^{(n+m)e}}\\
& \to & R_{A^{1/p^{ne}}}
\end{array}
\]
We are now in a position to relate the absolute $\sigma(X_{V^{1/p^{ne}}}, \gamma)$ with the relative $\sigma_n(X/V, \phi)$.

First we observe that for each $m \leq n$, $\Phi^m_A$ induces maps
\[
\begin{array}{rrl}
\mu_{m} : & & \Big(N^{p^{me} - 1 \over p^e - 1}\Big)^{1/p^{me}} \\
& = & \Big(L^{p^{me} - 1 \over p^e - 1}\Big)^{1/p^{me}} \otimes_{A^{1/p^{me}}} \Big(M^{p^{me}-1 \over p^e - 1}\Big)^{1/p^{(n+m)e}}\\
& \xrightarrow{\ldots \otimes (\Phi^{m}_A)^{1/p^{ne}}} & \Big(L^{p^{me} - 1 \over p^e - 1}\Big)^{1/p^{me}} \otimes_{A^{1/p^{me}}} A^{1/p^{ne}}.
\end{array}
\]
Furthermore these maps are surjective since $\Phi^m_A$ is.

\begin{lemma}
\label{lem.FactorizationOfSigmaImages}
Assuming \autoref{conv.BaseIsNice}, we have the following factorization of $\gamma^{n}$ and $\gamma^{n-1}$ as indicated by the diagram below

%\[
%   \scriptsize
%\xymatrix@C=12pt{
%& \Big(N^{p^{ne} - 1 \over p^e - 1}\Big)^{1/p^{ne}} \ar@{=}[d] & \Big(N^{p^{(n-1)e} - 1 \over p^e - 1}\Big)^{1/p^{(n-1)e}} \ar@{=}[d] \\
%& \ar`l[dd]`/1pc[dd][ddr]_-{\gamma^n} \Big(L^{p^{ne} - 1 \over p^e - 1}\Big)^{1/p^{ne}} \otimes_{A^{1/p^{ne}}} \Big(M^{p^{ne} - 1 \over p^{e}-1}\Big)^{1/p^{(n+n)e}} \ar[r]^-{\gamma'} \ar[d]_-{\mu_n} & \Big(L^{p^{(n-1)e} - 1 \over p^e - 1}\Big)^{1/p^{(n-1)e}} \otimes_{A^{1/p^{(n-1)e}}} \Big(M^{p^{(n-1)e} - 1 \over p^e - 1}\Big)^{1/p^{(n+n-1)e}} \ar[d]^-{\mu_{n-1}} \ar@/^10pc/[dd]^{\gamma^{n-1}} \\% \ar`r[d]`/4pc[dd]^{\gamma^{n-1}}[dd] \\
%& \Big(L^{p^{ne} - 1 \over p^e - 1}\Big)^{1/p^{ne}} \ar[r] \ar[dr]_{\phi^n} & \Big(L^{p^{(n-1)e} - 1 \over p^e - 1}\Big)^{1/p^{(n-1)e}} \otimes_{A^{1/p^{(n-1)e}}} A^{1/p^{ne}} \ar[d]^-{\phi^{n-1} \otimes \ldots} \\
%& & R_{A^{1/p^{ne}}}
%}
%\]
{
\scriptsize
\begin{tikzcd}
\Big(N^{p^{ne} - 1 \over p^e - 1}\Big)^{1/p^{ne}}
\arrow[d, equal]
& \Big(N^{p^{(n-1)e} - 1 \over p^e - 1}\Big)^{1/p^{(n-1)e}}
\arrow[d, equal]
\\
\Big(L^{p^{ne} - 1 \over p^e - 1}\Big)^{1/p^{ne}} \otimes_{A^{1/p^{ne}}} \Big(M^{p^{ne} - 1 \over p^{e}-1}\Big)^{1/p^{(n+n)e}}
\arrow[r, "\gamma'"]
\arrow[d, "\mu_n"]
\arrow[ddr, "\gamma^n"' near end, rounded corners, to path={ -- ([xshift=-2ex]\tikztostart.west)  |- (\tikztotarget.west)\tikztonodes} ]
%\arrow[dashed, rounded corners, to path={ -- ([yshift=-2ex]\tikztostart.south) -| ([xshift=-1.5ex]\tikztotarget.west) -- (\tikztotarget)}]{uulll}
& \Big(L^{p^{(n-1)e} - 1 \over p^e - 1}\Big)^{1/p^{(n-1)e}} \otimes_{A^{1/p^{(n-1)e}}} \Big(M^{p^{(n-1)e} - 1 \over p^e - 1}\Big)^{1/p^{(n+n-1)e}}
\arrow[d, "\mu_{n-1}"]
\arrow[dd, "\gamma^{n-1}"' near start, rounded corners, to path={ -- ([xshift=2ex]\tikztostart.east) |- (\tikztotarget.east)\tikztonodes} ]
\\
\Big(L^{p^{ne} - 1 \over p^e - 1}\Big)^{1/p^{ne}}
\arrow[r]
\arrow[dr, "\varphi^n"']
& \Big(L^{p^{(n-1)e} - 1 \over p^e - 1}\Big)^{1/p^{(n-1)e}} \otimes_{A^{1/p^{(n-1)e}}} A^{1/p^{ne}}
\arrow[d, "\varphi^{n-1} \otimes \ldots"]
\\
& R_{A^{1/p^{ne}}}
\end{tikzcd}
}
\noindent
where the arrow $\gamma'$ is induced by $\gamma$ and $\phi^{n-1} \otimes \ldots$ is simply $\phi^{n-1} \otimes_{A^{1/p^{(n-1)e}}} A^{1/p^{ne}}$.  Furthermore, since $\mu_{n-1}$ is surjective we have
\[
\Image(\gamma^{n-1}) = \ba_{n-1,n} = \Image( \phi^{n-1} \otimes_{A^{1/p^{(n-1)e}}} A^{1/p^{ne}}).
\]
Likewise since $\mu_n$ is surjective, we have
\[
\Image(\gamma^{n}) = \ba_{n,n} = \ba_n = \Image(\phi^n).
\]
\end{lemma}
\begin{proof}
The two equalities are immediate from the surjectivities of $\mu_{n-1}$ and $\mu_n$ and the commutativity of our diagram. The surjectivities of $\mu_{n-1}$ and $\mu_n$ are clear since $M^{1/p^{(n+1)e}}\to A^{1/p^{ne}}$ is surjective and tensor is right-exact. It remains to prove the commutativity of our diagram. It suffices to prove the commutativity of the square:
\[
\scriptsize
\xymatrix@C=18pt{
\Big(L^{p^{me} - 1 \over p^e - 1}\Big)^{1/p^{me}} \otimes_{A^{1/p^{me}}} \Big(M^{p^{me} - 1 \over p^{e}-1}\Big)^{1/p^{(n+m)e}} \ar[r]^-{\gamma'} \ar[d]^{\mu_m} &  \Big(L^{p^{(m-1)e} - 1 \over p^e - 1}\Big)^{1/p^{(m-1)e}} \otimes_{A^{1/p^{(m-1)e}}} \Big(M^{p^{(m-1)e} - 1 \over p^{e}-1}\Big)^{1/p^{(n+m-1)e}} \ar[d]^{\mu_{m-1}} \\
\Big(L^{p^{me} - 1 \over p^e - 1}\Big)^{1/p^{me}}\otimes_{A^{1/p^{me}}}A^{1/p^{ne}} \ar[r] & \Big(L^{p^{(m-1)e} - 1 \over p^e - 1}\Big)^{1/p^{(m-1)e}}  \otimes_{A^{1/p^{(m-1)e}}} A^{1/p^{ne}}
}
\]
and we will prove it by induction on $m$. It is clear that, if the above square is commutative for $m$, then, by tensoring it with $L\otimes_R(-)\otimes_{A^{1/p^{ne}}}M^{1/p^{(n+1)e}}$ and taking $p^e$th roots, we will have the commutativity for $m+1$. Hence it remains to prove the commutativity when $m=2$. To this end, we will denote the map $M^{1/p^{(n+1)e}}\to A^{1/p^{ne}}$ by $\alpha$. Then $\alpha$ induces a map
\begin{align}
\alpha'=(id \otimes \alpha)^{1/p^e}:\Big(M^{p^{2e} - 1 \over p^{e}-1}\Big)^{1/p^{(n+2)e}} &= \Big(M^{1/p^{ne}}\otimes_{A^{1/p^{ne}}}M^{1/p^{(n+1)e}}\Big)^{1/p^e} \notag\\
&\to M^{1/p^{ne}}\otimes_{A^{1/p^{ne}}}A^{1/p^{ne}}=M^{1/p^{(n+1)e}}\notag
\end{align}
It is clear that $\alpha\circ \alpha'=\alpha^2$. Likewise, the map $\varphi:L^{1/p^e}\to R\otimes_AA^{1/p^e}$ induces
\[\varphi'=(id \otimes \varphi)^{1/p^e}: \Big(L^{p^{2e} - 1 \over p^{e}-1}\Big)^{1/p^{2e}} = \Big(L^{1/p^{e}}\otimes_{R^{1/p^{e}}}L^{1/p^{2e}}\Big)^{1/p^e} \to L^{1/p^e}\otimes_{A^{1/p^e}}A^{1/p^{2e}}\]
We will analyze the 4 maps involved in the square. The left vertical map
\[
\mu_2:\Big(L^{p^{2e} - 1 \over p^e - 1}\Big)^{1/p^{2e}} \otimes_{A^{1/p^{2e}}} \Big(M^{p^{2e} - 1 \over p^{e}-1}\Big)^{1/p^{(n+2)e}}\to \Big(L^{p^{2e} - 1 \over p^e - 1}\Big)^{1/p^{2e}}
 \]
 is given by $id\otimes (\alpha\circ \alpha')$; similarly the right vertical map $\mu_1$ is $id\otimes \alpha$. The top horizontal map is given by
\begin{align}
\Big(L^{p^{2e} - 1 \over p^e - 1}\Big)^{1/p^{2e}} \otimes_{A^{1/p^{2e}}} \Big(M^{p^{2e} - 1 \over p^{e}-1}\Big)^{1/p^{(n+2)e}} & \xrightarrow{\varphi'\otimes \alpha'} L^{1/p^e}\otimes_{A^{1/p^e}}A^{1/p^{2e}}\otimes_{A^{1/p^{2e}}} M^{1/p^{(n+1)e}}\notag\\
&\xrightarrow{\sim} L^{1/p^e}\otimes_{A^{1/p^e}}M^{1/p^{(n+1)e}}\notag
\end{align}
And the bottom horizontal map is give by
\begin{align}
\Big(L^{p^{2e} - 1 \over p^e - 1}\Big)^{1/p^{2e}} \otimes_{A^{1/p^{2e}}} A^{1/p^{ne}} & \xrightarrow{\varphi'\otimes id} L^{1/p^e}\otimes_{A^{1/p^e}}A^{1/p^{2e}} \otimes_{A^{1/p^{2e}}} A^{1/p^{ne}}\notag\\
&\xrightarrow{\sim} L^{1/p^e}\otimes_{A^{1/p^e}}A^{1/p^{ne}}.\notag
\end{align}

Now given an arbitrary $x\otimes z \in \Big(L^{p^{2e} - 1 \over p^e - 1}\Big)^{1/p^{2e}} \otimes_{A^{1/p^{2e}}} \Big(M^{p^{2e} - 1 \over p^{e}-1}\Big)^{1/p^{(n+2)e}}$, write $\varphi'(x)=\sum_i x_i\otimes y_i$ with $x_i\in L^{1/p^e}$ and $y_i\in A^{1/p^{2e}}$. Then, on one hand, if we follow the top horizontal map and then the right vertical map, we will have
\[x\otimes z\mapsto (\sum_i x_i\otimes y_i)\otimes \alpha'(z)\mapsto \sum_i x_i\otimes y_i\alpha'(z)\mapsto \sum_i x_i \alpha(y_i\alpha'(z))= \sum_i x_i \otimes y_i\alpha(\alpha'(z)),\]
where the last equality holds since $\alpha$ is $A^{1/p^{ne}}$-linear and hence $A^{1/p^{2e}}$-linear.

On the other hand, if we follow the other path ({\it i.e.} the left vertical map first and then the bottom horizontal map), we will have
\[x\otimes z\mapsto x\otimes \alpha'(z)\mapsto x\otimes \alpha(\alpha'(z))\mapsto (\sum_i x_i\otimes y_i)\otimes \alpha(\alpha'(z))\mapsto \sum_i x_i \otimes y_i\alpha(\alpha'(z)).\]
This proves that the square is indeed commutative when $m=2$ and concludes the proof of our Lemma.
\end{proof}

%\begin{remark}
%In the proof of \autoref{lem.FactorizationOfSigmaImages} above, we could have weakened the hypothesis that $V$ was regular to the hypothesis that $V$ is quasi-Gorenstein and $F$-pure.
%\end{remark}

\begin{theorem}
\label{thm.RelativeSigmaVsAbsoluteOverOpen}
With notation as in \autoref{conv.BaseIsNice} and below, choose $n > n_0 = n_{\sigma(\phi), k(V)}$.  Then there exists a dense open set $U \subseteq V \cong V^e$ with $W = f^{-1}(U) \subseteq X$ such that
\[
\sigma(X_{V^{1/p^{ne}}}, \gamma)|_W = \sigma_n(X/V, \phi)|_W.
\]
Furthermore, shrinking $U$ further if necessary we can require for all perfect points $u \in U$, that
\[
\sigma(X_{V^{1/p^{ne}}}, \gamma)\cdot \O_{X_u} = \sigma_n(X_u/u, \phi_u) =  \sigma(X_u, \psi_u)
\]
where $\psi_u$ is $\phi_u$ viewed as an absolute $p^{-e}$-linear map.
\end{theorem}

\begin{proof}
The second statement is immediate from the first by \autoref{prop.RelativeNonFPureIdealsAndBaseChange}, so we need only prove the first statement.  By the diagram in \autoref{lem.FactorizationOfSigmaImages} and applying \autoref{prop.StabilizingSigma}, we see that
\[
\Image(\gamma^{n})|_W = \ba_{n,n}|_W = \ba_{n-1,n}|_W = \Image(\gamma^{n-1})|_W
\]
for all $n > n_0$.  Hence $\Image(\gamma^n)|_W = \sigma(X_{V^{1/p^{ne}}}, \gamma)|_W$ and the result follows.
\end{proof}

%We now tackle the points outside of $U$ at least if $V$ is regular.  First however we need two lemmas explaining how $\sigma(X, \gamma)$ behaves when modding out by a regular element and in the case that $V$ is a field, under field extensions.

%\todo{{\bf Wenliang:}  This is in flux.}

%\subsubsection*{Base change for $\gamma$}

\begin{convention}
\label{conv.BaseIsVeryNice}
Assuming everything in \autoref{conv.BaseIsNice}, we additionally assume that $\phi$ is relatively divisorial.
\end{convention}

Before we proceed, we prove a lemma about interpreting maps as divisors in the relative vs absolute setting.

\begin{lemma}
\label{lem.ChoiceOfDeltaPhiEqualsDeltaGamma}
With notation as in \autoref{conv.BaseIsVeryNice}, if $\Delta_{\phi}$ is the divisor corresponding to $\phi$ as in \autoref{def:D_phi_Delta_phi}.  Then $\Delta_{\gamma}$, the divisor on $X_{V^{ne}}$ corresponding to $\gamma$ as in \cite[Section 4]{BlickleSchwedeSurveyPMinusE}, is equal to $h^* \Delta$ where $h : X_{V^{ne}} = X \times_V V^{ne} \to X$ is the projection.
\end{lemma}
\begin{proof}
Working off a set of relative codimension $2$, and then working locally on $X$ and $V$, we can assume that $L$ and $M$ are trivial, and further assume that $\phi(\blank) = \Phi(z^{1/p^{e}} \cdot \blank)$ where $\Phi$ generates $\sHom_{R_{A^{1/p^e}}}(R^{1/p^e}, R_{A^{1/p^e}})$ as an $R^{1/p^e}$-module.  Then $\Delta_{\phi}$ is the divisor corresponding to ${1 \over p^e - 1} \Div_X(z)$.  On the other hand, since the map $\Gamma$ corresponding to $\Phi$ is a composition of generating maps, $\Gamma$ generates $\sHom_{R_{A^{1/p^{ne}}}}( R^{1/p^{e}}_{A^{1/p^{(n+1)e}}}, R_{A^{1/p^{ne}}})$, see for example \cite[Appendix F]{KunzKahlerDifferentials}, and so corresponds to the zero divisor.  Furthermore, we see that $\gamma(\blank) = \Gamma({\bar{z}}^{1/p^e} \cdot \blank)$ where $\bar{z}$ is the element corresponding to $z \tensor 1 \in R_{A^{1/p^{ne}}}$.  Hence $\Delta_\gamma = {1 \over p^e - 1} \Div_{X_{V^{ne}}}(\bar{z})$ and the claim follows.
\end{proof}

%\todo{{\bf Wenliang:} Hi Karl, I separate the commutation with completion as a lemma. I will look at \autoref{splitting generator} later. Meanwhile, please feel free to make any changes.}

We now need a Lemma which says roughly that in an $F$-pure $F$-finite ring $A$ with maximal ideal $\bm$, then at least locally, one can always choose a $\alpha : A^{1/p^{e}} \to A$ such that the closed point $V(\bm)$ is an $F$-pure center of $\Delta_{\alpha}$. To this end, we introduce the following.
\[M_{e,A}:=\{\phi\in \Hom_A(A^{1/p^e},A)| \phi(\bm^{1/p^e})\subseteq \bm\}\]
and
\[N_{e,A}:=\{\phi\in \Hom_A(A^{1/p^e},A)| \phi(A^{1/p^e})\subseteq \bm\}.\]
It is straightforward to check that both $M_{e,A}$ and $N_{e,A}$ are $A^{1/p^e}$-submodules of $\Hom_A(A^{1/p^e},A)$ and that $N_{e,A}\subseteq M_{e,A}$.  We now observe that:

\begin{lemma}
\label{lem.PreFindCompatibleMInFPure}
For any $F$-finite reduced G1 and S2 ring $A$ with maximal ideal $\bm$, the formation of $M_{e,A}$ and $N_{e,A}$ commute with localization and completion completion, {\it i.e.} if $W$ is a multiplicative system and $\hat\blank$ denotes completion along $\bm$ then
\[
\begin{array}{cc}
W^{-1} M_{e,A}\cong (W^{-1}M)_{e,W^{-1}{A}}\ {\rm and}\ W^{-1} N_{e,A} \cong (W^{-1} N)_{e,W^{-1}{A}} & {\rm and}\\
M_{e,A}\otimes_A\hat{A}\cong M_{e,\hat{A}}\ {\rm and}\ N_{e,A}\otimes_A\hat{A}\cong N_{e,\hat{A}} &
\end{array}
\]
\end{lemma}
The following proof was suggested to us by Manuel Blickle and Kevin Tucker.  We believe there are more general proofs that work without the G1 and S2 hypothesis (but this proof is short).
\begin{proof}
We have a natural injective map $\Hom_{A}(\bm^{1/p^e},\bm) \hookrightarrow \Hom_A(\bm^{1/p^e}, A)$, but $\Hom_A(\bm^{1/p^e}, A) \cong \Hom_A(A^{1/p^e}, A)$ since both modules are S2 (since a reflexive module is S2 in a G1 and S2 ring, \cite{HartshorneGeneralizedDivisorsOnGorensteinSchemes}).
Thus we have $\Hom_{A}(\bm^{1/p^e},\bm) \hookrightarrow \Hom_A(A^{1/p^e}, A)$, the image of which is $M_{e, A}$.  Therefore $M_{e, A}$ clearly commutes with localization and completion since the formation of both $\Hom$ sets commutes with localization and completion.   Similarly, the formation of $N_{e,A}$ commutes with localization and completion.
\end{proof}

\begin{lemma}
\label{splitting generator}
Suppose that $A$ is an $F$-pure $F$-finite G1 and S2 ring with maximal ideal $\bm$.  Then there exists an $A$-linear map $\alpha_A : A^{1/p^e} \to A$ such that $\alpha_A$ is surjective and $\alpha_A(\bm^{1/p^e}) \subseteq \bm$.
\end{lemma}
\begin{proof}
We first assume that $A$ is a local ring.  In this case, it suffices to show that $N_{e,A}\neq M_{e,A}$. By \autoref{lem.PreFindCompatibleMInFPure}, we may assume that $A$ is complete since $\hat{A}$ is faithfully flat over $A$. By the Cohen Structure Theorem, there is a ring of formal power series $S=k\llbracket x_1,\dots,x_n \rrbracket$ over a coefficient field $k$ of $A$ with surjection $S\onto A$. Write $A=S/I$. Recall that $\Hom_S(S^{1/p^e},S)$ is generated (as an $S^{1/p^e}$-module) by $g:S^{1/p^e}\to S$ that sends $(x_1\cdots x_n)^{p^e-1 \over p^e}$ to 1 and other basis monomials (including 1) to zero. By Fedder's lemma \cite[Lemma 1.6]{FedderFPureRat}, each element of $\Hom_A(A^{1/p^e},A)$ has the form $f(u^{1/p^e}-)$ with $u\in (I^{[p^e]}:I)$. Since $A$ is $F$-pure, $(I^{[p^e]}:I)\nsubset \bn^{[p^e]}$ where $\bn=(x_1,\dots,x_n)$. Let $u$ be an element of $(I^{[p^e]}:I)\backslash \bn^{[p^e]}$. Then $u$ must have a monomial term $cx^{a_1}_1\cdots x^{a_n}_n$ with $c\in k$ and $a_i<p^e$ for each $i$. Choose
such a monomial appearing in $u$ with the least degree and still denote it by $cx^{a_1}_1\cdots x^{a_n}_n$. Then it is clear that $(x^{p^e-1-a_1}_1\cdots x^{p^e-1-a_n}_n)u-c(x_1\cdots x_n)^{p^e-1} \in \bn^{[p^e]}$. Now define $\alpha:A^{1/p^e}\to A$ by $\alpha(-)=g((x^{p^e-1-a_1}_1\cdots x^{p^e-1-a_n}_nu)^{1/p^e}-)$. Then $\alpha$ is in $M_{e,A}\backslash N_{e,A}$.  This completes the local case by \autoref{lem.PreFindCompatibleMInFPure}.

For the non-local case, we have the map $M_{e,A} \to A$ which is evaluation at 1.  Since the formation of $M_{e,A}$ commutes with localization, this map is surjective if and only if it is surjective locally.  The above work proves the result after localizing at $\bm$.  However, clearly $(M_{e,A})_{\bn} = \big(\Hom_A(A^{1/p^e}, A)\big)_{\bn}$ for prime ideals not equal to $\bm$.  The result follows.
\end{proof}

For the moment, we work sufficiently locally so that $X = \Spec R$ and $V = \Spec A$ are affine and that $L$ and $M$ are isomorphic to $R$ and $A$ respectively.
Continue to assume that $A$ is $F$-pure and quasi-Gorenstein, and fix a point $s \in \Spec A$.  Shrinking $V$ if necessary, choose a map $\alpha_s \in \Hom_A(A^{1/p^e}, A)$ which is surjective and which satisfies $\alpha(\bm_s^{1/p^e}) \subseteq \bm_{s}$ (whose existence is guaranteed by \autoref{splitting generator}).  We write $\alpha_s(\blank) = \Phi_A^e(z^{1/p^e} \cdot \blank)$ for some $z \in A$.  Note we have an induced map $\overline{\alpha_s} : (A/{\bm_s})^{1/p^e} \to A/\bm_s$.

We thus induce the following map $\beta : (R_{A^{1/p^{ne}}})^{1/p^e} \to R_{A^{1/p^{ne}}}$ defined by the rule $\beta(\blank) = \gamma(z^{1/p^{(n+1)e}} \cdot \blank)$.  Certainly using the notation of \autoref{eq.defnGamma},
\[
\begin{array}{rl}
& \beta(\bm_s^{1/p^{(n+1)e}} \cdot R^{1/p^e}_{A^{1/p^{(n+1)e}}}) \\
= & \gamma(z^{1/p^{(n+1)e}} \bm_s^{1/p^{(n+1)e}} \cdot R^{1/p^e}_{A^{1/p^{(n+1)e}}})\\
= & \Big(\phi' \circ \big( (\Phi_A^e)^{1/p^{ne}} \otimes_{A^{1/p^{ne}}} R^{1/p^e} \big)\Big) \Big( z^{1/p^{(n+1)e}} \bm_s^{1/p^{(n+1)e}} \cdot R^{1/p^e}_{A^{1/p^{(n+1)e}}}\Big)\\
= & \Big(\phi' \circ \big( (\alpha_s)^{1/p^{ne}} \otimes_{A^{1/p^{ne}}} R^{1/p^e} \big)\Big) \Big( \bm_s^{1/p^{(n+1)e}} \cdot R^{1/p^e}_{A^{1/p^{(n+1)e}}}\Big)\\
\subseteq & \phi'(\bm_s^{1/p^{ne}} \cdot R^{1/p^e}_{A^{1/p^{ne}}})\\
\subseteq & \bm_s^{1/p^{ne}} R_{A^{1/p^{ne}}}.
\end{array}
\]
And hence we induce a map
\[
\overline{\beta} : R^{1/p^e}_{(A/\bm_s)^{1/p^{(n+1)e}}} \to R_{(A/\bm_s)^{1/p^{ne}}}.
\]
\begin{lemma}
\label{lem.FAdjunctionDeltaOnFibers}
With notation as above, $\Delta_{\overline{\beta}} = \Delta_{\gamma}|_{X_{s^{ne}}}$ where the divisors are induced as in \cite[Section 4]{BlickleSchwedeSurveyPMinusE}.
\end{lemma}
\begin{proof}
Since $\phi$ is relatively divisorial, by working in relative codimension 1, it suffices to show the result in the case that $\Delta_{\gamma}$ (and so $\Delta_{\phi}$) is trivial (since locally $\phi$ or $\psi$ is a generator pre-multiplied by the element which determines $\Delta_{\phi}$ or $\Delta_{\psi}$).  But now consider the composition constructing ${\beta}$.
\[
R^{1/p^e}_{A^{1/p^{(n+1)e}}} \xrightarrow{\cdot z^{1/p^{(n+1)e}} } R^{1/p^e}_{A^{1/p^{(n+1)e}}} \xrightarrow{\Phi'} R^{1/p^e}_{A^{1/p^{ne}}} \xrightarrow{\phi'} R_{A^{1/p^{ne}}}
\]
The composition of the first two maps sends $\bm^{1/p^{(n+1)e}} \cdot R^{1/p^e}_{A^{1/p^{(n+1)e}}}$ to $\bm^{1/p^{ne}} \cdot R^{1/p^e}_{A^{1/p^{ne}}}$ and modding out by these ideals induces a generating map for $\Hom_{R^{1/p^e}_{(A/\bm)^{1/p^{ne}}}}(R^{1/p^e}_{(A/\bm)^{1/p^{(n+1)e}}}, R^{1/p^e}_{(A/\bm)^{1/p^{ne}}})$ via the fact that $\alpha_s$ is surjective and so the induced map $(A/\bm)^{1/p^{(n+1)e}} \to (A/\bm)^{1/p^{ne}}$ is a generating map as well.  On the other hand, if $\phi$ is a generating map, so is $\phi' \tensor_{A^{1/p^{ne}}} (A/\bm)^{1/p^{ne}}$ by \autoref{lem:zero_divisor_stay_zero_divisor}.  Since the composition of two generating maps is a generating map, we are done.
\end{proof}

We now explain how the absolute $\sigma$ behaves when restricting to fibers.  First we need two lemmas.

\begin{lemma}
\label{lem.EasyContainmentOfNonFPureViaAdjunction}
With notation as above, and $v \in V$ a closed point.
\[
\sigma(X_{V^{ne}}, \Delta_{\gamma}) \cdot \O_{X_{v^{ne}}} = \sigma(X_{V^{ne}}, \gamma) \cdot \O_{X_{v^{ne}}} \supseteq \sigma(X_{v^{ne}}, \overline{\beta}) = \sigma(X_{v^{ne}}, \Delta_{\overline{\beta}}) = \sigma(X_{v^{ne}}, \Delta_{\gamma}|_{X_{v^{ne}}}).
\]
\end{lemma}
\begin{proof}
Obviously $\sigma(X_{V^{ne}}, \gamma) \supseteq \sigma(X_{V^{ne}}, \beta)$.  On the other hand, it follows easily that $\sigma(X_{V^{ne}}, \beta)$ restricts to $\sigma(X_{v^{ne}}, \overline{\beta})$ by $F$-adjunction, see for instance \cite{FujinoSchwedeTakagiSupplements}.
\end{proof}

The next lemma is a generalization of the fact that for an ideal in a regular ring, $I \subseteq (I^{[1/p]})^{[p]}$ using the ${\bullet}^{[1/p]}$ notation from \cite{BlickleMustataSmithDiscretenessAndRationalityOfFThresholds}.

\begin{lemma}
\label{lem:pushforward_pullback}
Let $B$ be an $F$-finite regular ring such that $B^{1/p}$ is a free $B$-module and that \mbox{$\Hom_B(B^{1/p}, B) \cong B^{1/p}$} (for example this happens if $B$ is local).  Further suppose that $Q$ is a $B^{1/p^a}$ submodule of $B^{1/p^a} \otimes_B P$ for some $B$-module $P$,
\[
Q \subseteq B^{1/p^a} \otimes_B P.
\]
Then $Q \subseteq B^{1/p^{a}} \otimes_B \big( (\theta \otimes \id_P )(Q)\big)$ where $\theta \in \Hom_B(B^{1/p^{a}},B)$ is the generator.
\end{lemma}

\begin{proof}
The strategy is similar to \cite[Proposition 2.5]{BlickleMustataSmithDiscretenessAndRationalityOfFThresholds}.
By \cite[Theorem 2.1]{KunzCharacterizationsOfRegularLocalRings}, any element of $Q$ can be written as a finite sum $\sum \lambda_i \otimes m_i$, where $\{ \lambda_i \}$ form a basis for $B^{1/p^{a}}$ over $B$ and $m_i \in P$.  Now since  $\theta$ is a local generating map of $\Hom_B(B^{1/p^{a}},B)$, the projection maps onto the $\lambda_i$ are multiples of $\theta$.  In other words, for each $\lambda_i$ there exists a $u_i \in B^{1/p^{a}}$ such that both $\theta(u_i \lambda_i) = 1 \in R$ and $\theta(u_i \lambda_j) = 0 \in R$ for $j \neq i$.

We observe that each $m_j$ is then in $(\theta \otimes \id_P )(Q)$ since $\theta(u_j  \sum \lambda_i \otimes m_i) = m_j$.  It follows then that $\sum \lambda_i \otimes m_i \in B^{1/p^{a}} \otimes_B \big( (\theta \otimes \id_P )(Q)\big)$.
\end{proof}

\begin{lemma}
\label{lem.EasyContainmentOfNonFPureViaFieldExtension}
Suppose that $A = k$ is an $F$-finite field and $K \supseteq k$ is a field extension of $k$ such that $K$ is perfect.  Choose $X = \Spec R \to V = \Spec A = \Spec k$ a flat map of finite type with $X \to V$ possessing geometrically reduced, G1 and S2 fibers.  Additionally suppose that $\gamma : R^{1/p^e} \to R$ is any $R$-linear map which is a local generator of $\Hom_R(R^{1/p^e}, R)$ at the generic points of the codimension-1 components of the non-smooth locus of $X$.  Then, setting $\Delta_{\gamma}$ as in \cite[Section 4]{BlickleSchwedeSurveyPMinusE}.
\[
\sigma(X, \Delta_\gamma) \otimes_{k} K \supseteq \sigma(X_K, (\Delta_\gamma) \times_k K)
\]
\end{lemma}
\begin{proof}
First choose $n > 0$ such that $\sigma(X, \Delta_{\gamma}) = \gamma^n(R^{1/p^{ne}})$.  Since $\Hom_k(k^{1/p^{ne}}, k)$ is a free $k^{1/p^{ne}}$-module of rank 1, by \cite[Appendix F]{KunzKahlerDifferentials} we can factor $\gamma^n$ as
\[
\xymatrix{R^{1/p^{ne}} \ar@/_2pc/[rr]_-{\gamma^n} \ar[r]^-{\beta} & R_{k^{1/p^{ne}}} \ar[r]^-{\theta} & R_k}
\]
with $\theta$ a local generator of $\sHom_{R_k}(R_{k^{1/p^{ne}}}, R_k)$
and so that $\Delta_{\beta}$ coincides with $\Delta_{\gamma}$.  To see this last point, notice that if $\gamma$ is a local generator of $\Hom_R(R^{1/p^e}, R)$, then so is $\gamma^n$ of $\Hom_R(R^{1/p^{ne}}, R)$ and hence $\beta$ is also a local generator of $\Hom_{R_{A^{1/p^{ne}}}}(R^{1/p^{ne}}, R_{A^{1/p^{ne}}})$.  More generally then, if $\gamma^n$ is a local generator pre-multiplied by some $g^{1/p^{ne}} \in R^{1/p^{ne}}$ (working in codimension 1) then $\beta$ is also a generator times $g^{1/p^{ne}}$ (possibly times a unit depending on our choice of $\theta$).  This proves that $\Delta_{\beta} = \Delta_{\gamma}$ as asserted.

We now observe that $\gamma^n(R^{1/p^{ne}}) \tensor_k k^{1/p^{ne}} = \theta(\beta(R^{1/p^{ne}}) \tensor_k k^{1/p^{ne}} \supseteq \beta(R^{1/p^{ne}})$ by \autoref{lem:pushforward_pullback} (here $B = k$ and $P = \beta(R^{1/p^{ne}})$).
%\end{claim}
%\begin{proof}[Proof of claim]
%Any element in the image of $\beta(R^{1/p^{ne}})$ can be written as a sum of elements of $R$ times elements of a basis $\{ \lambda_i \}$ for $k^{1/p^{ne}}$ over $k$, in other words $\beta(y^{1/p^{ne}}) = \sum \lambda_i r_i.$  Using the same idea as \cite[Proposition 2.5]{BlickleMustataSmithDiscretenessAndRationalityOfFThresholds}, for each $\lambda_i$ we can find an element $u_i$ of $k^{1/p^{ne}}$ over $k$ such that multiplying $\theta$ by $u_i$ gives a map which projects onto $\lambda_i$ and sends the other $\lambda_j$ to zero.  This implies that each $r_i$ is in fact in the image of $\gamma^n$.  Thus since each $r_i \in \gamma^n(R^{1/p^{ne}})$, we have proved the claim.
%\end{proof}
%
After embedding $k^{1/p^{ne}} \subseteq K$ we tensor $\beta$ by $K$ to obtain a map
\[
\beta_K : R^{1/p^{ne}} \tensor_{k^{1/p^{ne}}} K \to R_{k^{1/p^{ne}}} \tensor_{k^{1/p^{ne}}} K.
\]
We see immediately that $\Delta_{\beta_K} = (\Delta_\gamma) \times_k K$ (again by an argument about local generators).  Then, by \autoref{lem:pushforward_pullback},
\[
\sigma(X, \gamma) \tensor_k K \cong \gamma^n(R^{1/p^{ne}}) \tensor_k k^{1/p^{ne}} \tensor_{k^{1/p^{ne}}} K \supseteq \beta(R^{1/p^{ne}}) \tensor_{k^{1/p^{ne}}} K \supseteq \sigma(X, \beta_K).
\]
This completes the proof.
\end{proof}

%We also need the following:

%\begin{lemma}
%Suppose $(A,\bm)$ is an $F$-finite regular local ring.  Then there exists a Frobenius splitting $\Psi_A : A^{1/p^e} \to A$ which generates $\Hom_{A}(A^{1/p^e}, A)$ as an %$A^{1/p^e}$-module.
%\end{lemma}
%\begin{proof}
%\end{proof}

\begin{theorem}
\label{cor.SigmaRestrictsToAllPoints}
With notation as in \autoref{conv.BaseIsVeryNice}, there exists an $N > 0$ such that for all $n > N$ we have $\sigma(X_{V^{1/p^{ne}}}, \Delta_{\gamma} )\cdot \O_{X_{s^{ne}}} = \sigma(X_{s^{ne}}, (\Delta_{\gamma}) |_{X_{s^{ne}}})$ for all perfect points $s \in V$.
\end{theorem}
\begin{proof}
The statement is local, so we may assume $X$ and $V$ are affine and that $L$ and $M$ are trivial just as above in \autoref{lem.FAdjunctionDeltaOnFibers}.
We know
\[
\sigma(X_{V^{1/p^{ne}}}, \gamma^n) \subseteq \Image(\gamma^n) = \sigma_n(X/V, \phi)
 \]
by \autoref{lem.FactorizationOfSigmaImages}.  On the other hand, by \autoref{thm.BaseChangeSigmaGeneral}, we have $\sigma_n(X/V, \phi)\cdot \O_{X_s} = \sigma(X_s, \phi_s)$.  Therefore, it is sufficient to show that $\sigma(X_{V^{1/p^{ne}}}, \gamma)\cdot \O_{X_{s^{ne}}} \supseteq \sigma(X_{s^{ne}}, \psi_{s^{ne}})$.

For this we may assume that $V$ is the spectrum of an $F$-Finite $F$-pure local ring and that $\Spec K = s \mapsto V$ has image the closed point $\Spec k = v \in V$.  %Say $v$ is defined by a regular system of parameters ${\bf x}$.
We first see that $\sigma(X_{V^{ne}}, \Delta_\gamma) \cdot \O_{X_{v^{ne}}} \supseteq \sigma(X_{v^{ne}}, \Delta_{\overline{\beta}}) = \sigma(X_{v^{ne}}, \Delta_\gamma|_{X_{v^{ne}}})$ by and using the notation of \autoref{lem.EasyContainmentOfNonFPureViaAdjunction}. On the other hand, by \autoref{lem.EasyContainmentOfNonFPureViaFieldExtension}
\[
\sigma(X_{v^{ne}}, \Delta_\gamma|_{X_{v^{ne}}}) \cdot \O_{X_{s^{ne}}} = \sigma(X_{v^{ne}}, \Delta_{\overline{\beta}}) \cdot \O_{X_{s^{ne}}} \supseteq \sigma(X_{s^{ne}}, \Delta_{\overline{\beta}} \times_{v^{ne}} s^{ne}) = \sigma(X_{s^{ne}}, \Delta_{\gamma}|_{X_{s^{ne}}}).
\]
This completes the proof.
%
%First, certainly $\sigma(X_{V^{1/p^{ne}}}, \gamma) \supseteq \sigma(X_{V^{1/p^{ne}}}, \delta)$ since $\delta$ is a multiple is a $\gamma$.  Let $v$ denote the image of $s$ in $V$.  Then $\sigma(X_{V^{1/p^{ne}}}, \delta) \cdot \O_{X_{v^{ne}}} = \sigma(X_{v^{1/p^{ne}}}, \overline{\delta})$.
\end{proof}
Using the same method as \autoref{lem.EasyContainmentOfNonFPureViaAdjunction}, we also obtain the following result which may also be of independent interest.

\begin{proposition}
\label{prop.SigmaUnderBaseChange}
Assume that $f : X \to V$ is a flat finite type reduced G1 and S2 morphism and that $V$ is regular.  Additionally assume that $\gamma : R^{1/p^e} \to R$ is an $R$-linear map which is a local generator at the generic points of the codimension-1 components of the non-smooth locus of $f$ for \emph{each} fiber of $f$.  Then for any $a > 0$ we have
\[
\sigma(X, \Delta_\gamma) \otimes_A A^{1/p^{a}} \supseteq \sigma(X_{V^{a}}, \Delta_{\gamma} \times_V V^{a}).
\]
\end{proposition}
\begin{proof}
The statement is local over the base and so we may assume that $K_V \sim 0$ and that $A^{1/p^e}$ is a free $A$-module since $V$ is regular \cite{KunzCharacterizationsOfRegularLocalRings}.  Note then that $K_{X/V} \cong K_X$.
Choose $ne \geq a > 0$ such that $\sigma(X, \Delta_{\gamma}) = \gamma^n(R^{1/p^{ne}})$.  Since $\Hom_A(A^{1/p^{a}}, A)$ is a free $A^{1/p^{ne}}$-module of rank 1, by \cite[Appendix F]{KunzKahlerDifferentials} we can factor $\gamma^n$ as
\[
\xymatrix{R^{1/p^{ne}} \ar@/_2pc/[rr]_-{\gamma^n} \ar[r]^-{\beta} & R_{A^{1/p^{a}}} \ar[r]^-{\theta} & R}
\]
where $\theta$ is a local generator of $\sHom_{R}(R_{A^{1/p^{a}}}, R)$.
Note $\beta \in \Hom_{R_{A^{1/p^{a}}}}(R^{1/p^{ne}}, R_{A^{1/p^{a}}})$ which is identified with
\[
\begin{array}{rl}
& \sHom_{R_{A^{1/p^{a}}}}\big(R^{1/p^{ne}} \tensor_{R_{A^{1/p^a}}} (A^{1/p^a} \tensor_A \omega_{R/A}), (A^{1/p^a} \tensor_A \omega_{R/A})\big)\\
\cong & \sHom_{R_{A^{1/p^{a}}}}\big( (\O_X(p^{ne}K_{X/V}))^{1/p^{ne}}, \omega_{R_{A^{1/p^a}}/A^{1/p^a}}\big) \\
\cong & \sHom_{R_{A^{1/p^{a}}}}\big( (\O_X(p^{ne}K_{X}))^{1/p^{ne}}, \omega_{R_{A^{1/p^a}}}\big) \\
\cong & \sHom_{R^{1/p^{ne}}}\big( (\O_X(p^{ne}K_{X}))^{1/p^{ne}}, \omega_{R^{1/p^{ne}}}\big)\\
\cong & (\O_X( (1-p^{ne})K_X))^{1/p^{ne}}.
\end{array}
\]
Note that since the base is regular, all these sheaves are automatically reflexive and so there is no need to double dualize as we did before in \autoref{def:D_phi_Delta_phi}.
 By dividing by $(p^{ne} - 1)$
 we see that $\beta$ coincides with a divisor $\Delta_{\beta}$ such that $(1-p^{ne})(K_X + \Delta_{\beta}) \sim 0$.
Furthermore, it is easy to see that $\Delta_{\beta}$ coincides with $\Delta_{\gamma}$ just as in \autoref{lem.EasyContainmentOfNonFPureViaAdjunction} since $\theta$ is a local generator.

By composing $\beta$ with a local generating (and surjective) map $R^{1/p^{ne}} \tensor_{A^{1/p^{ne}}} A^{1/p^{ne+a}} \xrightarrow{\ldots \tensor (\Phi_{A}^{a})^{1/p^{ne}}} R^{1/p^{ne}} $ we then obtain a map $\gamma' : (R_{A^{1/p^{a}}})^{1/p^{ne}} \cong  R^{1/p^{ne}} \tensor_{A^{1/p^{ne}}} A^{1/p^{ne+a}} \to R_{A^{1/p^{a}}}$ satisfying the condition $\Delta_{\gamma'} = \Delta_{\gamma} \times_V V^{a}$.  We also notice that $\gamma'$ has the same image as $\beta$.

Now, we observe that $\gamma^n(R^{1/p^{ne}}) \tensor_A A^{1/p^{a}} = \theta(\beta(R^{1/p^{ne}})) \tensor_A A^{1/p^{a}} \supseteq \beta(R^{1/p^{ne}})$ by \autoref{lem:pushforward_pullback} setting $B = A$ and $P = \beta(R^{1/p^{ne}})$.
%\begin{claim}
%$\gamma^n(R^{1/p^{ne}}) \tensor_A A^{1/p^{a}} \supseteq \beta(R^{1/p^{ne}})$.
%\end{claim}
%\begin{proof}[Proof of claim]
%Any element in the image of $\beta(R^{1/p^{ne}})$ can be written as a sum of elements of $R$ times elements of a basis $\{ \lambda_i \}$ for $A^{1/p^{a}}$ over $A$, in other words $\beta(y^{1/p^{ne}}) = \sum \lambda_i r_i$, $r_i \in R$.  Now since $R_{A^{1/p^a}}$ is a free $R$-module with basis $\{\lambda_i\}$ and $\theta$ is a local generating map, the projection maps onto the $\lambda_i$ are multiples of $\theta$.  In other words, for each $\lambda_i$ there exists a $u_i \in R_A^{1/p^{a}}$ such that both $\theta(u_i^{1/p^a} \lambda_i) = 1 \in R$ and $\theta(u_i^{1/p^a} \lambda_j) = 0 \in R$ for $j \neq i$.
%
%We observe that each $r_j$ is then in $\theta(\langle \sum \lambda_i r_i \rangle_{R_{A^{1/p^{a}}}})$ since $\theta(u_j  \sum \lambda_i r_i) = r_j$.  It follows that each $r_j \in \gamma^n(R^{1/p^{ne}})$.  But then the claim follows immediately.
%\end{proof}
%We now return to the proof of the proposition.  But this is easy since
The remainder of the proof is easy since
\[
\begin{array}{rl}
& \sigma(X, \Delta_\gamma) \otimes_A A^{1/p^{a}} \\
= & \gamma^n(R^{1/p^{ne}}) \tensor_A A^{1/p^{a}} \\
\supseteq & \beta(R^{1/p^{ne}}) \\
= & \gamma'\big((R_{A^{1/p^{a}}})^{1/p^{ne}}\big) \\
\supseteq & \sigma(X_{V^a}, \Delta_{\gamma'}) \\
= & \sigma(X_{V^a},  \Delta_{\gamma} \times_V V^{a}).
\end{array}
\]
\end{proof}

\mysubsection{Application to sharply $F$-pure singularities and HSL numbers in families}
We observe that \autoref{thm.UniversalNForAllClosedFibers} has an immediate application.  Recall that if $\gamma : L^{1/p^e} \to R$ is an $R$-linear map, then the \emph{HSL number} is the first integer $n$ such that $\Image(\gamma^n) = \Image(\gamma^{n+1}) = \sigma(R, \gamma)$, \cf \cite{SharpOnTheHSLThm}.
\begin{corollary}[Uniform behavior of HSL numbers]
Given a flat family $f : X \to V$ over a excellent integral scheme $V$ with a dualizing complex of characteristic $p > 0$ and $\phi : L^{1/p^e} \to \O_{X_{V^e}}$ as before, there exists an integer $N > 0$ such that gives an upper bound on the HSL number of $(X_s, \phi_s : L_s^{1/p^e} \to \O_{X_s})$ for every perfect point $s \in V$.
\end{corollary}
\begin{proof}
The statement immediately follows from \autoref{thm.BaseChangeSigmaGeneral} since clearly $\sigma_m(X_s/s, \phi_s)$ agrees with $\sigma(X_s, \gamma_s = \phi_s)$ for $m \gg n$.
\end{proof}

Now we study the deformation of sharp $F$-purity, a question which has been studied before in \cite{ShimomotoZhangOnTheLocalizationProblem,HashimotoFPureHoms}. We believe that the hypothesis that $f$ is proper is necessary in \autoref{thm.OpennessOfSharpFPurity}.

\begin{definition}
Recall that a pair $(X, \psi : L^{1/p^e} \to \O_X = R)$ is called \emph{sharply $F$-pure} if $\sigma(X, \psi) = \O_X$.  Given $(X/S, \phi : L^{1/p^e} \to R_{A^{1/p^e}})$, we define $(X/S, \phi)$ to be \emph{relatively sharply $F$-pure} if $\sigma_{n}(X/S, \phi) = R_{A^{1/p^{ne}}}$ for some $n > 0$, \cf \cite{HashimotoFPureHoms}.
\end{definition}

\begin{lemma}
If $\sigma_n(X/S, \phi) = R_{A^{1/p^{ne}}}$ for some $n > 0$, then the same holds for all $n > 0$.
\end{lemma}
\begin{proof}
Suppose $\sigma_n(X/S, \phi) = R_{A^{1/p^{ne}}}$ for some $n > 0$.  Choose $m < n$ to start.  Then we know $\ba_{m,n} \supseteq \ba_{n,n} = \sigma_n(X/S, \phi)$ and so $\ba_{m,n} = R_{A^{1/p^{ne}}}$.  But $\ba_{m,n} = \Image\big(\ba_m \tensor_{A^{1/p^m}} A^{1/p^n} \to R_{A^{1/p^n}} \big)$ which is just the extension of
$\sigma_m(X/S, \phi) \subseteq R_{A^{1/p^{me}}}$
to $R_{A^{1/p^{ne}}}$.  Additionally, $R_{A^{1/p^{me}}} \subseteq R_{A^{1/p^{ne}}}$ is an integral extension and hence we must have $\sigma_m(X/S, \phi) = R_{A^{1/p^{me}}}$ as well.

We finish the proof by showing $\sigma_{2n}(X/S, \phi) = R_{A^{1/p^{2ne}}}$.  This follows quickly since $\sigma_{2n}(X/S, \phi)$ factors as a composition of two surjective maps (as $\tensor$ is right exact) as in \autoref{eq.Twist1} and \autoref{eq.Twist1}.
\end{proof}

\begin{remark}
\label{rem.RelativelyFPureOverPtIsGeometric}
If $S$ is a point, then $(X/S, \phi)$ being relatively sharply $F$-pure is equivalent to $(X/S, \phi)$ being geometrically sharply $F$-pure ({\it i.e.}, that $(X_t, \gamma_t)$ is sharply $F$-pure for every geometric point $t \to S$ where $\gamma_t$ is induced from $\phi_t$ as in \autoref{subsec.IteratedNonFPureIdealsVsAbsolute}).  To see this, certainly observe that if $(X/S, \phi)$ is relatively sharply $F$-pure then so is any base change $(X_t/t, \phi_t)$.  But then $(X_t, \gamma)$ is sharply $F$-pure by \autoref{thm.RelativeSigmaVsAbsoluteOverOpen}.  Conversely, if $(X_t, \gamma_t)$ is sharply $F$-pure, then certainly $(X_t/t, \phi_t)$ is relatively sharply $F$-pure by \autoref{lem.FactorizationOfSigmaImages}.  But then so is $(X/S, \phi)$ near that $t$ by Nakayama.  %since $t \to S$ is faithfully flat.
\end{remark}

\begin{theorem}[Openness of sharp $F$-purity]
\label{thm.OpennessOfSharpFPurity}
With notation as before, assume that $f : X \to V$ is proper.  Assume that $s \in V$ is a point and that $(X_s/s, \phi_s)$ is relatively sharply $F$-pure (in other words, geometrically $F$-pure).  Then there exists a dense open set $U \subseteq V$ containing $s$ such that $(X_u/u, \phi_u)$ is relatively sharply $F$-pure for all $u \in U$ (in particular, $(X_u, \phi_u)$ is sharply $F$-pure for all perfect points $u \in U$).
\end{theorem}
\begin{proof}
Choose $n \gg 0$ such that $\sigma_n(X_s/s, \phi_s) = R_{\O_s^{1/p^{ne}}}$ since $(X_s/s, \phi_s)$ is sharply $F$-pure.  By \autoref{thm.BaseChangeSigmaGeneral}, we know that $\sigma_n(X/V, \phi) = R_{A^{1/p^{ne}}}$ in a neighborhood $W \subseteq X$ of $X_s$.  Let $Z = X \setminus W \subseteq X$ be the complement of that neighborhood.  Since $f$ is proper, $f(Z)$ is closed, and also does not contain $s$.  Set $U = V \setminus f(Z)$.  Then $\sigma_n(X_U/U, \phi_U) = R_{\O_U^{1/p^{ne}}}$.  It follows from \autoref{thm.BaseChangeSigmaGeneral} that all the fibers $(X_u/u, \phi_u)$ are relatively sharply $F$-pure for perfect $u \in U$ as desired.
\end{proof}

We now state the above theorem in the context of divisors.

\begin{corollary}[Openness of sharp $F$-purity for divisor pairs]
\label{cor.OpennessOfSharpFPurityForDivisors}
Suppose that $f : X \to V$ is a proper G1 and S2 morphism as before and now additionally suppose that $\Delta$ is a $\bQ$-divisor satisfying conditions (a)--(d) from \autoref{rem:relative_canonical}.  Additionally suppose that $s \in V$ is a perfect point and that $(X_s, \Delta|_{X_s})$ is sharply $F$-pure (or more generally, if $s$ is not perfect, that $(X_{s^e}, \Delta|_{X_{s^e}})$ is sharply $F$-pure for some/all $e > 0$).  Then there exists an open set $U \subseteq V$ containing $s$ such that for all $u \in U$ we have $(X_u, \Delta|_{X_u})$ is sharply $F$-pure  (respectively, we have $(X_{u^e}, \Delta|_{X_{u^e}})$ is sharply $F$-pure for some/all $e > 0$ and all $u \in U$).
\end{corollary}
\begin{proof}
Associate to $\Delta$ a relatively divisorial $\phi : L^{1/p^e} \to R_{A^{1/p^e}}$ as in \autoref{rem:relative_canonical}.  The result then follows immediately from \autoref{thm.OpennessOfSharpFPurity} and \autoref{cor.RestrictDeltaAndPhiToFiberIsOk}.
\end{proof}

\section{Relative test ideals}

In this section we define a notion of a relative test ideal.  We construct it using the following method.  We first find some ideal $I$ which when restricted to every geometric fiber is contained in the test ideal of that fiber.  We then sum up the image of $I^{1/p^{ne}}$ under $\phi^n$ (the same method is used to construct the test ideal in the absolute case).  This sounds simple enough, but it actually leads to a somewhat disappointing construction.  The problem is that the ideal we pick is in no sense canonical and we do not see a way to make it both canonical and stable under base change.  In the case that $\phi$ corresponds to a trivial divisor, we can make our choices canonical by setting $I$ to be the Jacobian ideal (this is done later when we discuss relative $F$-rationality in \autoref{sec.RelativeFRationalityFInjectivity}), but for more general $\phi$ we don't know of an analogous choice.

We fix the notation of the previous sections (in particular, we fix $\phi : L^{1/p^e} \to R_{A^{1/p^e}}$).  We begin by choosing an \emph{arbitrary ideal} $I \subseteq \O_X$.  After base change to $k(V^{\infty})$ as before, setting $I_{\infty} = I \cdot R_{\infty}$, we can form the following sum:
\[
\tau(R_{k(V^{\infty})}, \psi_{\infty} I_{\infty} ) :=  \sum_{i = 0}^{\infty} \psi_{\infty}^i\big((I_{\infty} \cdot L_{{k(V^{\infty})}}^{p^{ie} - 1 \over p^e - 1})^{1/p^{ie}}\big).
\]
If it happens that $I_{\infty}$ is contained in the test ideal of $( R_{k(V^{\infty})}, \phi_{\infty})$ and non-zero on any component of $X_{\infty}$, then we see that the absolute test ideal $\tau(R_{k(V^{\infty})}, \psi_{\infty}(I_{\infty}) )$ is equal to $\tau(R_{k(V^{\infty})}, \psi_{\infty})$ for example by \cite[Section 7]{SchwedeTuckerTestIdealSurvey}.

We can also describe this as follows, first essentially observed in \cite{KatzmanParameterTestIdealOfCMRings}.
Fix ideals
\begin{equation}
\label{eq.DefiningBBInfinity}
\begin{array}{rcll}
\bb^{\infty}_0 & := & I_{\infty} = \psi^0_{\infty}(I_{\infty}) \\
\bb^{\infty}_1 & := & \bb_0 + \psi_{\infty}\big((\bb_0^{\infty} \cdot L_{{k(V^{\infty})}})^{1/p^{e}}\big) = I_{\infty} + \psi_{\infty}\big((I_{\infty} \cdot L_{{k(V^{\infty})}})^{1/p^{e}}\big) \\\\
\bb^{\infty}_2 & := & \bb^{\infty}_1 + \psi_{\infty}\big((\bb^{\infty}_1 \cdot L_{{k(V^{\infty})}})^{1/p^{e}}\big) \\ &=& \bb_1^{\infty} + \psi_{\infty}\big((I_{\infty}+\psi_{\infty}\big((I_{\infty} \cdot L_{{k(V^{\infty})}})^{1/p^{e}}\big)  \cdot L_{{k(V^{\infty})}})^{1/p^{e}}\big)\\
%& = & I_{\infty} + \psi_{\infty}\big((I_{\infty} \cdot L_{{k(V^{\infty})}})^{1/p^{e}}\big) + \psi_{\infty}\big((I_{\infty} \cdot L_{{k(V^{\infty})}})^{1/p^{e}}\big) + \psi_{\infty}^2\big((I_{\infty} \cdot (L^{p^{2e} - 1 \over p^e - 1}))^{1/p^{2e}} \big)\\
& = & \sum_{i = 0}^2 \psi_{\infty}^i\big((I_{\infty} \cdot L_{{k(V^{\infty})}}^{p^{ie} - 1 \over p^e - 1})^{1/p^{ie}}\big)\\
\ldots & \ldots\\
\bb^{\infty}_n & := & \bb^{\infty}_{n-1} + \psi_{\infty}\big((\bb^{\infty}_{n-1} \cdot L_{{k(V^{\infty})}})^{1/p^{e}}\big) = \sum_{i = 0}^n \psi_{\infty}^i\big((I_{\infty} \cdot L_{{k(V^{\infty})}}^{p^{ie} - 1 \over p^e - 1})^{1/p^{ie}}\big)\\
\end{array}
\end{equation}
And we notice that this ascending chain stabilizes say at $t$.  Also note that the first time $\bb^{\infty}_t = \bb^{\infty}_{t+1}$, we then have $\bb^{\infty}_t = \bb^{\infty}_{t+1} = \bb^{\infty}_{t+2} = \ldots$.  We fix this integer $t$.

\begin{definition}
With notation as above, we define the integer $t$ to be the \emph{uniform integer for $\tau$ and $I$ over the generic point of $V$}, and in general, it will be denoted by $n_{\tau(\phi I), k(V)}$.  We notice that for any point $\eta \in V$, by base change we can replace $V$ by $\Spec k(\eta)$ and form a corresponding integer $n_{\tau(\phi_{\eta} I_{\eta}), k(\eta)}$.
\end{definition}

On the other hand, without the passing to $k(V^{\infty})$, we have the images
\[
\begin{array}{rcll}
\bb_1 & := & \Image\big(I \otimes_{A} A^{1/p^e} \to R_{A^{1/p^{e}}} \big) + \phi((I \cdot L)^{1/p^{e}}) \subseteq R_{A^{1/p^{e}}},\\
\bb_2 & := & \sum_{i = 0}^2 \Image\Big(\phi^i((I \cdot L^{p^{ie} - 1 \over p^e - 1})^{1 /p^{ie}}) \otimes_{A^{1/p^{ie}}} A^{1/p^{2e}} \to R_{A^{1/p^{2e}}}\Big)\subseteq R_{A^{1/p^{2e}}},\\
  \dots \\
\bb_n & := & \sum_{i = 0}^n \Image\Big(\phi^i((I \cdot L^{p^{ie} - 1 \over p^e - 1})^{1 /p^{ie}}) \otimes_{A^{1/p^{ie}}} A^{1/p^{ne}} \to R_{A^{1/p^{ne}}} \Big)\subseteq R_{A^{1/p^{ne}}},\\
  \dots
  \end{array}
\]
Notice that $\Image\Big(\bb_1 \otimes_{A^{1/p^{e}}} {A^{1/p^{2e}}} \to R_{A^{1/p^{2e}}} \Big)\subseteq \bb_2$ and more generally for $j > i$ that $\Image\Big(\bb_i \otimes_{A^{1/p^{ie}}} {A^{1/p^{je}}} \to R_{A^{1/p^{ne}}} \Big)\subseteq \bb_j$.  Also observe that this is the \emph{opposite} containment compared to what we had in \autoref{sec.RelativeNonFPureIdeals}.

By the same argument as in \autoref{sec.RelativeNonFPureIdeals}, we know that there exists an open set $U \subseteq V$ with $W = f^{-1}(U)$ such that
\begin{equation}
\label{eq.StabilizingUForTau}
\Image(\bb_t \otimes_{A^{1/p^{te}}} A^{1/p^{(t+1)e}} \to R_{A^{1/p^{(t+1)e}}})|_W = \bb_{t+1}|_W.
\end{equation}
As before in \autoref{prop.StabilizingSigma}, we claim that:
\begin{lemma}
\label{lem.StabilizingTau}
With notation as above, $\Image(\bb_{n} \otimes_{A^{1/p^{ne}}} A^{1/p^{(n+1)e}} \to R_{A^{1/p^{(n+1)e}}})|_W = \bb_{n+1}|_W$ for all $n \geq t$.
\end{lemma}
\begin{proof}
%\todo{{\bf Wenliang:} I rewrote the proof to simply it a bit; I hope you don't mind.\\{\bf Karl:}  Certainly not, thanks!}
Replacing $V$ by $U$ and $X$ by $W$,  we may assume that
\[
\bb_{t+1} = \Image\Big(\bb_{t} \otimes_{A^{1/p^{te}}} A^{1/p^{(t+1)e}} \to R_{A^{1/p^{(t+1)e}}} \Big)
\]
By induction on $n$, it suffices to show that, if $\bb_{n+1}=\Image(\bb_{n} \otimes_{A^{1/p^{ne}}} A^{1/p^{(n+1)e}} \to R_{A^{1/p^{(n+1)e}}})$, then $\bb_{n+2}=\Image(\bb_{n+1} \otimes_{A^{1/p^{(n+1)e}}} A^{1/p^{(n+2)e}} \to R_{A^{1/p^{(n+2)e}}})$. Note that $\bb_{n+1}=\Image(\bb_{n} \otimes_{A^{1/p^{ne}}} A^{1/p^{(n+1)e}} \to R_{A^{1/p^{(n+1)e}}})$ if and only if
\begin{align}
&\Image(\phi^{n+1}((I \cdot L^{p^{(n+1)e} - 1 \over p^e - 1})^{1 /p^{(n+1)e}})\to R_{A^{1/p^{(n+1)e}}}) \notag\\
&\subseteq \sum_{i = 0}^{n} \Image\Big(\phi^i((I \cdot L^{p^{ie} - 1 \over p^e - 1})^{1 /p^{ie}}) \otimes_{A^{1/p^{ie}}} A^{1/p^{(n+1)e}} \to R_{A^{1/p^{(n+1)e}}}\Big)
\notag
\end{align}
Tensoring with $\otimes_{R} L$, taking $1/p^e$th roots, and applying $\phi$ yields
\begin{align}
&\Image(\phi^{n+2}((I \cdot L^{p^{(n+2)e} - 1 \over p^e - 1})^{1 /p^{(n+2)e}})\to R_{A^{1/p^{(n+2)e}}}) \notag\\
&\subseteq \sum_{i = 1}^{n+1} \Image\Big(\phi^i((I \cdot L^{p^{ie} - 1 \over p^e - 1})^{1 /p^{ie}}) \otimes_{A^{1/p^{ie}}} A^{1/p^{(n+2)e}} \to R_{A^{1/p^{(n+2)e}}}\Big)
\notag
\end{align}
Hence,
\[\bb_{n+2}=\Image(\bb_{n+1} \otimes_{A^{1/p^{(n+1)e}}} A^{1/p^{(n+2)e}} \to R_{A^{1/p^{(n+2)e}}}).\]
This finishes the proof.
\end{proof}

As before, we define relative test ideals as follows.

\begin{definition}
The \emph{$n$th limiting relative test ideal with respect to $I$ (of the pair $(X/V, \phi)$)} is defined to be $\bb_n \subseteq R_{A^{1/p^{ne}}}$ and is denoted by $\tau_n(X/V, \phi I)$.
\end{definition}

Before proceeding to various properties of relative test ideals, we give one example of relative test ideals.

\begin{example}\textnormal{(\autoref{MustataYoshidaExample} revisited)}
Fix $k$ to be an algebraically closed field of characteristic $p > 2$, set $A = k[t]$ and set $R = k[x,t]$ with the obvious map $X=\Spec(R) \to V=\Spec(A)$.
Let $\phi : R^{1/p} = k[x^{1/p}, t^{1/p}] \to R_{A^{1/p}} = k[x,t^{1/p}]$ be the composition of the local generator $\beta \in \Hom_{R_{A^{1/p}}}(R^{1/p}, {R_{A^{1/p}}})$ with pre-multiplication by $(x^{p^2} + t )^{1\over p} = f^{1\over p}$.  Set $I=\langle x^p+t \rangle$. Then
\[\bb_1=\Image(I\otimes_AA^{1/p}\to R_{A^{1/p}})+\phi(I^{1/p})=\langle x^p+t \rangle +\langle (x^p+t^{1/p})(x+t^{1/p}) \rangle =\langle x^p+t, (x^p+t^{1/p})(x+t^{1/p}) \rangle.\]
To calculate $\bb_2$, it suffices to calculate $\Image(\phi^2(I^{1/p^2}))=\Image(\beta^2((x^{p^2}+t)^{\frac{1+p}{p^2}}I^{1/p^2}))$. Since $(x^{p^2}+t)^{\frac{1+p}{p^2}}\in R_{A^{1/p^2}}$, we can see that $\Image(\phi^2(I^{1/p^2}))\subseteq \langle (x^{p^2}+t)^{\frac{1+p}{p^2}} \rangle$. On the other hand, it is straightforward to check that $\beta^2((x^{p^2}+t)^{\frac{1+p}{p^2}}x^{\frac{p^2-p-1}{p^2}}(x^p+t)^{1/p^2})=(x^{p^2}+t)^{\frac{1+p}{p^2}}=(x+t^{1/p^2})(x^p+t^{1/p}) $. Therefore,
\[\bb_2=\langle x^p+t, (x^p+t^{1/p})(x+t^{1/p}),(x+t^{1/p^2})(x^p+t^{1/p}) \rangle.\]
For each $i\geq 2$, to calculate $\bb_{i+1}$, it suffices to calculate $\Image(\beta^{i+1}((x^{p^2}+t)^{\frac{1+p+\cdots +p^i}{p^{i+1}}}I^{1/p^{i+1}}))$. It is straightforward to check that it is contained in $\langle (x^p+t^{1/p})(x+t^{1/p^2}) \rangle$ and hence is contained in $\bb_2$. Therefore, $\bb_{i+1}=\bb_2$ as well. So we have
\[\bb_n=\begin{cases} \langle x^p+t, (x^p+t^{1/p})(x+t^{1/p}) \rangle & {\rm when\ }n=1\\ \bb_2=\langle x^p+t, (x^p+t^{1/p})(x+t^{1/p}),(x+t^{1/p^2})(x^p+t^{1/p})\rangle & n\geq 2.\end{cases}\]

\end{example}

We now prove a base change statement for relative test ideals.  Compare with \autoref{prop.RelativeNonFPureIdealsAndBaseChange}.  As before, we fix $q_i : X \times_{V} T^i \to X \times_V V^i$ to be the natural map.

\begin{theorem}[Relative test ideals and base change]
\label{thm.BaseChangeTauInitial}

Fix $g : T \to V$ to be any morphism with $T$ excellent, integral and admitting a dualizing complex.  Then
\[
\Image\left((q_{ne})^* \tau_n(X/V, \phi I) \hookrightarrow \O_{X_{T^{ne}}}\right) :=  \tau_n(X/V, \phi I) \cdot \O_{X_{T^{ne}}} = \tau_n(X_T/T, \phi_T I_T).
\]
Furthermore, if $U = U_n$ satisfies the condition of \autoref{lem.StabilizingTau}, then $W = g^{-1}(U) \subseteq T$ satisfies the same condition for $\tau_n(X_T/T, \phi_T I_T)$.
\end{theorem}
\begin{proof}
We write $B = \O_T$ and work locally.
By construction and right exactness of tensor:
\[
\tau_n(X/V, \phi) \cdot \O_{X_{T^{ne}}} = \bb_{n} \tensor_{A^{1/p^{ne}}} B^{1/p^{ne}} =  \sum_{i = 0}^n \big(\phi^i((I \cdot L^{p^{ie} - 1 \over p^e - 1})^{1 /p^{ie}}) \otimes_{A^{1/p^{ie}}} B^{1/p^{ne}} \big)\subseteq R_{B^{1/p^{ne}}}.
\]
But by \autoref{lem.BaseChangeForPhin}, we have
\[
\begin{array}{rl}
&  \displaystyle\sum_{i = 0}^n \big(\phi^i((I \cdot L^{p^{ie} - 1 \over p^e - 1})^{1 /p^{ie}}) \otimes_{A^{1/p^{ie}}} B^{1/p^{ne}} \big) \\
= &  \displaystyle\sum_{i = 0}^n \big(\phi_T^i((I_T \cdot L_T^{p^{ie} - 1 \over p^e - 1})^{1 /p^{ie}}) \otimes_{B^{1/p^{ie}}} B^{1/p^{ne}} \big)\\
= & \tau_n(X_T/T, \phi_T I_T)
 \end{array}
\]
which completes the proof.
\end{proof}

\begin{theorem}[Restriction of $\tau_n$ to fibers]
\label{thm.BaseChangeTauGeneral}
With notation as above, there exists an integer $N \geq 0$, such that for all points $s \in V$, $N \geq n_{\tau(\phi_s I), k(s)}$.  In other words, we have both that $\tau_n(X/V, \phi)\cdot \O_{X_{s^{ne}}} = \tau_n(X_s/s, \phi_s)$ (which always holds) and also that for all $m \geq n \geq N$ that
\[
\tau_n(X/V, \phi I) \otimes_{V^{n}} k(s)^{1/p^m} = \tau_m(X_s/s, \phi_s I_s).
\]
\end{theorem}
\begin{proof}
The idea is the same as for \autoref{thm.BaseChangeSigmaGeneral} and we only sketch it here.  We stratify $V$ as follows.  The result holds on a dense open subset of $U_0 \subseteq V$.  Let $V_1'$ be the complement and let $V_1$ denote the regular locus of $V_1'$.  Base change to $V_1$, and then repeat.  This procedure stops after finitely many iterations by Noetherian induction and we choose an $N$ that works for them all.
\end{proof}

We now construct $I \subseteq R$ whose restriction is contained in the test ideal of every geometric fiber.

\begin{proposition}[The existence of relative test elements]
\label{prop.ExistRelativeTestElements}
With notation as above suppose additionally that $f : X \to V$ is relatively G1 and S2. Then there exists an ideal $I \subset \O_X$ such that for every point $s \in V$ and every perfect extension $K \supseteq k(s)$, we have that $I_K \subseteq \tau(X_K, \psi_K)$ where again $\psi_K$ is simply $\phi_K$ interpreted as in \cite{BlickleSchwedeSurveyPMinusE}.  Additionally, we can assume that $I = \O_X$ at all points of $X$ such that both $X/V$ is smooth and that $\phi$ locally generates $\sHom_{\O_X}(L^{1/p^e}, R_{A^{1/p^e}})$. Furthermore, if $V$ is strongly $F$-regular and quasi-Gorenstein, then $I$ can additionally be chosen so that $I_{V^e}$ is within the absolute test ideal of $(X_{V^e}, \gamma)$ (where $\gamma$ is as in \autoref{subsec.IteratedNonFPureIdealsVsAbsolute}).
\end{proposition}
We caution the reader that it is possible that $\psi_s$ (and thus $\psi_K$) could be the zero-map, and hence $I_K$ the zero ideal.
\begin{proof}
First let $Z_1 \subseteq X$ denote the locus where $\phi R^{1/p^e} \subseteq \sHom_{R_{A^{1/p^e}}}(L^{1/p^e}, R_{A^{1/p^e}})$ is not an isomorphism.  Additionally let $Z_2$ denote the locus where $X$ is not smooth over $V$.  Set $Z = Z_1 \cup Z_2$.  Set $W = X \setminus Z$ (and note that we can also view this as a subset of $X_{V^e}$ for any $e$) and notice that on $W$ we know that $\phi$ can be identified with the trace map up to multiplication by a unit.

Since $f|_W$ is smooth, we observe that $\phi$ (which is identified with trace) is surjective when restricted to $W$ by \autoref{cor.SmoothMapImpliesSurjective}.  Now choose an exponent $m > 0$ such that
\[
I_Z^m \cdot R_{A^{1/p^e}} \subseteq \ba_{1,1} = \sigma_1(X/S, \phi)
\]
Now choose an integer $l > 0$ such that $I_Z^{ml}$ is locally generated by cubes of elements of $I_Z^m$ (this only depends on the number of local generators of $I_Z$).  Then the formation of $I_Z^{ml}$ is obviously compatible with base change (in that the extension of $I_Z^{ml}$ to the base change will also satisfy the same containment condition.).  Thus set $I = I_Z^{ml}$.  The first result follows immediately from \cite[Theorem on Page 90]{HochsterFoundations} or \cite[Proof of Proposition 3.21]{SchwedeTestIdealsInNonQGor}.

For the second result, we notice that $\gamma$ locally generates $\sHom_{R_{A^{1/p^e}}}(N^{1/p^e}, R_{A^{1/p^e}})$ and that $I_Z^{m}$ is also in the image of $\gamma$.  Therefore, $I^{ml}$ works by the above references.
\end{proof}
Therefore we obtain:

\begin{corollary}
\label{cor.RestrictionOfRelativeTauToFibers}
Using the notation of \autoref{thm.BaseChangeTauGeneral} assume further that $I$ satisfies the condition of \autoref{prop.ExistRelativeTestElements}.  If $s$ is a perfect point then for all $n \geq N$ as in \autoref{thm.BaseChangeSigmaGeneral}
\[
\tau_n(X/V, \phi I) \otimes_{V^{n}} k(s)^{1/p^m} = \tau(X_s, \psi_s)
\]
 where $\psi_s$ is $\phi_s$ interpreted as a $p^{-e}$-linear map.
 \end{corollary}
 %\todo{{\bf Wenliang:} Do we want to assume $s$ is a perfect point here?\\{\bf Karl:} Yes, definitely.}
 \begin{proof}
Simply observe that in general the ascending $\tau_{j}(X_s/s, \phi_s I_s) \subseteq \tau_{j+1}(X_s/s, \phi_s I_s) \subseteq \ldots$ stabilize to $\tau(X_s, \psi_s)$.  Furthermore, by \autoref{thm.BaseChangeTauGeneral} this ascending chain stabilizes.
\end{proof}

\begin{corollary}
\label{cor.FRegularityofPerfectClosyreImpliesGeneralClosedPoint}
Using the notation of \autoref{thm.BaseChangeTauGeneral} and assume further that $I$ satisfies the condition of \autoref{prop.ExistRelativeTestElements}.  If the perfect closure of the generic fiber of $f:X\to V$ is strongly $F$-regular, {\it i.e.} $\tau(R_{k(V^{\infty})}, \psi_{\infty}I_{\infty})=R_{k(V^{\infty})}$, then there exists an open subset $U\subseteq V$ such that $(X_s,\psi_s)$ is also strongly $F$-regular for each perfect point $s\in U$.
\end{corollary}
\begin{proof}
It follows immediately from the proof of \autoref{prop.StabilizingSigma} and \autoref{cor.RestrictionOfRelativeTauToFibers}.
\end{proof}

\mysubsection{Relative test ideals vs absolute test ideals}
\label{subsec.RelativeTauVsAbsoluteTau}
We make the following assumption.

\begin{convention}
\label{conv.NiceBaseForTau}
For the rest of \autoref{subsec.RelativeTauVsAbsoluteTau}, we assume that $V$ is regular and $F$-finite.
\end{convention}
It is natural to try to relate the relative test ideal $\tau_n(X/V, \phi I)$ with the absolute test ideal $\tau(X_{V^{ne}}, \gamma^n)$ as in \autoref{subsec.IteratedNonFPureIdealsVsAbsolute}.  Indeed, if we construct $I$ as we did in the proof of \autoref{prop.ExistRelativeTestElements} then it follows easily that $I \cdot R_{A^{1/p^{ne}}} \subseteq \tau(X_{V^{ne}}, \gamma^n)$.  Thus fix such an $I$.

We now recall the following diagram from \autoref{lem.FactorizationOfSigmaImages}

%\[
%\scriptsize
%\xymatrix@C=12pt{
%& \ar`l[dd]`/4pc[dd][ddr]_-{\gamma^n} \Big(L^{p^{ne} - 1 \over p^e - 1}\Big)^{1/p^{ne}} \otimes_{A^{1/p^{ne}}} \Big(M^{p^{ne} - 1 \over p^{e}-1}\Big)^{1/p^{(n+n)e}} \ar[r]^-{\gamma'} %\ar[d]_-{\mu_n} & \Big(L^{p^{(n-1)e} - 1 \over p^e - 1}\Big)^{1/p^{(n-1)e}} \otimes_{A^{1/p^{(n-1)e}}} \Big(M^{p^{(n-1)e} - 1 \over p^e - 1}\Big)^{1/p^{(n+n-1)e}} \ar[d]^-{\mu_{n-1}} %\ar@/^10pc/[dd]^{\gamma^{n-1}} \\% \ar`r[d]`/4pc[dd]^{\gamma^{n-1}}[dd] \\
%& \Big(L^{p^{ne} - 1 \over p^e - 1}\Big)^{1/p^{ne}} \ar[r] \ar[dr]_{\phi^n} & \Big(L^{p^{(n-1)e} - 1 \over p^e - 1}\Big)^{1/p^{(n-1)e}} \otimes_{A^{1/p^{(n-1)e}}} A^{1/p^{ne}} %\ar[d]^-{\phi^{n-1} \otimes \ldots} \\
%& & R_{A^{1/p^{ne}}}
%}
%\]
{
\scriptsize
\begin{tikzcd}
\Big(L^{p^{ne} - 1 \over p^e - 1}\Big)^{1/p^{ne}} \otimes_{A^{1/p^{ne}}} \Big(M^{p^{ne} - 1 \over p^{e}-1}\Big)^{1/p^{(n+n)e}}
\arrow[r, "\gamma'"]
\arrow[d, "\mu_n"]
\arrow[ddr, "\gamma^n"' near end, rounded corners, to path={ -- ([xshift=-2ex]\tikztostart.west)  |- (\tikztotarget.west)\tikztonodes} ]
%\arrow[dashed, rounded corners, to path={ -- ([yshift=-2ex]\tikztostart.south) -| ([xshift=-1.5ex]\tikztotarget.west) -- (\tikztotarget)}]{uulll}
& \Big(L^{p^{(n-1)e} - 1 \over p^e - 1}\Big)^{1/p^{(n-1)e}} \otimes_{A^{1/p^{(n-1)e}}} \Big(M^{p^{(n-1)e} - 1 \over p^e - 1}\Big)^{1/p^{(n+n-1)e}}
\arrow[d, "\mu_{n-1}"]
\arrow[dd, "\gamma^{n-1}"' near start, rounded corners, to path={ -- ([xshift=2ex]\tikztostart.east) |- (\tikztotarget.east)\tikztonodes} ]
\\
\Big(L^{p^{ne} - 1 \over p^e - 1}\Big)^{1/p^{ne}}
\arrow[r]
\arrow[dr, "\varphi^n"']
& \Big(L^{p^{(n-1)e} - 1 \over p^e - 1}\Big)^{1/p^{(n-1)e}} \otimes_{A^{1/p^{(n-1)e}}} A^{1/p^{ne}}
\arrow[d, "\varphi^{n-1} \otimes \ldots"]
\\
& R_{A^{1/p^{ne}}}
\end{tikzcd}
}
\noindent
We let $I^{1/p^{je}}_{A^{1/p^{ke}}}$ denote the extension of $I^{1/p^{je}}$ to $R^{1/p^{je}}_{A^{1/p^{ke}}}$.  Then notice that $\mu_{j}$ sends $I^{1/p^{je}}_{A^{1/p^{(j + n)e}}}$ (times $\mu_j$'s domain) back \emph{onto} $I^{1/p^{je}}$ since each $\mu_j$ is surjective.

We then observe that:
\begin{lemma}
With notation as above and assuming \autoref{conv.NiceBaseForTau}, then $\tau_n(X/V, \phi I) \subseteq \tau(X_{V^{ne}}, \gamma)$.
\end{lemma}
\begin{proof}
We know $\tau_n(X/V, \phi I)$ is the sum
\begin{equation}
\label{eq.SumTauRelativeN}
\sum_{j = 0}^n \Big(\phi^{j} \otimes_{A^{1/p^{je}}} A^{1/p^{ne}}\Big)\Bigg(I^{1/p^{je}} \cdot \Big(L^{p^{je} - 1 \over p^e - 1}\Big)^{1/p^{je}} \otimes_{A^{1/p^{je}}} A^{1/p^{ne}} \Bigg)
\end{equation}
On the other hand, $\tau(X, \gamma)$ is the sum
\begin{equation}
\label{eq.SumTauGamma}
\sum_{j = 0}^{\infty} \gamma^j\Big( I^{1/p^{je}}_{A^{1/p^{(j+n)e}}} \cdot \Big(L^{p^{je} - 1 \over p^e - 1}\Big)^{1/p^{je}} \otimes_{A^{1/p^{je}}} \Big(M^{p^{je} - 1 \over p^{e}-1}\Big)^{1/p^{(j+n)e}}\Big)
\end{equation}
By our above observations about the surjectivities of the $\mu_j$ above, the sum of the $0$th through $j$th terms of \autoref{eq.SumTauGamma} is equal to the sum \autoref{eq.SumTauRelativeN}.  The result follows.
\end{proof}

Additionally, at least over an open subset of the base, we actually have that the relative test ideal and absolute test ideal agree, \cf \autoref{thm.RelativeSigmaVsAbsoluteOverOpen}.

\begin{theorem}
\label{thm.RelativeTauVsAbsoluteOverOpen}
With notation as above and assuming \autoref{conv.NiceBaseForTau}, choose $n > t = n_{\tau(\phi I), k(V)}$.  Then there exists a dense open set $U \subseteq V \cong V^e$ with $W = f^{-1}(U) \subseteq X$ such that $\tau(X_{V^{1/p^{ne}}}, \gamma)|_W = \tau_n(X/V, \phi I)|_W$.  Furthermore, shrinking $U$ further if necessary and possibly increasing $n$, we can require for all perfect points $u \in U$, that
\[
\tau(X_{V^{1/p^{ne}}}, \gamma)\cdot \O_{X_u} = \tau_n(X_u/u, \phi_u I_u) = \tau(X_u, \psi_u).
\]
where again $\psi_u$ is $\phi_u$ viewed as a $p^{-e}$-linear map.
\end{theorem}
\begin{proof}
First we observe that by taking $U$ as in \autoref{eq.StabilizingUForTau}, we can assume that $\bb_{n-1,n} = \bb_{n,n}$.  But by the diagram above, we see that the $n-1$st partial sum defining $\tau(X_{V^{1/p^{ne}}}, \gamma)$ in \autoref{eq.SumTauGamma} is equal to $\bb_{n-1,n}$.  Likewise the $n$th partial sum is equal to $\bb_{n,n}$.  However, once two adjacent partial sums defining $\tau(X_{V^{1/p^{ne}}}, \gamma)$ coincide, the sum stabilizes for further powers by the computation we made when defining the $\bb_{n}^{\infty}$ in \autoref{eq.DefiningBBInfinity}.
The second statement follows from the first statement and from \autoref{thm.BaseChangeTauInitial}, \cf the argument of \autoref{cor.RestrictionOfRelativeTauToFibers}.
\end{proof}

When dealing with relatively non-$F$-pure ideals $\sigma$, we actually obtained the above restriction theorem without shrinking $X$ to $U$ in \autoref{cor.SigmaRestrictsToAllPoints}.   The difference is that for $\sigma$ we have the easy containment $\sigma(X_{V^{1/p^{ne}}}, \gamma^n) \subseteq \sigma_n(X/V, \phi)$ by \autoref{lem.FactorizationOfSigmaImages}.  For $\tau$ however, the easy containment is reversed.  This leads us to the following question:

\begin{question}
Is it true that
$
\tau(X_{V^{1/p^{ne}}}, \gamma)\cdot \O_{X_{s^{ne}}} = \tau(X_s, \psi_s)
$
for all perfect points $s \in V$ at least when $n \gg 0$?
\end{question}

Rephrasing \autoref{thm.RelativeTauVsAbsoluteOverOpen} for divisors we obtain:

\begin{corollary}
\label{cor.RestrictionOfTauForDivisors}
Suppose that $f : X \to V$ is a flat finite map to an excellent regular scheme $V$.  Additionally suppose that $f$ is relatively G1 and S2 and $I$ is chosen as in \autoref{prop.ExistRelativeTestElements}.
Choose $\Delta$ satisfying conditions (a)--(d) of \autoref{rem:relative_canonical}.  Then there exists an open dense set $U \subseteq V$ such that
\[
\tau(X_{V^{1/p^{ne}}}, \Delta)\cdot \O_{X_{u^{1/p^{ne}}}} = \tau(X_u, \Delta|_{X_u})
\]
for all perfect points $u \in U$.
\end{corollary}
\begin{proof}
Using \autoref{lem.DeltaPhiRestrictedEqualsDeltaPsi} and \autoref{lem.ChoiceOfDeltaPhiEqualsDeltaGamma}, we see that $\Delta$ corresponds to some $\phi : L^{1/p^e} \to R_{A^{1/p^e}}$.  Thus we simply apply \autoref{cor.RestrictDeltaAndPhiToFiberIsOk} and \autoref{thm.RelativeTauVsAbsoluteOverOpen}.
\end{proof}

\mysubsection{Applications to strong $F$-regularity and divisor pairs}
\begin{definition}
\label{defn.RelativeFRegularity}
Recall that a pair $(X, \psi : L^{1/p^e} \to \O_X = R)$ is called \emph{strongly $F$-regular} if $\tau(X, \psi) = \O_X$.  Given $(X/S, \phi : L^{1/p^e} \to R_{A^{1/p^e}})$ and some $I \subseteq R$ satisfying the condition of \autoref{prop.ExistRelativeTestElements}, we say that $(X/S, \phi I)$ is \emph{relatively strongly $F$-regular} if we have $\tau_{n}(X/S, \phi I) = R_{A^{1/p^{ne}}}$ for some $n > 0$ (equivalently, all $n \gg 0$ since the $\tau_n$ ascend), \cf \cite{HashimotoFPureHoms}.
\end{definition}

\begin{remark}
\label{rem.RelativelyFRegOverPtIsGeometric}
Suppose that $S$ is a point and $I$ is chosen as in \autoref{prop.ExistRelativeTestElements}, then $(X/S, \phi I)$ is relatively strongly $F$-regular if and only if it is geometrically strongly $F$-regular (\ie, $(X_t, \gamma_t)$ is strongly $F$-regular for every point $t \to S$).  The argument is the same as in \autoref{rem.RelativelyFPureOverPtIsGeometric}
\end{remark}

The same ideas imply the strongly $F$-regular locus of a proper map is open.
\begin{corollary}[Openness of the strongly $F$-regular locus]
\label{cor.OpennessOfStrongFRegularity}
With notation as before, assume additionally that $f : X \to V$ is proper and that $I$ satisfies the condition of \autoref{prop.ExistRelativeTestElements}.  Assume that $s \in V$ is a point and that $(X_s/s, \phi_s I_s)$ is relatively strongly $F$-regular (for example, if $s$ is a perfect point, this just means it is strongly $F$-regular and is independent of $I$).  Then there exists a dense open set $U \subseteq V$ containing $s$ such that $(X_u/u, \phi_u I_u)$ is relatively strongly $F$-regular for all $u \in U$ (in particular, $(X_u, \phi_u)$ is strongly $F$-regular for all perfect points $u \in U$).
\end{corollary}
\begin{proof}
Choose $n \gg 0$ such that $\tau_n(X_s/s, \phi_s I_s) = R_{\O_s^{1/p^{ne}}}$ since $(X_s/s, \phi_s I_s)$ is strongly $F$-regular.  By \autoref{thm.BaseChangeTauGeneral}, we know that $\tau_n(X/V, \phi I) = R_{A^{1/p^{ne}}}$ in a neighborhood $W \subseteq X$ of $X_s$.  Let $Z = X \setminus W \subseteq X$ be the complement of that neighborhood.  Since $f$ is proper, $f(Z)$ is closed, and also doesn't contain $s$.  Set $U = V \setminus f(Z)$.  Then $\tau_n(X_U/U, \phi_U I|_U) = R_{\O_U^{1/p^{ne}}}$.  It follows from \autoref{thm.BaseChangeTauGeneral} that all the fibers $(X_u/u, \phi_u I_u)$ are relatively strongly $F$-regular as desired.
\end{proof}

At least for proper maps, we also obtain that the definition of relative strongly $F$-regularity \autoref{defn.RelativeFRegularity} is independent of the choice of $I$.

\begin{lemma}
Suppose that $f : X \to V$ is proper and that $(X/V, \phi I)$ is relatively strongly $F$-regular for some $I$ satisfying the condition of \autoref{prop.ExistRelativeTestElements}.  Then for all $J \subseteq R$ such that $J_s$ is non-zero on every component of every fiber $X_s$, we have that
$\tau_n(X/S, \phi J) = R_{A^{1/p^{ne}}}$ for some $n > 0$ (equivalently, all $n \gg 0$).
\end{lemma}
In particular in \autoref{defn.RelativeFRegularity}, it would be equivalent to require $\tau_{n}(X/S, \phi J) = R_{A^{1/p^{ne}}}$ for all such $J$ and some $n$.
\begin{proof}
Choose a closed point $s \in V$ and a perfect extension $K \supseteq k(s)$.  Then $\tau_m(X_K/K, \phi_K I_K) = R_{K^{1/p^{ne}}}$ for some $m > 0$ by \autoref{thm.BaseChangeTauGeneral}.  But since $K$ is a perfect field extension, this implies that $\tau_m(X_K, \psi_K I_K) = R_{K^{1/p^{me}}}$ as well.  On the other hand, since $(X_K, \psi_K)$ is strongly $F$-regular, we see that $\tau_n(X_K/K, \phi_K J_K) = R_{K^{1/p^{ne}}}$ for some $n > 0$.  Hence since $k(s) \subseteq K$ is faithfully flat, we see that $\tau_n(X_s/s, \phi_s J_s) = R_{k(s)^{1/p^{ne}}}$.  But then since $\tau_n(X/V, \phi J)$ restricts to $\tau_n(X_s/s, \phi_s J_s)$ we observe that $\tau_n(X/V, \phi J) = R_{A^{1/p^{ne}}}$ at least in a neighborhood of $s$ using that $f$ is proper and the same argument we made in \autoref{cor.OpennessOfStrongFRegularity}.  But we can find such an $n$ for every $s \in V$, and so the lemma holds by the quasi-compactness of $V$.
\end{proof}

We now state our result in the divisorial case.

\begin{corollary}
\label{cor.OpennessOfFregDivisors}
With notation as above, suppose that $f : X \to V$ is a proper map and that $\Delta$ is a $\bQ$-divisor satisfying conditions (a)--(d)\footnote{For example, these conditions hold if $V$ is regular, $f : X \to V$ is geometrically normal, $K_X + \Delta$ is $\bQ$-Cartier with index not divisible by $p$ and $\Delta$ does not contain any fiber of $f$ in its support.} of \autoref{rem:relative_canonical}.  Additionally suppose that for some perfect point $s \in V$, the fiber $(X_s, \Delta|_{X_s})$ is strongly $F$-regular.  Then there exists a dense open set $U \subseteq V$ containing $s$ such that $(X_u, \Delta|_{X_u})$ is strongly $F$-regular for all perfect $u \in U$.
\end{corollary}
\begin{proof}
Using \autoref{rem:relative_canonical} we construct a relatively divisorial $\phi : L^{1/p^e} \to R_{A^{1/p^e}}$ corresponding to $\Delta$.   Choose now $I$ satisfying the condition of \autoref{prop.ExistRelativeTestElements}.  It follows that $(X_s, \psi_s)$ is strongly $F$-regular and hence that $(X_s/s, \phi_s I_s)$ is relatively strongly $F$-regular since $s$ is a perfect point.  Then \autoref{cor.OpennessOfStrongFRegularity} and \autoref{cor.RestrictDeltaAndPhiToFiberIsOk} complete the proof.
\end{proof}

By a perturbation trick we can also handle the case that the index of $K_X + \Delta$ is divisible by $p > 0$, at least over curves.

\begin{corollary}
\label{cor.OpennessOfFregDivisorsBadIndex}
With notation as above, suppose that $f : X \to V$ is a projective map to a regular 1-dimensional base, and that $\Delta$ is a $\bQ$-divisor satisfying the following 4 conditions:
\begin{itemize}
\item[(a$'$)] $\Delta' = {1 \over m} D$ for some Weil divisor $D$.
\item[(b)] $D$ is a Weil divisor on $X$ which is Cartier in codimension 1 and Cartier at every codimension 1 point of every fiber.
\item[(c)] $D$ is trivial along the codimension-1 components of the non-smooth locus of $X \to V$ and the codimension-1 components of the non-smooth locus of every fiber.
\item[(d$'$)] $l/m \in \bZ$ and $\big(\omega_{X/V}^{l} \otimes \O_X ( l \Delta') \big)^{**}$ is a line bundle.
\end{itemize}
Additionally suppose that for some perfect point $s \in V$, the fiber $(X_s, \Delta|_{X_s})$ is strongly $F$-regular.  Then there exists a dense open set $U \subseteq V$ containing $s$ such that $(X_u, \Delta|_{X_u})$ is strongly $F$-regular for all perfect $u \in U$.
\end{corollary}
\begin{proof}
We may certainly suppose that $V$ is affine.  Since the fiber $X_s$ is normal, and hence geometrically normal, we may also assume that the nearby fibers satisfy the same condition (using that $f$ is proper).  Thus we may assume that all the fibers are geometrically normal and in fact that the map $f : X \to V$ is geometrically integral.  Without loss of generality, by base change we can assume that $V$ is normal and hence $X$ is normal itself.  Thus we may assume that $X$ is normal.

Now, since the base is $1$-dimensional, we claim that can assume that $K_{X/V}$ doesn't contain any fiber.  The only fiber we must worry about is $X_s$ (as the others can be handled by shrinking $V$).  We argue as follows:  note first that $K_{X/V}$ can be viewed as an honest Weil divisor since we already assumed that $X$ is normal.  On the other hand, each fiber is a Cartier divisor.  Hence, if $K_{X/V}$ is non-trivial along the generic point of $X_s$, by twisting by the pullback of $s$, we can assume that $K_{X/V}$ does not contain $X_s$.

Write $l \Delta' =c p^{e_0} \Delta'$ where $p$ does not divide $c$.
Let $E$ be an effective Cartier divisor on $X$ not containing any fiber such that $K_{X/V} + E$ is effective.  We also observe that $K_{X/V} + E$ satisfies conditions (b) and (c) above.  Let $\Gamma = {E + K_{X/V} + \Delta' \over p^{e} - 1}$ for some $e \gg e_0$.  Then consider
\[
\Delta := \Delta' + \Gamma.
\]
Note that
\[
\begin{array}{rl}
& c(p^e-1)(K_{X/V} + \Delta) \\
=&  c(p^e - 1)(K_{X/V} + \Delta' + \Gamma) \\
=&  c(p^e - 1)(K_{X/V} + \Delta') + c(E + K_{X/V} + \Delta') \\
= & cp^e (K_{X/V} + \Delta') + cE
\end{array}
\]
which is certainly Cartier.  Hence $K_{X/V} + \Delta$ satisfies conditions (a)--(d) from \autoref{rem:relative_canonical}.  On the other hand, since $e \gg 0$, we know that $(X_s, \Delta|_{X_s}) = (X_s, \Delta'|_{X_s} + \Gamma|_{X_s})$ is still strongly $F$-regular and so by \autoref{cor.OpennessOfFregDivisors} there exists an open set $U$ such that $(X_u, \Delta|_{X_u})$ is strongly $F$-regular for all $u \in U$.  But $\Delta \geq \Delta'$ and so the result follows for $\Delta'$ as well.
\end{proof}

\begin{remark}
In the case of a normal $X$ and in the non-relative case, we know that $\tau(X; \Delta) = \tau(\omega_X, K_X + \Delta)$.  Furthermore, we then know that
\[
\Tr^e(F^e_* \tau(\omega_X, K_X + \Delta)) = \tau(\omega_X, {1 \over p^e}(K_X + \Delta)) =  \tau(X, {1 \over p^e}(K_X + \Delta) - K_X).
\]
Reversing this process gives us a nice means to compute $\tau(X, \Delta)$ when the index of $K_X + \Delta$ is divisible by $p$.  It would be natural try to prove a relative version of this, which may yield a suitable definition of relative test ideals for $K_X + \Delta$ of any index.  We won't attempt this here.
\end{remark}

\section{Relative test submodules, $F$-rationality and $F$-injectivity}
\label{sec.RelativeFRationalityFInjectivity}

Our goal in this (somewhat shorter) section is to explore relative test submodules and non-$F$-injective modules.
Throughout this section, we assume that $f : X \to V$ is a Cohen-Macaulay morphism.  This provides us with base change for relative canonical sheaves $\omega_{X/V}$ \cite{ConradGDualityAndBaseChange}.

We first define relative non-$F$-injective modules $\sigma_n(X/V, \omega_{X/V})$ and relative test submodules $\tau_n(X/V, \omega_{X/V})$.
%Consider the relative Frobenius map $X^n \to X_{V^n}$ over $V^n$.
Recall the trace map, \autoref{lem.BaseChangeOfTraceForCMMaps}
\[
\Phi_{X/V,n} : \omega_{X/V}^{1/p^n} = \omega_{X^n/V^n} \to \omega_{X_{V^n}/V^n} \cong \omega_{X/V} \otimes_{A} A^{1/p^{ne}}.
\]
Here the final isomorphism follows from \cite[Theorem 3.6.1]{ConradGDualityAndBaseChange} since $f$ is a Cohen-Macaulay morphism.  We will form $\sigma$ and $\tau$ relative to these maps, instead of relative to the map $\phi$ discussed previously.  One key point to remember is that by \autoref{lem.BaseChangeOfTraceForCMMaps}, the map $\Phi_{X/V,n}$ is compatible with arbitrary base change (since in this section, $f$ is a Cohen-Macaulay morphism).  We also observe that the composition of trace maps
\begin{equation}
\label{eq.CompositionOfTraceMaps}
\omega_{X^m/V^m} \xrightarrow{\Phi_{X/V,m-n}^{1/p^{m-n}}} \omega_{X^n_{V^{m}}/V^m}  \xrightarrow{\Phi_{X/V,n} \otimes_{A^{1/p^n}} A^{1/p^m}} \omega_{X/V} \otimes_A A^{1/p^{m}} \cong \omega_{X_{V^{m}}/V^m}.
\end{equation}
can be identified with $\Phi_{X/V,m}$.
%\begin{lemma}
%\label{lem.FactorizationOfTraceRelative}
%With notation as above and $m > n > 0$, then the map defined by the following composition
%\[
%\omega_{X^m/V^m} \xrightarrow{\Phi_{X/V,m-n}^{1/p^{m-n}}} \omega_{X^n_{V^{m}}/V^m}  \xrightarrow{\Phi_{X/V,n} \otimes_{A^{1/p^n}} A^{1/p^m}} \omega_{X/V} \otimes_A A^{1/p^{m}} \cong \omega_{X_{V^{m}}/V^m}.
%\]
%can be identified with $\Phi_{X/V,m}$.
%\end{lemma}
%\begin{proof}
%The first map is identified with the evaluation-at-$1$ map
%\[
%\sHom_{R^{1/p^{n}}_{A^{1/p^m}}}(R^{1/p^m}, \omega_{R^{1/p^{n}}_{A^{1/p^m}}\big/A^{1/p^m}}) \to \omega_{R^{1/p^{n}}_{A^{1/p^m}}\big/A^{1/p^m}}
%\]
%whereas the second is nothing but the evaluation-at-$1$ map
%\[
%\omega_{R^{1/p^{n}}_{A^{1/p^m}}} = \sHom_{R_{A^{1/p^m}}}(R^{1/p^n}_{A^{1/p^m}}, \omega_{R_{A^{1/p^m}}\big/A^{1/p^m}}) \to \omega_{R_{A^{1/p^m}}\big/A^{1/p^m}} \cong \omega_{R_{A^{1/p^n}}\big/A^{1/p^n}} \otimes_{A^{1/p^n}} A^{1/p^m}.
%\]
%It follows easily that the composition
%\[
%\begin{array}{rl}
%& \sHom_{R^{1/p^{n}}_{A^{1/p^m}}}(R^{1/p^m}, \sHom_{R_{A^{1/p^m}}}(R^{1/p^n}_{A^{1/p^m}}, \omega_{R_{A^{1/p^m}}\big/A^{1/p^m}}))\\
%\cong & \sHom_{R_{A^{1/p^m}}}(R^{1/p^m} \otimes_{R^{1/p^{n}}_{A^{1/p^m}}} R^{1/p^n}_{A^{1/p^m}}, \omega_{R_{A^{1/p^m}}\big/A^{1/p^m}})\\
%\cong & \sHom_{R_{A^{1/p^m}}}(R^{1/p^m}, \omega_{R_{A^{1/p^m}}\big/A^{1/p^m}})\\
%\to & \omega_{R_{A^{1/p^m}}\big/A^{1/p^m}}
%\end{array}
%\]
%is also an evaluation-at-1 map.  This completes the proof.
%\end{proof}

\mysubsection{The definition and basic properties of $\sigma_n(X/V, \omega_{X/V})$}

For each integer $n > 0$, define $\bc_n := \Image(\Phi_{X/V, n}) \subseteq \omega_{X_{V^n}/V^n} \cong \omega_{X/V} \otimes_{A} A^{1/p^n}$.
Furthermore, for each $m \geq n$, using the factorization in \autoref{eq.CompositionOfTraceMaps} it is easy to see that
\begin{equation}
\label{eq.ContainmentOfRelativeNonFinjective}
\bc_m \subseteq \Image\big(\bc_n \otimes_{A^{1/p^n}} A^{1/p^m} \to \omega_{R_{A^{1/p^m}}}\big)
\end{equation}
just as we observed in \autoref{sec.RelativeNonFPureIdeals}.

\begin{definition}
With notation as above, we define the \emph{$n$th relative non-$F$-injective submodule}, denoted $\sigma_n(X/V, \omega_{X/V})$ to be $\bc_n \subseteq \omega_{X_{V^{n}}/V^n}$.
\end{definition}

As before, we will prove a stabilization statement for $\bc_n$ and $\bc_m$, in particular that the containment \autoref{eq.ContainmentOfRelativeNonFinjective} is an equality for all $m > n \gg 0$.  By base changing with $k(V^{\infty})$, the perfection of the residue field of the generic point of $V$, we can again find an integer $n_0 > 0$ and an open set $U \subseteq V$ such that
\[
\bc_{n_0+1}|_{f^{-1} U} = \Image\big(\bc_{n_0} \otimes_{A^{1/p^{n_0}}} A^{1/p^{n_0+1}} \to \omega_{R_{A^{1/p^{n_0+1}}}}\big)|_{f^{-1} U}.
\]
Again observe that for $U$ inside the regular locus of $V$, we can identify the image above with $\big(\bc_{n_0} \otimes_{A^{1/p^{n_0}}} A^{1/p^{n_0+1}}\big)|_U$ by the flatness of $A^{1/p^m}$ over $A^{1/p^n}$.
  We then obtain:
\begin{proposition}
\label{prop.StabilizingSigmaCanonical}
Fix notation as above.  For every integer $n \geq n_0$, there exists a nonempty open subset $U_n \subseteq V$ of the base scheme $V$ satisfying the following condition.  If one sets %$U_n^{ne} \subseteq V^{ne}$ to be the corresponding open set of $V^{ne} \cong V$ with
$X_n = f^{-1}(U_n)$, then we have that for every $m \geq n$
\begin{equation}
\label{eq.sigmaNRestrictsCanonical}
\sigma_m(X/V, \omega_{X/V})|_{X_n} = \Image\big(\sigma_{n}(X/V, \omega_{R/A})\cdot \O_{X_n} \otimes_{A^{1/p^n}} A^{1/p^m} \to \omega_{R_{A^{1/p^m}}} \big).
\end{equation}
Furthermore we may assume that $U_{n_0} \subseteq U_{n_0+1} \subseteq \cdots \subseteq U_{n} \subseteq U_{n+1} \subseteq \cdots$ form an ascending chain of open sets.
\end{proposition}
\begin{proof}
The proof is identical to that of \autoref{prop.StabilizingSigma} and so we omit it.
%\todo{{\bf Karl:} Write it down carefully, and then comment it out.\\
%{\bf Wenliang:} We will have to reproduce the huge diagram in the proof of \autoref{prop.StabilizingSigma} here. If it is necessary, I will do it.\\
%{\bf Karl:} If everyone agrees it is exactly the same as the previous proof, I'm ok with just leaving it as that.}
\end{proof}

%\begin{definition}
%With notation as above, we define the \emph{$n$th limiting non-$F$-injective module} to be $\bc_n$, we denote it by $\sigma_n(X/V, \omega_{X/V})$.
%\end{definition}

We now point out that relative non-$F$-injective modules behave well with respect to base change.  Recall that if $g : T \to V$ is a map, then $q_n : X_{T^n} \to X_{V^n}$ is the induced map.

\begin{proposition}[Base change for $\sigma_n(X/V, \omega_{X/V})$]
\label{prop.BaseChangeForSigmaOmega}
Suppose that $g : T \to V$ is a map from an excellent scheme with a dualizing complex, then using the notation of \autoref{subsec.BaseChangeOfPhi}
\[
\Image\left((q_{n})^* \sigma_n(X/V, \omega_{X/V}) \rightarrow \omega_{X_{T^{n}}/T^n}\right) = \sigma_n(X_T/T, \omega_{X_T/T}).
\]
Furthermore, if $U = U_n$ satisfies condition \eqref{eq.sigmaNRestrictsCanonical} from \autoref{prop.StabilizingSigmaCanonical}, then $W = g^{-1}(U) \subseteq T$ satisfies the same condition for $\sigma_n(X_T/T, \omega_{X_T/T})$.
\end{proposition}
\begin{proof}
The first statement is an immediate consequence of base change relative canonical sheaves and trace \autoref{lem.BaseChangeOfTraceForCMMaps}.  For the second, if $\Phi_{X/V,n} \otimes_{A^{1/p^n}} A^{1/p^{n+1}}$ and $\Phi_{X/V,n+1}$ have the same image in $\omega_{X_{V^n}/V^n}$, then its easy to see that the base changed maps also have the same images.
\end{proof}

Recall for any $F$-finite scheme $Y$ with canonical module $\omega_Y$, then $\sigma(\omega_Y)$ is equal to $\Im(F^e_* \omega_Y \xrightarrow{\Tr^e} \omega_Y)$ for any $e \gg 0$.

\begin{corollary}[Restriction theorem for $\sigma_n(X/V, \omega_{X/V})$]
\label{cor.RestrictionTheoremForSigmaOmega}
With notation as above, there exists an integer $N > 0$ such that for every perfect point $s \in V$, we have
\[
\sigma(\omega_{X_s}) = \Image\left( \sigma_n(X/V, \omega_{X/V}) \tensor_{A^{1/p^{ne}}} k(s)^{1/p^{ne}} \to \omega_{X_s/s} \right).
\]
for all $n > N$.
\end{corollary}
\begin{proof}
Taking $g : s \to V$ in the previous theorem and so obtain that the image is equal to $\sigma_n(X_s/s, \omega_{X_s/s})$.  Furthermore, since $k(s)$ is perfect, we have containments $\ldots \supseteq \sigma_n(X_s/s, \omega_{X_s/s}) \supseteq \sigma_{n+1}(X_s/s, \omega_{X_s/s}) \supseteq \ldots$ with a descending intersection that coincides with $\sigma(\omega_{X_s})$.  Furthermore, by \autoref{prop.BaseChangeForSigmaOmega}, over a dense open set $U \subseteq V$ and some $N > 0$ we have $\sigma_n(X_s/s, \omega_{X_s/s}) = \sigma_{n+1}(X_s/s, \omega_{X_s/s})$ and hence $\sigma_n(X_s/s, \omega_{X_s/s}) = \sigma(\omega_{X_s})$ by the construction of $\sigma(\omega_{X_s})$ for all $n \geq N_0$.  Let $V_1 = V \setminus U$, base change with $V_1$ and obtain the result over a dense open subset $U_1$ of $V_1$ (for some $N_1$).  By Noetherian induction, this process terminates.
\end{proof}

\mysubsection{The definition and basic properties  of $\tau_n(X/V, \omega_{X/V})$}

The goal of this section is to develop the basics of a relative theory of test submodules.
We fix the notation of the previous sections and additionally assume that $X$ is geometrically normal over $V$ which we now also assume is regular.  We let $J = J_{X/V} \subseteq R$ be the Jacobian ideal sheaf of $X$ over $V$.  We observe that the formation of $J_{X/V}$ commutes with base change in the following sense: for any $T \to V$, we have $J_{X/V} \cdot \O_{X_T} = J_{X_T/T}$. To see this, just note that $J_{X/V}$ can be defined as a Fitting ideal of $\Omega_{X/V}$ (\cite[Discussion 4.4.7]{HunekeSwansonIntegralClosure}) and that the formation of $\Omega_{X/V}$ (\cite[Proposition 16.4]{EisenbudCommutativeAlgebraWithAView}) and Fitting ideals (\cite[Discussion 4.4.7]{HunekeSwansonIntegralClosure}) commutes with arbitrary base change. Note that $J$ is nonzero at any generic point of $X$ since $X$ is geometrically reduced.  Furthermore, on every perfect fiber, the Jacobian ideal is contained in the (big) test ideal \cite[Theorem on Page 213]{HochsterFoundations}.

 We have the images
\[
\begin{array}{rcll}
\bd_1 & := & \Image\big(J \cdot \omega_{R_{A^{1/p}}/A^{1/p}} \to \omega_{R_{A^{1/p}}/A^{1/p}} \big) + \Phi_{X/V,1}((J \cdot \omega_{R})^{1/p}) \subseteq \omega_{R_{A^{1/p}}},\\
\bd_2 & := & \sum_{i = 0}^2 \Image\Big((\Phi_{X/V,i}(J \cdot \omega_{R/A})^{1 /p^{i}}) \otimes_{A^{1/p^{i}}} A^{1/p^{2}} \to \omega_{R_{A^{1/p^2}}}\Big)\subseteq \omega_{R_{A^{1/p^2}}},\\
  \dots \\
\bd_n & := & \sum_{i = 0}^n \Image\Big((\Phi_{X/V,i}(J \cdot \omega_{R/A})^{1 /p^{i}}) \otimes_{A^{1/p^{i}}} A^{1/p^{n}} \to \omega_{R_{A^{1/p^n}}}\Big)\subseteq \omega_{R_{A^{1/p^n}}},\\ \\
  \dots
  \end{array}
\]
Notice that $\bd_{1,2} := \Image\Big(\bd_1 \otimes_{A^{1/p}} {A^{1/p^{2}}} \to \omega_{R_{A^{1/p^2}}} \Big)\subseteq \bd_2$ and more generally for $j > i$ that $\bd_{i,j} := \Image\Big(\bd_i \otimes_{A^{1/p^{i}}} {A^{1/p^{j}}} \to \omega_{R_{A^{1/p^n}}} \Big)\subseteq \bd_j$.

By the same argument as in \autoref{sec.RelativeNonFPureIdeals}, we know that there exists an open set $U \subseteq V$ with $W = f^{-1}(U)$ such that $\Image\big(\bd_t \otimes_{A^{1/p^{te}}} A^{1/p^{(t+1)e}} \to \omega_{R_{A^{1/p^{t+1}}}}\big)|_W = \bd_{t+1}|_W$.
\begin{lemma}
\label{lem.StabilizingTauOmega}
With notation as above, $\Image\big(\bd_{n} \otimes_{A^{1/p^{ne}}} A^{1/p^{(n+1)e}} \to \omega_{R_{A^{1/p^{(n+1)e}}}}\big)|_W = \bd_{n+1}|_W$ for all $n > t$.
\end{lemma}
\begin{proof}  The proof is the same as the proof of \autoref{lem.StabilizingTau} and so we omit it.
\end{proof}
%\todo{{\bf Karl:} Write down the proof carefully, then comment it out.\\
%{\bf Wenliang:} I added a proof.}
%We follow the same strategy as in the proof of \autoref{lem.StabilizingTau}. Replacing $V$ by $U$ and $X$ by $W$ , we may assume that
%\[\Image\big(\bd_t \otimes_{A^{1/p^{te}}} A^{1/p^{(t+1)e}} \to \omega_{R_{A^{1/p^{(t+1)e}}}}\big) = \bd_{t+1}.\]
%For each $m$, note that $\Image\big(\bd_m \otimes_{A^{1/p^{me}}} A^{1/p^{(m+1)e}} \to \omega_{R_{A^{1/p^{m+1}}}}\big) = \bd_{m+1}$ if and only if
%\[\Image\Big((\Phi_{X/V,m+1}(J \cdot \omega_{R/A})^{1 /p^{m+1}}) \to \omega_{R_{A^{1/p^{(m+1)e}}}}\Big) \subseteq \Image\big(\bd_{m} \otimes_{A^{1/p^{me}}} A^{1/p^{(m+1)e}} %\to \omega_{R_{A^{1/p^{(m+1)e}}}}\big).\]
%We will induct on $n$. Assume that $\Image\big(\bd_{n} \otimes_{A^{1/p^{ne}}} A^{1/p^{(n+1)e}} \to \omega_{R_{A^{1/p^{(n+1)e}}}}\big) = \bd_{n+1}$. Hence we have
%\[\Image\Big((\Phi_{X/V,n+1}(J \cdot \omega_{R/A})^{1 /p^{n+1}}) \to \omega_{R_{A^{1/p^{(n+1)e}}}}\Big) \subseteq \Image\big(\bd_{n} \otimes_{A^{1/p^{ne}}} A^{1/p^{(n+1)e}} %\to \omega_{R_{A^{1/p^{(n+1)e}}}}\big).\]
%Taking $p^e$-th root and applying the natural map $\omega^{1/p^e}_{R/A}\to \omega_{R/A}$, we can see that
%\[\Image\Big((\Phi_{X/V,n+2}(J \cdot \omega_{R/A})^{1 /p^{n+2}}) \to \omega_{R_{A^{1/p^{(n+2)e}}}}\Big) \subseteq \Image\big(\bd_{n+1} \otimes_{A^{1/p^{(n+1)e}}} %A^{1/p^{(n+2)e}} \to \omega_{R_{A^{1/p^{(n+2)e}}}}\big),\]
%and hence $\Image\big(\bd_{n+1} \otimes_{A^{1/p^{(n+1)e}}} A^{1/p^{(n+2)e}} \to \omega_{R_{A^{1/p^{(n+2)e}}}}\big) = \bd_{n+2}$. This completes our proof.
%\end{proof}

\begin{definition}[Relative test submodules]
With notation as above (in particular, $\omega_{X/V}$ is still compatible with base change), we define the \emph{$n$th iterated relative test submodule} to be $\bd_{n}$ and denote it by $\tau_n(X/V, \omega_{X/V})$.
\end{definition}

We now discuss base change for relative test ideals.

\begin{proposition}[Base change for $\tau_n(X/V, \omega_{X/V})$]
\label{prop.BaseChangeForTauOmega}
Suppose that $g : T \to V$ is a map from a excellent scheme with a dualizing complex, then using the notation of \autoref{subsec.BaseChangeOfPhi}
\[
\Image\left((q_{n})^* \tau_n(X/V, \omega_{X/V}) \rightarrow \omega_{X_{T^{n}}/T^n}\right) = \tau_n(X_T/T, \omega_{X/T}).
\]
Furthermore, if $U = U_n$ satisfies condition from \autoref{lem.StabilizingTauOmega}, then $Y = g^{-1}(U) \subseteq T$ satisfies the same condition for $\tau_n(X_T/T, \omega_{X_T/T})$.
\end{proposition}
\begin{proof}
It is just as before since $\omega_{X/V}$ and $J$ are compatible with arbitrary base change.
\end{proof}

\begin{corollary}[Restriction theorem for $\tau_n(X/V, \omega_{X/V})$]
\label{cor.RestrictionTheoremForTauOmega}
With notation as above, there exists an integer $N > 0$ such that for every perfect point $s \in V$, we have
\[
\Image\left( \tau_n(X/V, \omega_{X/V}) \tensor_{A^{1/p^{ne}}} k(s)^{1/p^{ne}} \to \omega_{X_s/s} \right) = \tau(\omega_{X_s}).
\]
for all $n > N$.
\end{corollary}
\begin{proof}
Taking $g : s \to V$ in \autoref{prop.BaseChangeForTauOmega} and so obtain that the image is equal to $\tau_n(X_s/s, \omega_{X_s/s})$.  Furthermore, since $k(s)$ is perfect, we can make identifications so as to have containments $\ldots \subseteq \tau_n(X_s/s, \omega_{X_s/s}) \subseteq \tau_{n+1}(X_s/s, \omega_{X_s/s}) \subseteq \ldots$ with a ascending union that coincides with $\tau(\omega_{X_s})$.  Furthermore, by \autoref{prop.BaseChangeForSigmaOmega}, over a dense open set $U \subseteq V$ and some $N > 0$ we have $\tau_n(X_s/s, \omega_{X_s/s}) = \tau_{n+1}(X_s/s, \omega_{X_s/s})$ for all $n \geq N_0$ and hence $\tau_n(X_s/s, \omega_{X_s/s}) = \tau(\omega_{X_s})$ by the computation of \autoref{eq.DefiningBBInfinity}.  Let $V_1 = V \setminus U$, base change with $V_1$ and obtain the result over a dense open subset $U_1$ of $V_1$ (for some $N_1$).  By Noetherian induction, this process must terminate.
\end{proof}

\begin{remark}[Relative versus absolute $\tau(\omega)$]
It would be natural to relate the relative $\sigma(\omega)$ and $\tau(\omega)$ with the absolute $\sigma(\omega)$ and $\tau(\omega)$.  While the authors believe that this is possible along the lines of \autoref{subsec.IteratedNonFPureIdealsVsAbsolute} or \autoref{subsec.RelativeTauVsAbsoluteTau}, we won't work out the details here.
\end{remark}

\mysubsection{Applications to families of $F$-injective and $F$-rational singularities }

In this section, we obtain new proofs of results of M.~Hashimoto \cite{HashimotoCMFinjectiveHoms}, deformation of $F$-injectivity and $F$-rationality in proper flat families. Note Hashimoto proved the local version of these results (and in the $F$-rationality case, over a variety).  However, the local results generalize to the non-local case via straightforward computations.
We begin with a definition essentially first made by Hashimoto.

\begin{definition}[\cf \cite{HashimotoCMFinjectiveHoms, HashimotoFPureHoms}]
We say that $X/V$ (which is still assumed to be relatively Cohen-Macaulay) is \emph{relatively $F$-injective} if for some $n > 0$ we have $\sigma_{n}(X/V, \omega_{X/V}) = \omega_{X_{V^n}/V^n}$.  Likewise, $X/V$ is \emph{relatively $F$-rational} if for some $n > 0$ we have $\tau_{n}(X/V, \omega_{X/V}) = \omega_{X_{V^n}/V^n}$.
\end{definition}

\begin{remark}
\label{rem.RelativelyFInjectiveOverPerfectField}
If $V$ is the spectrum of a perfect field $k$, it is easy to see that $f : X \to V$ is relatively $F$-injective (respectively relatively $F$-rational) if and only if $X$ is $F$-injective (respectively $F$-rational) in the usual sense \cite[Section 8]{SchwedeTuckerTestIdealSurvey} via an identification of $k \cong k^{1/p}$.  Note we are implicitly assuming that $X$ is Cohen-Macaulay in this case.
\end{remark}

\begin{lemma}\textnormal{(\cf \cite[Proposition 5.5]{HashimotoCMFinjectiveHoms})}
With notation as above, if $\sigma_{n}(X/V, \omega_{X/V}) = \omega_{X_{V^n}/V^n}$ for some $n > 0$, then $\sigma_{m}(X/V, \omega_{X/V}) = \omega_{X_{V^m}/V^m}$ for all $m > 0$ divisible by $n$.  Furthermore, if $V$ is regular, then the result holds for all $m \geq n$.
Additionally, if $\tau_{n}(X/V, \omega_{X/V}) = \omega_{X_{V^n}/V^n}$ for some $n > 0$, then $\tau_{m}(X/V, \omega_{X/V}) = \omega_{X_{V^m}/V^m}$ for all $m \gg 0$.
\end{lemma}
\begin{proof}
We begin with $\sigma$.  We notice that
\[
\sigma_{2n}(X/V, \omega_{X/V}) = \Image\Big( \Image\big( \omega_{X^n/V^n} \to \omega_{X_V^n/V^n} \big) \otimes_{A^{1/p^{n}}} A^{1/p^{2n}} \to \omega_{X_{V^{2n}}/V^{2n}}  \Big)
\]
and our hypothesis $\sigma_{n}(X/V, \omega_{X/V}) = \omega_{X_{V^n}/V^n}$ implies that the inner map is surjective.  But then the outer map is surjective too by right exactness of tensor and so $\sigma_{2n}(X/V, \omega_{X/V}) = \omega_{X_{V^{2n}}/V^{2n}}$.  The general case repeats this process and so follows similarly.

Now we assume that $V$ is regular for the second statement about $\sigma_n$.  We fix an $n > 0$ such that $\sigma_{n}(X/V, \omega_{X/V}) = \omega_{X_{V^n}/V^n}$.  Then for all $m \leq n$, we have
\[
\sigma_{m}(X/V, \omega_{X/V}) \otimes_{A^{1/p^m}} A^{1/p^n} \supseteq \sigma_{n}(X/V, \omega_{X/V}) = \omega_{X_{V^n}/V^n} \cong \omega_{X_{V^m}/V^m} \otimes_{A^{1/p^m}} A^{1/p^n}
\]
and so $\sigma_{m}(X/V, \omega_{X/V}) \otimes_{A^{1/p^m}} A^{1/p^n} = \omega_{X_{V^m}/V^m} \otimes_{A^{1/p^m}} A^{1/p^n}$.  On the other hand, if we know $\sigma_{m}(X/V, \omega_{X/V}) \subsetneq \omega_{X_{V^m}/V^m}$ then the previous equality is impossible since $A^{1/p^n}$ faithfully flat over $A$ \cite{KunzCharacterizationsOfRegularLocalRings} (since $V$ is regular).

Handling $\tau$ is easy. Again use that $\omega_{X/V}$ is compatible with base change, since
\[
\begin{array}{rcl}
 \omega_{X_V^m/V^m} & = & \Image\big( \omega_{X/V} \otimes_{A^{1/p^n}} A^{1/p^m} \to \omega_{X_V^m/V^m}\big) \\
 &  = & \Image\big(\tau_n(X/V, \omega_{X/V}) \otimes_{A^{1/p^n}} A^{1/p^m} \to \omega_{X_V^m/V^m}\big) \\
& = & \bd_{n,m}\\
& \subseteq & \bd_{m}\\
& = & \tau_n(X/V, \omega_{X/V})\\
&  \subseteq & \omega_{X_V^m/V^m}
\end{array}
\]
\end{proof}

\begin{theorem} \textnormal{(Deformation of $F$-rationality and $F$-injectivity, \cf \cite[Theorem 5.8, Remark 6.7]{HashimotoCMFinjectiveHoms})}
\label{thm.OpenSetNakayama}
Suppose that $f : X \to V$ is a proper flat finite type equidimensional reduced Cohen-Macaulay morphism to an excellent integral scheme $V$ with a dualizing complex.  Suppose that for some point $s \in V$, the fiber $X_s/s$ is relatively $F$-injective (respectively, $F$-rational).  Then there exists an open neighborhood $U \subseteq V$ containing $s$ such that $X_u \to u$ is relatively $F$-injective (respectively, $F$-rational) for all $u \in U$.
\end{theorem}
\begin{proof}
We first show that $\sigma_n(\omega_{X/V}) = \omega_{X_{V^{ne}}/V^{ne}}$ at each point of the fiber $X_s$.  Indeed, let $z \in X$ be a point on $X_s$ and let $I$ denote the ideal sheaf of $X_s$.  We observe that for some $n$, the natural map $\sigma_n(\omega_{X/V})/(I \cdot \sigma_n(\omega_{X/V}))  \to \omega_{X/V} / (I \cdot \omega_{X/V}) = \omega_{X_s/V_s}$ is surjective.  This is preserved after localizing at $z$, and so since $I \subseteq \bm_z$, we see that the generators of the stalk $(\sigma_n(\omega_{X/V}))_z$ generate $( \omega_{X/V})_z$ by Nakayama's lemma.  Hence $(\sigma_n(\omega_{X/V}))_z = ( \omega_{X/V})_z$.  Since this holds at every point $z \in X_s \subseteq X$, it holds in a neighborhood of $X_s$.

As before, let $Z$ denote the locus where $\sigma_n(\omega_{X/V}) \neq \omega_{X_{V^{ne}}/V^{ne}}$.  This is closed and since $f$ is proper, its image $f(Z)$ is closed too.  But $f(Z)$ is then a closed set not containing $s$.  The result follows for $F$-injectivity.  The proof for $F$-rationality is the same.
\end{proof}

%% file: GlobalStuff.tex
\section{Global applications}
\label{sec:global}

The purpose of this section is to develop a global theory of the previous sections for a projective family $f : X \to V$.  In the last few years, there has been a new push to use Frobenius and the trace map to replace the Kodaira vanishing theorem.   In this section we extend some of these ideas to families. We study how the canonical linear subsystems $S^0(X_s,\sigma(X_s,\Delta_s) \otimes M_s) \subseteq H^0(X_s, M_s)$, introduced in \cite{SchwedeACanonicalLinearSystem}, behave as we vary $s \in V$. Furthermore, as mentioned in the introduction we also obtain some global generation and semi-positivity statements.

%\begin{remark}
%\label{rem.CanWeUseSectionRings}
%In a number of results throughout this section, we assume that various divisors on $X$ are relatively ample over $V$.  In several of these cases when dealing with large multiples of %the given divisor so that certain cohomologies are compatible with base change, we believe the results contained herein could also be obtained by forming section rings on the base %with respect to these divisors and then applying results of previous sections.
%\end{remark}

\subsection{Basic definitions}

We use the following setup throughout \autoref{sec:global}.

\begin{notation}
\label{notation:S_f_*}
In the situation of Notation \ref{notation:basic}, assume also that $f : X \to V$ is projective and additionally that $V$ is regular (which implies that $F^e_V : V^e \to V$ is flat). Furthermore, fix  a line bundle $M$ on $X$. Sometimes we also assume the following (in which case we write \autoref{notation:S_f_*}\textsuperscript{*}):

(*):  there is an integer $N \geq 0$, such that for every integer $m \geq N$,
\begin{equation*}
 \sigma_m(X/V,\phi) = \sigma_N(X/V,\phi) \otimes_{A^{\frac{1}{p^{Ne}}}} A^{\frac{1}{p^{me}}} .
\end{equation*}
In this situation  we denote $\sigma_N(X/V,\phi)$ by $\sigma_N$.  We notice that this condition (*) always holds over a dense open set of the base by \autoref{prop.StabilizingSigma}.
\end{notation}

\begin{remark}
Note that since $X$ (resp. $V$) is topologically isomorphic to $X^{ne}_{V^{me}}$ (resp. $V^{me}$) for every $m \geq n$,  $f_*$ can be identified with $g_*$, where $g$ is any of the induced morphisms $X^{ne}_{V^{me}} \to V^{me}$. Hence, we use only $f_*$ for all purposes, even when $g_*$ for one of the above maps $g : X^{ne}_{V^{me}} \to V^{me}$ would be more natural. The downside of this notation is that it does not show if a sheaf has a $A^{1/p^{me}}$ structure, and consequently its pushforward by $f$  a $\sO_{V^{me}}$ structure. We decided to still use it, because it simplifies greatly the  notations.
% Note, since all maps of the form $V^{ne} \to V^{me}$ and of $V^{ne} \to V^{\infty}$ are flat, $f_*$ commutes with tensoring with $A^{1/p^{me}}$ and $\sO_{V^{me}}$ on the top and bottom respectively.

Summarizing: when reading the following arguments it is important to trace through the space $X^{ne}_{V^{me}}$ on which the adequate sheaves live. Then $f_*$ of these sheaves will live on $V^{me}$.
\end{remark}

\begin{notation}
When dealing with sheaves $G$ on $V^{ne}$, we will frequently pull them back to $V^{me}$ for $m \geq n$.  When doing this, instead of writing $\O_{V^{me}} \otimes_{\O_{V^{ne}}} G$ or $\O_{V}^{1/p^{me}} \otimes_{\O_{V}^{1/p^{ne}}} G$, we will write simply $V^{me} \times_{V^{ne}} G$ or $G_{V^{me}}$.  We trust this abuse of notation will cause no confusion, as it helps compactify the notation substantially.
\end{notation}

Our approach to understanding how the canonical linear systems of \cite{SchwedeACanonicalLinearSystem} behave in families is to define a relative version of them, and then show certain base-change properties. These objects will be relative versions of $\sigma$ for global sections, and they play the same role for $\sigma_n$ that $T^0$ and $S^0$ plays for $\tau$ and $\sigma$ in \cite{BlickleSchwedeTuckerTestAlterations} and \cite{SchwedeACanonicalLinearSystem} respectively.

\begin{definition}
\label{defn:S_f_*}
In the situation of \autoref{notation:S_f_*}, define
\begin{equation*}
S_{\varphi^n}^0 f_*(M):=
\im \left( f_* \left( \left( L^{\frac{p^{ne} -1}{p^e -1}} \right)^{\frac{1}{p^{ne}}} \otimes_R M \right)
\xrightarrow{f_* \left(\varphi^n \otimes_R \id_{M} \right)}
f_*  \left( A^{\frac{1}{p^{ne}}} \otimes_A M \right) \right) .
\end{equation*}
Note that $S_{\varphi^n}^0 f_*(M)$ is a sheaf on ${V^{ne}}$ and it is a subsheaf of $(f_*(M))_{V^{ne}}$ by flat base-change. In case $(X, \Delta)$ is a pair we define $S_{\Delta,ne}^0 f_*(M):= S_{\varphi_\Delta^n}^0 f_*(M)$ (assuming $(p^{ne}-1)(K_X + \Delta)$ is Cartier and $\phi$ is the corresponding map). If $\Delta = 0$, then we write $S_{ne}^0 f_*(M)$ for $S_{\Delta,ne}^0 f_*(M)$.
\end{definition}

Since the image of $\varphi^n \otimes_R \id_{M} $ in the above definition is $\sigma_n(X/V, \phi) \otimes_R M$ the following proposition is immediate.

\begin{proposition}
\label{prop:obvious_containment}
In the situation of \autoref{notation:S_f_*},
\begin{equation*}
S_{\varphi^n}^0 f_*(M) \subseteq f_* (\sigma_n(X/V, \phi) \otimes_R M ) .
\end{equation*}
\end{proposition}

Our first goal is to show that the images in \autoref{defn:S_f_*} descend (up to appropriate base change by Frobenius).  Compare with the containments $\ba_{i,n} \supseteq \ba_{i+1,n} \supseteq \dots \supseteq \ba_{n,n}$ of \autoref{sec.RelativeNonFPureIdeals}.

\begin{proposition}
\label{prop:S_pushforward_containment}
For all integers $m \geq n \geq 0$,
\begin{equation}
\label{eq:S_pushforward_containment:statement}
 V^{me} \times_{V^{ne}} S^0_{\varphi^n} f_* (M)  \supseteq S^0_{\varphi^m} f_* (M)
\end{equation}
as subsheaves of $V^{me} \times_V f_* (M)$.
\end{proposition}

\begin{remark}
To be precise the left and right hand side of \autoref{eq:S_pushforward_containment:statement} are subsheaves of $$V^{me} \times_{V^{ne}} f_* (A^{\frac{1}{p^{ne}}} \otimes_A M) \;\;\text{ and of }\;\; f_* (A^{\frac{1}{p^{me}}} \otimes_A M),$$ respectively. However, both $V^{me} \times_{V^{ne}} f_* (A^{\frac{1}{p^{ne}}} \otimes_A M)$ and  $f_* (A^{\frac{1}{p^{me}}} \otimes_A M)$ are canonically isomorphic to $V^{me} \times_V f_* (M)$ via flat base-change (since $V$ is regular).
\end{remark}

\begin{proof}[Proof of \autoref{prop:S_pushforward_containment}]
The following commutative diagram shows that $f_* \left(\varphi^m \otimes_R  \id_{M} \right)$ factors through $V^{me} \times_{V^{ne}} f_* \left(\varphi^n \otimes_R \id_{M} \right)$.
\begin{equation*}
%\xymatrixrowsep{32pt}
\xymatrix@R=32pt{
 V^{me} \times_{V^{ne}} f_* \left( \left( L^{\frac{p^{ne} -1}{p^e -1}} \right)^{\frac{1}{p^{ne}}} \otimes_R M \right)
 \ar@{}[d]^(0.5){\begin{rotate}{-90}$\cong$\end{rotate} \quad \raisebox{-6pt}{\parbox{85pt}{\tiny flat base-change \cite[Proposition III.9.3]{Hartshorne} of $f_*$ by $V^{me} \to V^{ne}$}}}
\ar[rr]^(0.55){V^{me} \times_{V^{ne}} f_* \left(\varphi^n \otimes_R \id_{M} \right)}
& \hspace{15pt}
& V^{me} \times_{V^{ne}}  f_*  \left(  A^{\frac{1}{p^{ne}}}  \otimes_A M \right)
\ar@{}[d]^(0.4){\begin{rotate}{-90}$\cong$\end{rotate} \quad \raisebox{-6pt}{\parbox{40pt}{\tiny  flat base-change}}} \\
%
%%%%%%%%%%%%%%%%%%%%%%%%%%%%%%%%%%%%%%%%%%%%%%%%%%%%%%%%%%%%%
%
f_* \left( A^{\frac{1}{p^{me}}} \otimes_{A^{\frac{1}{p^{ne}}}}  \left(  L^{\frac{p^{ne} -1}{p^e -1}} \right)^{\frac{1}{p^{ne}}} \otimes_R M \right)
\ar@{}[d]^(0.5){\begin{rotate}{-90}$\cong$\end{rotate} \quad \raisebox{-6pt}{}}%{\parbox{80pt}{\tiny   flat base-change of the relative Frobenius by $V^{me} \to V^{ne}$}}}
\ar[rr]^(0.59){f_* \left(A^{\frac{1}{p^{me}}} \otimes_{A^{\frac{1}{p^{ne}}}} \left(\varphi^n \otimes_R \id_{M} \right) \right)}
%h
& &     f_* \left(  A^{\frac{1}{p^{me}}} \otimes_A M \right)
\ar@{}[d]^(0.4){\begin{rotate}{-90}$=$\end{rotate} } \\
%
%%%%%%%%%%%%%%%%%%%%%%%%%%%%%%%%%%%%%%%%%%%%%%%%%%%%%%%%%%%%%
%
f_* \left(  \left( R_{A^{\frac{1}{p^{(m-n)e}}}} \otimes_{R}  L^{\frac{p^{ne} -1}{p^e -1}} \right)^{\frac{1}{p^{ne}}} \otimes_R M \right)
\ar[rr]%^(0.61){f_*  \left( \left(A^{\frac{1}{p^{me}}} \otimes_{A^{\frac{1}{p^{ne}}}} \varphi \right)^n \otimes_R  \id_{M} \right)}
& &      f_* \left(    A^{\frac{1}{p^{me}}} \otimes_A M \right) \\
%
%%%%%%%%%%%%%%%%%%%%%%%%%%%%%%%%%%%%%%%%%%%%%%%%%%%%%%%%%%%%%
%
f_*   \left( \left( \left(  L^{\frac{p^{(m-n)e}-1}{p^e-1}} \right)^{\frac{1}{p^{(m-n)e}}} \otimes_R L^{\frac{p^{ne} -1}{p^e -1}}  \right)^{\frac{1}{p^{ne}}} \otimes_R  M \right)
\ar[u]_(0.55){f_* \left(    \left(\phi^{m-n} \otimes_R \id_{L^{\frac{p^{ne} -1}{p^e -1}}} \right)^{\frac{1}{p^{ne}}} \otimes_R \id_{M} \right) }
 \ar@{}[d]^(0.4){\begin{rotate}{-90}$\cong$\end{rotate} \quad \raisebox{-6pt}{\parbox{80pt}{\tiny projection formula}}}
& & \\ % \ar[ur]_{\hspace{50pt}(f_{V^{me}} )_* (\varphi^m \otimes \id_{M_{V^{me}}})} &  \\
%
%
%%%%%%%%%%%%%%%%%%%%%%%%%%%%%%%%%%%%%%%%%%%%%%%%%%%%%%%%%%%%%
%
f_* \left(   \left( L^{\frac{p^{me} -1}{p^e -1}} \right)^{\frac{1}{p^{me}}}  \otimes_R  M \right)
\ar`r[rru][uurr]_{f_* \left(\varphi^m \otimes_R  \id_{M} \right) } & &
}
\end{equation*}
Hence, the statement of the proposition holds by the following computation.
\begin{equation*}
\begin{split}
S_{\phi^m}^0 f_* (M) &  = \im \left( f_* \left(\varphi^m \otimes_R  \id_{M} \right) \right)
\\ & \subseteq  \im \left( V^{me} \times_{V^{ne}}  f_* \left(\varphi^n \otimes_R \id_{M} \right) \right)
\\ & =   \underbrace{V^{me} \times_{V^{ne}} \im \left( f_* \left(\varphi^n \otimes_R \id_{M} \right) \right)}_{\textrm{$V^{ne} \to V^{me}$ is flat}} .
\\ & = V^{me} \times_{V^{ne}} S_{\phi^n}^0 f_* (M)
\end{split}
\end{equation*}
\end{proof}

\subsection{Auxiliary definition and stabilization}
We would now like to obtain a global result similar to \autoref{prop.StabilizingSigma}.  In particular, we'd like the containments of \autoref{prop:S_pushforward_containment} to be equality over a dense open subset subset of the base $V$.  There is a complicating factor however, while we can still find an open set $U$ of $V$ such that
\[
\left( V^{(n+1)e} \times_{V^{ne}} S^0_{\varphi^n} f_* (M)\right)\Big|_U = \left( S^0_{\varphi^{n+1}} f_* (M) \right)\Big|_U
\]
we do not see how to use this to show that we have the $n+1$ to $n+2$ equality without additional assumptions.  The issue is that in the proof of \autoref{prop.StabilizingSigma}, twisting by line bundles is exact. For $S_{\phi^n}^0$ however we also push forward.  Therefore, in order to obtain our stabilization over a dense open set of the base we need additional positivity assumptions on $M$ and $L$.

Furthermore, we need an auxiliary version of $S_{\phi}^0$ which involves the $\sigma_N := \sigma_N(X/V,\phi)$ from \autoref{notation:S_f_*}*.  To do this, first observe that since $f$ is flat, the tensor product $L^{\frac{p^{me} -1}{p^e -1}} \otimes_R   \sigma_N$ is naturally identified with a subsheaf of $L^{\frac{p^{me} -1}{p^e -1}} \otimes_R   R_{A^{1/p^{Ne}}} \cong L^{\frac{p^{me} -1}{p^e -1}} \otimes_A   {A^{1/p^{Ne}}} $.  In order to motivate this auxiliary definition, we make the following observation:

\begin{lemma}
\label{lem.ImageStabilizesSigmaN}
In the situation of \autoref{notation:S_f_*}\textsuperscript{*}, we have that the image of the natural map
\[
\alpha_{N+m} : \left(   L^{\frac{p^{me} -1}{p^e -1}} \otimes_R   \sigma_N \right)^{\frac{1}{p^{me}}} \hookrightarrow \left(   L^{\frac{p^{me} -1}{p^e -1}} \otimes_R   R_{A^{1/p^{Ne}}} \right)^{\frac{1}{p^{me}}} \xrightarrow{\phi^m \otimes_{A^{1/p^{me}}} A^{1/p^{(N+m)e}} }  R_{A^{1/p^{(N+m)e}}}
\]
is equal to $\sigma_N \otimes_{A^{1/p^{Ne}}} A^{1/p^{(N+m)e}} \cong \sigma_{m+N}(X/V,\phi)$.
\end{lemma}
\begin{proof}
The first term is the image of the $p^{me}$th root of $ L^{\frac{p^{me} -1}{p^e -1}} \otimes_{R} \phi^N $.  Therefore the image of the composition also equals $\sigma_{m+N}(X/V,\phi)$ by our construction of $\phi^j$ in \autoref{subsec.ComposingMaps}.
\end{proof}

This stabilization suggests it is reasonable to make the following definition.

\begin{definition}
\label{defn:S_0_f_auxilliary}
In the situation of \autoref{notation:S_f_*}\textsuperscript{*},
\begin{multline*}
S_{\varphi^n,\sigma_N}^0 f_*(M):=
\im \left(\rule{0cm}{1cm} f_* \left( \left(   L^{\frac{p^{me} -1}{p^e -1}} \otimes_R   \sigma_N \right)^{\frac{1}{p^{me}}} \otimes_R  M \right) \xrightarrow{f_* (\alpha_{N+m} \otimes_R M)}  f_*  \left( A^{\frac{1}{p^{(n+N)e}}} \otimes_A M \right)  \rule{0cm}{1cm}\right)
%
%\\ \to %\xrightarrow{f_* \left( \left( \left. \varphi^n \otimes_{A^{\frac{1}{p^{me}}}} A^{\frac{1}{p^{(N+m)e}}} \right|_{L^{\frac{p^{me} -1}{p^e -1}} \otimes_R   \sigma_N} \right) \otimes_R \id_{M} \right)}
%
% \left. f_*  \left( A^{\frac{1}{p^{(n+N)e}}} \otimes_A M \right)  \rule{0cm}{1cm}\right) .
\end{multline*}
where $\alpha_{N+m}$ is as in \autoref{lem.ImageStabilizesSigmaN}.
\end{definition}

Consider the following proposition relating the various $S^0 f_*$ objects so far described.

\begin{proposition}
\label{prop:S_f_*_sigma_not_sigma_containments}
In the situation of \autoref{notation:S_f_*}\textsuperscript{*}, for every integer $m \geq 0$,
\begin{equation*}
V^{(N+m)e} \times_{V^{me}} S_{\varphi^m}^0 f_*(M) \supseteq S_{\varphi^m,\sigma_N}^0 f_*(M) \supseteq S_{\varphi^{m+N}}^0 f_*(M)
\end{equation*}
(Here all sheaves are regarded as subsheaves of $V^{(N+m)e} \times_{V} f_*(M)$ via flat base-change.)
\end{proposition}

\begin{proof}
This follows directly from the definitions.  The first containment is trivial, and the second follows from a factorization similar to the one from \autoref{prop:S_pushforward_containment}.
\end{proof}

\begin{proposition}
\label{prop:relative_sigma_pushforward_stabilizes}
In the situation of \autoref{notation:S_f_*}\textsuperscript{*},
\begin{enumerate}
\item \label{itm:relative_sigma_pushforward_stabilizes:containment} for all integers $m \geq n \geq 0$,
\begin{equation*}
 V^{(N + m)e} \times_{V^{(N + n)e}} S^0_{\varphi^n,\sigma_N} f_* (M)  \supseteq S^0_{\varphi^m,\sigma_N} f_* (M),
\end{equation*}
as subsheaves of $V^{(N+m)e} \times_{V} f_*(M)$, and
\item \label{itm:relative_sigma_pushforward_stabilizes:equality} if furthermore $L \otimes M^{p^e -1}$ is $f$-ample, then there is an integer $n>0$, such that for every integer $m \geq n$, the above inclusion  is equality.
\item \label{itm:relative_sigma_pushforward_stabilizes:big_power_ample} if furthermore $ M= Q^l \otimes P$, where $Q$ is $f$-ample, then there is an integer $l_0>0$, such that for  every integer $l \geq l_0$, $m \geq n \geq 0$ and nef line bundle $P$ the above inclusion  is equality.
\end{enumerate}
\end{proposition}

\begin{proof}
  Define the coherent sheaf $B_{\phi}$  on $X_{V^{(N+1)e}}$ as the kernel of the  top horizontal map in the following commutative diagram.
\begin{equation}
\label{eq:relative_sigma_pushforward_stabilizes:B_phi}
\xymatrix{
B_{\phi} \ar@{^(->}[r] & ( L \otimes_R \sigma_N)^{\frac{1}{p^e}} \ar@{->>}[r]^-{\phi \otimes } &  A^{\frac{1}{p^{(N+1)e}}} \otimes_{A^{\frac{1}{p^{Ne}}}} \sigma_N \\
& \left( L \otimes \left( L^{\frac{p^{Ne}-1}{p^e -1}} \right)^{\frac{1}{p^{Ne}}} \right)^{\frac{1}{p^e}} \ar@{->>}[u]^{\left(\id_L \otimes_R \phi^{N} \right)^{\frac{1}{p^e}}} \ar@{}[r]|{\cong}
& \left( L^{\frac{p^{(N+1)e}-1}{p^e -1}} \right)^{\frac{1}{p^{(N+1)e}}}  \ar@{->>}[u]_-{\phi^{N+1}}
}
\end{equation}
Note here we used that the image of  $\phi^{N+1}$ is $A^{\frac{1}{p^{(N+1)e}}} \otimes_{A^{\frac{1}{p^{Ne}}}} \sigma_N$ by the assumptions made in \autoref{notation:S_f_*}\textsuperscript{*}. Also, the horizontal arrow is surjective by either \autoref{lem.ImageStabilizesSigmaN} or from the diagram (which is how we proved \autoref{lem.ImageStabilizesSigmaN}).  Then one can apply $f_* \left( \left( L^{\frac{p^{me}-1}{p^e-1}} \otimes_R \underline{\qquad} \right)^{\frac{1}{p^{me}}} \otimes_R M \right)$ to the top row of \autoref{eq:relative_sigma_pushforward_stabilizes:B_phi}, which is shown in the following commutative diagram. We have also included important isomorphisms to the different terms of the exact sequence.
\begin{equation*}
\xymatrix@R=16pt{
%0 \ar[d] \\
%
%% 2nd row %%%%%%%%%%%%%%%%%%%%%%%%%%%%%%%%%%%%%%%%%%%%%%%%%%%%%%%%%%%%%%%%%%%%%%%%%%%%%%%%%%%%%%%%%%%%%%%%%%%%%%%%%%%%%%%
%
f_* \left( \left( L^{\frac{p^{me}-1}{p^e-1}} \otimes_R B_{\phi} \right)^{\frac{1}{p^{me}}} \otimes_R M \right) \ar@{^{(}->}[d]
&
\\
%
%% 3rd row %%%%%%%%%%%%%%%%%%%%%%%%%%%%%%%%%%%%%%%%%%%%%%%%%%%%%%%%%%%%%%%%%%%%%%%%%%%%%%%%%%%%%%%%%%%%%%%%%%%%%%%%%%%%%%%
%
f_* \left( \left( L^{\frac{p^{me}-1}{p^e-1}} \otimes_R ( L \otimes \sigma_N)^{\frac{1}{p^e}} \right)^{\frac{1}{p^{me}}} \otimes_R M \right) \ar@{}[r]|{\cong} \ar[d]_{\nu}
&
f_* \left( \left(  L^{\frac{p^{(m+1)e}-1}{p^e-1}} \otimes_R  \sigma_N  \right)^{\frac{1}{p^{(m+1)e}}} \otimes_R M \right)  \ar@{->>}[ddddd]
\\
%
%% 4th row %%%%%%%%%%%%%%%%%%%%%%%%%%%%%%%%%%%%%%%%%%%%%%%%%%%%%%%%%%%%%%%%%%%%%%%%%%%%%%%%%%%%%%%%%%%%%%%%%%%%%%%%%%%%%%%
%
f_* \left( \left( L^{\frac{p^{me}-1}{p^e-1}} \otimes_R \left(\sigma_N \otimes_{A^{\frac{1}{p^{Ne}}}} A^{\frac{1}{p^{(N+1)e}}}\right) \right)^{\frac{1}{p^{me}}} \otimes_R M \right) \ar@/^15pc/@{<->}[dd]^{\cong} \ar[d]
&
\\
%
%% 5th row %%%%%%%%%%%%%%%%%%%%%%%%%%%%%%%%%%%%%%%%%%%%%%%%%%%%%%%%%%%%%%%%%%%%%%%%%%%%%%%%%%%%%%%%%%%%%%%%%%%%%%%%%%%%%%%
%
R^1 f_* \left( \left( L^{\frac{p^{me}-1}{p^e-1}} \otimes_R B_{\phi} \right)^{\frac{1}{p^{me}}} \otimes_R M \right) &
\\
%
%% 6th row %%%%%%%%%%%%%%%%%%%%%%%%%%%%%%%%%%%%%%%%%%%%%%%%%%%%%%%%%%%%%%%%%%%%%%%%%%%%%%%%%%%%%%%%%%%%%%%%%%%%%%%%%%%%%%%
%
f_* \left( \left( \left( L^{\frac{p^{me}-1}{p^e-1}} \otimes_R \sigma_N  \right)^{\frac{1}{p^{me}}} \otimes_{A^{\frac{1}{p^{(N+m)e}}}} A^{\frac{1}{p^{(N+m+1)e}}}\right) \otimes_R M \right) \ar@{}[d]|{\begin{rotate}{-90}$\cong$\end{rotate} \quad \raisebox{-6pt}{\textrm{\tiny by flat base change}}}
\\
%
%% 7th row %%%%%%%%%%%%%%%%%%%%%%%%%%%%%%%%%%%%%%%%%%%%%%%%%%%%%%%%%%%%%%%%%%%%%%%%%%%%%%%%%%%%%%%%%%%%%%%%%%%%%%%%%%%%%%%
%
V^{(N+m+1)e} \times_{V^{(N+m)e}} f_* \left( \left( L^{\frac{p^{me}-1}{p^e-1}} \otimes_R \sigma_N  \right)^{\frac{1}{p^{me}}}  \otimes_R M \right) \ar@{->>}[d]
\\
%
%% 8th row %%%%%%%%%%%%%%%%%%%%%%%%%%%%%%%%%%%%%%%%%%%%%%%%%%%%%%%%%%%%%%%%%%%%%%%%%%%%%%%%%%%%%%%%%%%%%%%%%%%%%%%%%%%%%%%
%
V^{(N+m+1)e} \times_{V^{(N+m)e}} S^0_{\phi^m,\sigma_N} f_*(M) & \ar[l] S^{0}_{\phi^{m+1},\sigma_N} f_*(M)
}
\end{equation*}
Both $V^{(N+m+1)e} \times_{V^{(N+m)e}} S^0_{\phi^m,\sigma_N} f_*( M)$ and $S^{0}_{\phi^{m+1},\sigma_N} f_*(M)$ can be regarded as subsheaves of $V^{(N+m+1)e} \times_{V}  f_* (M)  $ via flat base change. Furthermore, the bottom horizontal arrow becomes a map of subsheaves, i.e., an injection, this way.  This shows point \autoref{itm:relative_sigma_pushforward_stabilizes:containment}.

To prove point \autoref{itm:relative_sigma_pushforward_stabilizes:equality}, we are supposed to prove that whenever $L \otimes M^{p^e -1}$ is $f$-ample, this arrow is also surjective for $m \gg 0$. This would  follow if the third vertical arrow (from above), labeled $\nu$, was surjective. Therefore, it is sufficient to show that for $m \gg 0$,
\begin{multline*}
0
= R^1 f_* \left( \left( L^{\frac{p^{me}-1}{p^e-1}} \otimes_R B_{\phi} \right)^{\frac{1}{p^{me}}} \otimes_R M \right)
= R^1 f_* \left( \left( L^{\frac{p^{me}-1}{p^e-1}} \otimes_R B_{\phi} \otimes_R M^{p^{me}} \right)^{\frac{1}{p^{me}}}  \right)
\\
= R^1 f_* \left( \left(  \left(L \otimes M^{p^e -1}  \right)^{\frac{p^{me}-1}{p^e-1}} \otimes_R B_{\phi} \otimes_R M \right)^{\frac{1}{p^{me}}}  \right)       .
\end{multline*}
Note that since applying $( \underline{\quad})^{\frac{1}{p^{me}}}$ does not change the sheaf of abelian groups structure, this is equivalent to showing
\begin{multline*}
 0 = R^1 f_*  \left(  \left(L \otimes M^{p^e -1}  \right)^{\frac{p^{me}-1}{p^e-1}} \otimes_R B_{\phi} \otimes_R M   \right)
\\ = R^1 f_*  \left(  \left(L \otimes M^{p^e -1} \otimes_A A^{\frac{1}{p^{N+1}}}  \right)^{\frac{p^{me}-1}{p^e-1}} \otimes_{R_{A^{1/p^{(N+1)e}}}}  \left( B_{\phi} \otimes_R M \right)   \right)       .
\end{multline*}
Furthermore, $L \otimes M^{p^e -1} \otimes_A A^{\frac{1}{p^{N+1}}}$ is a relatively ample line bundle by the assumption of point \autoref{itm:relative_sigma_pushforward_stabilizes:equality} and $B_{\phi} \otimes_R M$ is a coherent sheaf on $X_{V^{(N+1)e}}$. Therefore, relative Serre vanishing concludes our proof.

Point \autoref{itm:relative_sigma_pushforward_stabilizes:big_power_ample} follows immediately from the above argument. Indeed, if $M= Q^l \otimes P$, then by relative Fujita vanishing \cite[Theorem 1.5]{KeelerAmpleFilters} there is an integer $l_0>0$, such that the above vanishing holds for every $m \geq 0$, $l \geq l_0$ and nef line bundle $P$.
\end{proof}

\begin{corollary}
\label{prop:pushforward_stabilizes}
In the situation of \autoref{notation:S_f_*}\textsuperscript{*}, if $L  \otimes M^{p^e-1}$ is $f$-ample, then
there is an integer $n>0$, such that for every integer $m \geq n$
\begin{equation*}
 V^{me} \times_{V^{ne}} S^0_{\varphi^n} f_* (M)  = S^0_{\varphi^m} f_* (M) \quad  \left(= S^0_{\varphi^{m-N}, \sigma_N} f_* (M) \right)
\end{equation*}
as subsheaves of $V^{me} \times_{V} f_*(M)$. Furthermore in the situation of point \autoref{itm:relative_sigma_pushforward_stabilizes:big_power_ample} of \autoref{prop:relative_sigma_pushforward_stabilizes} where $M = Q^l \otimes P = \text{\textnormal{(ample)}$^l \otimes $\textnormal{(nef)}}$ with $l \gg 0$, we can pick $n=N$.
\end{corollary}

\begin{proof}
By \autoref{prop:S_f_*_sigma_not_sigma_containments},
\begin{equation}
\label{eq:pushforward_stabilizes:two_inclusions}
V^{(m+2N)e} \times_{V^{(m+N)e}} S_{\varphi^m,\sigma_N}^0 f_*(M) \supseteq V^{(m+2N)e} \times_{V^{(m+N)e}} S_{\varphi^{m+N}}^0 f_*(M) \supseteq S^0_{\varphi^{m+N},\sigma_N} f_*(M).
\end{equation}
Furthermore, by Proposition \ref{prop:relative_sigma_pushforward_stabilizes}.\ref{itm:relative_sigma_pushforward_stabilizes:equality}, we know the statement holds when  $S^0_{\varphi^m} f_* (M)$ is replaced by $ S^0_{\varphi^m,\sigma_N} f_* (M)$. Hence the inclusion of the right end of \autoref{eq:pushforward_stabilizes:two_inclusions} to the left end is an equality for $m \gg 0$. However, then all inclusions in \autoref{eq:pushforward_stabilizes:two_inclusions} are equalities. In particular from the first equality of \autoref{eq:pushforward_stabilizes:two_inclusions}, using that $V^{(m+2N)e} \to V^{(m+N)e}$ is faithfully flat, we obtain that $ S_{\varphi^m,\sigma_N}^{0} f_*(M) = S_{\varphi^{m+N}} f_*(M)$ for $m \gg 0$. Using Proposition \ref{prop:relative_sigma_pushforward_stabilizes}.\ref{itm:relative_sigma_pushforward_stabilizes:equality} once more concludes our proof of the main statement.  The final statement is similar.
\end{proof}

We can now obtain our promised analog of \autoref{prop.StabilizingSigma}.

\begin{theorem}
In the situation of \autoref{notation:S_f_*}, if $L \otimes M^{p^e - 1}$ is ample, there exists an integer $n_0$ and a dense open set $U \subseteq V$ such that for all $m \geq n \geq n_0$ we have that
\[
\Big( V^{me} \times_{V^{ne}} S^0_{\varphi^n} f_* (M)\Big) \Big|_U  =  \Big(S^0_{\varphi^m} f_* (M)\Big) \Big|_U.
\]
\end{theorem}
\begin{proof}
By shrinking $V$ and applying  \autoref{prop.StabilizingSigma} we can reduce to the case of \autoref{notation:S_f_*}\textsuperscript{*}.  The result follows directly from \autoref{prop:pushforward_stabilizes}.
\end{proof}

\begin{corollary}
\label{cor:S_f_if_big_power_of_rel_ample}
In the situation of \autoref{notation:S_f_*},  there is a natural inclusion
\begin{equation*}
 S^0_{\phi^n} f_* (M) \subseteq f_* (\sigma_n(X/V, \phi) \otimes_R M).
\end{equation*}
as subsheaves of $V^{ne} \times_{V} f_*(M)$.

Further, in the situation of \autoref{notation:S_f_*}\textsuperscript{*}, if  $M=Q^l \otimes P$ where $Q$ is $f$-ample, then there is an integer $ l_0>0$, such that for every integer $l \geq l_0$, $n \geq N$ and  $f$-nef line bundle $P$, the above inclusion is equality.
% \begin{equation*}
% S^0_{\phi^n} f_* (M)=f_* \left(\sigma_{n}(X/V, \phi) \otimes_R M \right)  \qquad \left( = f_* \left(\sigma_N(X/V, \phi) \otimes_R M \right) \times_{V^{Ne}} V^{ne} \right)
% \end{equation*}
\end{corollary}

\begin{proof}
Consider the surjection.
\begin{equation*}
\xymatrix{
\xi : \left( L^\frac{p^{ne} -1}{p^e -1}   \right)^{\frac{1}{p^{ne}}}
\ar@{->>}[rrr]^(0.6){\phi^n}
& & &
\sigma_{n}(X/V,\phi)
}
\end{equation*}
Then $S^0_{\phi^n} f_* (M)= \im f_*(\xi \otimes_R \id_M)$, which is a subsheaf of $f_* (\sigma_n(X/V,\phi) \otimes_R M)$.

We prove the addendum by induction on $n$. If $n=N$,
\begin{equation*}
S^0_{\phi^N} f_* (M)= \underbrace{S^0_{\phi^0, \sigma_N} f_*(M)}_{\textrm{by \autoref{prop:pushforward_stabilizes}}}
= \underbrace{f_* (\sigma_N \otimes_R M)}_{\textrm{by \autoref{defn:S_0_f_auxilliary}}}.
\end{equation*}
Let us assume then that $n>N$. By flat base change and the assumption of \autoref{notation:S_f_*}\textsuperscript{*}
\begin{equation*}
f_* (\sigma_n(X/V) \otimes_R M) = f_* (\sigma_N(X/V) \otimes_R M) \times_{V^{Ne}} V^{ne}.
\end{equation*}
Hence, using the inductional hypothesis, it is enough to see that for every $n \geq N$,
\begin{equation*}
S^0_{\phi^n} f_* (M) = S^0_{\phi^{n-1}} f_* (M).
\end{equation*}
However, this follows from  point \autoref{itm:relative_sigma_pushforward_stabilizes:big_power_ample} of \autoref{prop:relative_sigma_pushforward_stabilizes} and \autoref{prop:pushforward_stabilizes}
\end{proof}

% \begin{remark}
% By \autoref{prop:pushforward_stabilizes}, $S^m_{\varphi} f_* (M)$ can be thought of almost one sheaf for $m \geq n$. Although, it is a different sheaf for each $m$ (it lives on different spaces $V^{me}$), it is the pullback of a single sheaf $S^n_\varphi f_* (M)$. We call this the stabilization $S^-_{\varphi} f_* (M)$ of $S^m_{\varphi} f_* (M)$. It refers to any of $S^m_{\varphi} f_* (M)$ for $m \geq n$. We use this language when we talk about properties that are satisfied for a sheaf $\sF$ on $V^{ne}$ if and only if they are satisfied for $\sF \times_{V^{ne}} V^{me}$.
% \end{remark}

\subsection{Base-change}

We now prove base change for $S^0_{\phi^n} f_*$.

\begin{notation}
\label{notation:base_change}
In the situation of \autoref{notation:S_f_*}, choose
\begin{itemize}
\item $T$ a regular integral excellent scheme with a dualizing complex.
\item $T \to V$  a morphism.
\end{itemize}
We indicate base-change via this morphism by $T$ in subscript. Define the following sheaves on $X_T$.
\begin{enumerate}
\item $B:= \left(f_T \right)^{-1} \sO_T$
\item $Q:= \sO_{X_T}$
\end{enumerate}
Denote by $(p_1)^{\frac{1}{p^i}}$ the natural projection morphism $X^i \times_{V^i} T^i \to X^i$.
\end{notation}

\begin{proposition}
\label{prop:S_f_*_base_change}
In the situation of \autoref{notation:base_change}, if $L \otimes M^{p^e -1}$ is $f$-ample then for $n \gg 0$ the natural base change morphism
\begin{equation*}
f_* \left( A^{\frac{1}{p^{ne}}} \otimes_A M    \right) \times_{V^{ne}} T^{ne}  \to \left(f_T \right)_* \left( B^{\frac{1}{p^{ne}}} \otimes_B M_T    \right)
\end{equation*}
induces a surjective morphism on subsheaves
\begin{equation}
\label{eq:S_f_*_base_change:compatibility}
 S^0_{\phi^n} f_* (M) \times_{V^{ne}} T^{ne} \twoheadrightarrow  S^0_{\phi_T^n} \left(f_T \right)_* (M_T) .
\end{equation}
Furthermore,
\begin{enumerate}
\item \label{itm:S_f_*_base_change:independent} the lower bound on the $n$ for which the above statement holds depends only on $f$ and $M$. In particular, it is independent of  $T$.
\item \label{itm:S_f_*_base_change:high_multiple_of_rel_ample} If $M= Q^l \otimes P$ for some $f$-ample line bundle $Q$, nef line bundle $P$ and integer $l>0$, then there is a uniform lower bound on  $n$ independent of $l$ and $P$.
\item \label{itm:S_f_*_base_change:flat} Even if $L \otimes M^{p^e -1}$ is not assumed to be $f$-ample, \autoref{eq:S_f_*_base_change:compatibility} is an isomorphism if $T \to V$ is flat.
\item \label{itm:S_f_*_base_change:isomorphism_on_an_open} There is a dense open set $U \subseteq V$, such that if the image of $T \to V$ is contained in $U$, then \autoref{eq:S_f_*_base_change:compatibility} is an isomorphism.
\end{enumerate}

\end{proposition}

% \begin{remark}
% Since $S^{ne} \to S$ and $W^{ne} \to T$ are flat, by flat base change the above statement can be also interpreted as a base-change statement for the morphism $W^{ne} \to S^{ne}$.
% \end{remark}

\begin{proof}[Proof of \autoref{prop:S_f_*_base_change}]
First, note that
\begin{equation*}
\left( (p_1)^{\frac{1}{p^{ne}}} \right)^* \left(  \left( L^{\frac{p^{ne} -1}{p^e -1}} \otimes_R M^{p^{ne}}\right)^{\frac{1}{p^{ne}}}  \right) \cong
\left( (p_1)^{\frac{1}{p^{ne}}} \right)^* \left(  \left( L^{\frac{p^{ne} -1}{p^e -1}} \right)^{\frac{1}{p^{ne}}} \otimes_R M \right) \cong
\left( L^{\frac{p^{ne} -1}{p^e -1}}_T \right)^{\frac{1}{p^{ne}}} \otimes_Q M_T.
\end{equation*}
Hence, there is a natural base-change morphism below in \autoref{eq:S_f_*_base_change:base_change_map}. We will show that it is an isomorphism for $n \gg 0$. Furthermore, this $n$ can be chosen independently of $T$.
\begin{equation}
\label{eq:S_f_*_base_change:base_change_map}
\left( f_* \left( \left( L^{\frac{p^{ne} -1}{p^e -1}} \right)^{\frac{1}{p^{ne}}} \otimes_R M \right) \right) \times_{V^{ne}} T^{ne}
\to
\left(f_T \right)_* \left( \left( L^{\frac{p^{ne} -1}{p^e -1}}_T \right)^{\frac{1}{p^{ne}}} \otimes_Q M_T \right)
\end{equation}
Indeed by cohomology and base-change \cite[Lemma 25.20.1]{stacks-project} it is enough to show that for every integer $i>0$,
\begin{equation*}
R^i f_* \left( \left( L^{\frac{p^{ne} -1}{p^e -1}} \right)^{\frac{1}{p^{ne}}} \otimes_R M \right) =0 .
\end{equation*}
Since $V^{ne} \to V$ is affine, this is equivalent to showing that for $i>0$,
\begin{equation*}
0
=
R^i \left( f_{V^{ne}} \right)_* \left( \left( L^{\frac{p^{ne} -1}{p^e -1}} \right)^{\frac{1}{p^{ne}}} \otimes_R M \right)
=
R^i \left( f_{V^{ne}} \right)_* \left( \left( L^{\frac{p^{ne} -1}{p^e -1}} \otimes_R M^{p^{ne}} \right)^{\frac{1}{p^{ne}}}  \right),
\end{equation*}
which is further equivalent to showing that  for $i>0$,
\begin{equation}
\label{eq:S_f_*_base_change:vanishing}
0=R^i f_* \left( L^{\frac{p^{ne} -1}{p^e -1}} \otimes_R M^{p^{ne}}   \right) .
\end{equation}
However, the last vanishing holds for $n \gg 0$ by relative Serre vanishing, independently of $T$. Hence the base change homomorphism of \autoref{eq:S_f_*_base_change:base_change_map} is isomorphism indeed for $n \gg 0$, which can be chosen independently of $T$. In particular,
for $n \gg 0$ there is a commutative base change diagram as follows, which implies the statement of the proposition together with addendum \autoref{itm:S_f_*_base_change:independent}.
\begin{equation}
\label{eq:S_f_*_base_change:commutative_diagram}
\xymatrix{
 f_* \left( \left( L^{\frac{p^{ne} -1}{p^e -1}} \right)^{\frac{1}{p^{ne}}} \otimes_R M \right)  \times_{V^{ne}} T^{ne} \ar[r] \ar@{}[d]^(0.4){\begin{rotate}{-90}$\cong$\end{rotate}}
&
 f_* \left( A^{\frac{1}{p^{ne}}} \otimes_A M    \right) \times_{V^{ne}} T^{ne} \ar[d] \\
%
%%%%%%%%%%%%%%%%%%%%%%%%%%%%%%%%%%%%%%%%%%%%%%%%%%%%%%%%%%%%%%%%%%%%%%%%%%%%%%%%%%%%%%%%%%%%%%%
%
\left(f_T \right)_* \left( \left( L^{\frac{p^{ne} -1}{p^e -1}}_T \right)^{\frac{1}{p^{ne}}} \otimes_Q M_T \right) \ar[r]
&
\left(f_T \right)_* \left( B^{\frac{1}{p^{ne}}} \otimes_B M_T    \right)
}
\end{equation}
For addendum \autoref{itm:S_f_*_base_change:flat} note that if $T \to V$ is flat, then both vertical morphisms in the above diagram are isomorphisms even if $L \otimes M^{p^e -1}$ is not assumed to be $f$-ample. Furthermore, images via flat pullbacks are pullbacks of images, which concludes the proof of \autoref{itm:S_f_*_base_change:flat}.

For addendum \autoref{itm:S_f_*_base_change:high_multiple_of_rel_ample}, note that if $M= Q^l \otimes P$ for an $f$-ample line bundle $Q$ and a nef line bundle $P$, then \autoref{eq:S_f_*_base_change:vanishing} holds for every $l>0$ uniformly, by relative Fujita vanishing \cite{KeelerAmpleFilters}.

For addendum \autoref{itm:S_f_*_base_change:isomorphism_on_an_open} note that by \cite[Theorem 12.8 and Corollary 12.9]{Hartshorne} there is a dense open set $U \subseteq V$ (i.e., the open set where $h^0(X_s,M_s)$ is constant), such that if $T$ maps into $U$ then the right vertical arrow in \autoref{eq:S_f_*_base_change:commutative_diagram} is an isomorphism. In particular then the homomorphism \autoref{eq:S_f_*_base_change:compatibility} is injective, because it is a homomorphism  between subsheaves of the two sheaves involved in the above vertical map. Therefore, by the already proven surjectivity \autoref{eq:S_f_*_base_change:compatibility} is an isomorphism. This concludes addendum \autoref{itm:S_f_*_base_change:isomorphism_on_an_open}.
\end{proof}

\begin{remark}
\label{rem:base_change}
In particular if in \autoref{notation:base_change}, $T=s$ is a perfect point of $V$ (that is, $T=\Spec \left( K \right)$ for some perfect field $K$), then $S^0_{\phi_T^n} \left(f_T \right)_* (M_T)$ can be interpreted as follows. Since, the natural map $K \to \left( K \right)^{1/p^i}$ is an isomorphism, the natural projection $X_{T^i} \to X_T$ is an isomorphism. Furthermore $(f_{T})_*(M_{T^i})$ can be identified with $(f_{T})_*(M_{T}) = H^0(X_T,M_T)$ via $K \to K^{1/p^i}$. Thus $S^0_{\phi_T^n} \left(f_T \right)_* (M_T)$ is identified with
\begin{equation}
\label{eq:S_0_v_m}
S^0_{s,m}:=\im \left( H^0\left(X_{s}, \left( L^{\frac{p^{me}-1}{p^e -1}}_{s} \right)^{\frac{1}{p^{me}}} \otimes M_{s} \right)  \to H^0(X_{s}, M_{s}) \right) ,
\end{equation}
In particular, for $m \gg 0$, \autoref{eq:S_0_v_m} can be identified with $S^0(X_{s}, \sigma(X_{s},\phi_{s}) \otimes M_{s})$.
\end{remark}

\subsection{Uniform stabilization}

\begin{theorem}
\label{prop:S^0_uniform_stabilization}
In the situation of \autoref{notation:S_f_*}, if $L \otimes M^{p^e -1}$ is $f$-ample, then there is an integer $n>0$, such that for all integers $m \geq n$ and perfect points  $s \in V$,
\begin{equation*}
%\label{eq:S^0_uniform_stabilization:statement}
\im \left( S^0_{\phi^m} f_* (M) \otimes_{\sO_{V^{me}}} k(s)^{\frac{1}{p^{me}}} \to H^0(X_{s}, M_{s}) \right)
= S^0_{s,m}
 = S^0(X_{s}, \sigma(X_{s}, \psi_{s}) \otimes M_{s}).
\end{equation*}
(Note $\psi_s$ is defined  in \autoref{sec.BaseChangeToPerfectPoints} and $S^0_{s,m}$ in \autoref{rem:base_change}.)
\end{theorem}

\begin{proof}
 We show the statement  by induction on the dimension of $V$. There is nothing to prove if $\dim V = 0$. Hence, we may proceed to the inductional step and assume that $\dim V > 0$.
By \autoref{prop:S_f_*_base_change}, there is an $n>0$ such that for every perfect point $s \in V$ and every $m \geq n$
\begin{equation*}
\im \left( S^0_{\phi^m} f_* (M) \otimes_{\sO_{V^{me}}} k(s)^{\frac{1}{p^{me}}} \to H^0(X_{s}, M_{s}) \right) = S^0_{s,m}.
\end{equation*}
According to \autoref{lem.BaseChangeForPhin}, by possibly increasing $n$,  we may find a  non-empty, dense open set $U$ such that for all $m \geq n$,
\begin{equation*}
 \left. \sigma_m(X/V, \phi) \right|_{U}  =\left.  \sigma_n(X/V, \phi) \times_{V^{ne}} V^{me} %\otimes_{R_{A^{\frac{1}{p^{ne}}}}} R_{A^{\frac{1}{p^{me}}}}
\right|_{U} .
\end{equation*}
Therefore, applying \autoref{prop:pushforward_stabilizes}, yields (also by possibly increasing $n$) that for all $m \geq n$,
\begin{equation*}
\left. V^{me} \times_{V^{ne}} S^0_{\varphi^n} f_* (M) \right|_{U}  = \left. S^0_{\varphi^m} f_* (M)\right|_{U}.
\end{equation*}
It follows then that for every perfect point $s \in U$,
\begin{equation*}
\im \left( S^0_{\phi^m} f_* (M) \otimes_{\sO_{V^{me}}} k(s)^{\frac{1}{p^{me}}} \to H^0(X_{s}, M) \right)
\end{equation*}
is the same for all $m \geq n$. Therefore,  $S^0_{s,m}$ is stabilized for all perfect $s \in U$ for values of $m \geq n$. Hence it is equal to $S^0(X_{s}, \sigma(X_{s}, \psi_{s}) \otimes M_{s})$. This shows the statement of the proposition for all perfect $s \in U$.  We fix the $n_0 = n$ used above for future use.

On the other hand, consider the reduced scheme $V_1 = V \setminus U$.  Write $V_1 = \coprod V_{1,i}$ as a disjoint union of regular locally closed integral subschemes.  Each $V_{1,i}$ has dimension smaller than $V$.  In particular, there exists an $n_i$ such that the statement holds for $f_i : X_{V_i} \to V_i$.  Letting $n$ be the maximum of $n_0$ and the $n_i$ we obtain our result since every perfect point of $X$ factors through a point of $U$ or of one of the $V_i$.

\end{proof}

\begin{corollary}
\label{cor:S^0_uniform_stabilization}
In the situation of \autoref{notation:S_f_*}, if $L \otimes M^{p^e -1}$ is $f$-ample, then there is a dense open set $U \subseteq V$, such that for every perfect point $s \in U$,
\begin{equation*}
%\label{eq:S^0_uniform_stabilization:statement}
S^0_{\phi^m} f_* (M) \otimes_{\sO_{V^{me}}} k(s)^{\frac{1}{p^{me}}} = S^0(X_{s}, \sigma(X_{s}, \psi_{s}) \otimes M_{s}).
\end{equation*}
In particular, $\dim_{k(s)} S^0(X_{s}, \sigma(X_{s}, \psi_{s}) \otimes M_{s})$ is constant on an open set, and  the rank of $S^0_{\phi^m} f_* (m)$ equals this dimension.
\end{corollary}

\begin{proof}
The main statement follows immediately from \autoref{prop:S^0_uniform_stabilization} and point \autoref{itm:S_f_*_base_change:isomorphism_on_an_open} of \autoref{prop:S_f_*_base_change}. The corollary follows from basic properties of coherent sheaves \cite[Exercise II.5.8]{Hartshorne}
\end{proof}

If there is a point $s_0 \in V$, such that in $X_{s_0}$ we have $H^0= S^0$, then we can say more than the above corollary: it turns out that it can be assumed that $s_0 \in U$, and further that $U$ has other useful properties. This is proved in the following theorem.

\begin{theorem}
\label{prop:S_0_maximal_is_open}
In the situation of \autoref{notation:S_f_*}, if $L \otimes M^{p^e -1}$ is $f$-ample and there is a perfect point $s_0 \in V$ such that $ H^0(X_{s_0},M_{s_0})=S^0(X_{s_0}, \sigma(X_{s_0}, \psi_{s_0}) \otimes M_{s_0})$, then there is an open neighborhood $U$ of $s_0$, such that
\begin{enumerate}
 \item \label{itm:S_0_maximal_is_open:locally_free} $f_* M|_U$ is locally free and compatible with base change and
\item \label{itm:S_0_maximal_is_open:maximal} $ H^0(X_{v},M_{v})=S^0(X_{v}, \sigma(X_{v}, \psi_{v}) \otimes M_{v})$ for every perfect point $v \in U$.
\end{enumerate}
In particular, $\dim S^0(X_{v}, \sigma(X_{v}, \psi_{v}) \otimes M_{v})$ is constant for $v \in U$.
\end{theorem}

\begin{proof}
By \autoref{prop:S^0_uniform_stabilization}, there is an $n$, such that  for every $v \in V$ and integer $m \geq n$, the natural base-change homomorphism induces a surjection as follows
\begin{equation}
\label{eq:S_0_maximal_is_open:diagram}
\xymatrix@R=19pt{
S^0_{\phi^m} f_* (M) \otimes_{\sO_{V^{me}}} k(v)^\frac{1}{p^{me}} \ar@{->>}[r] \ar[d] & \ar@{^(->}[d] S^0(X_{v},\sigma (X_{v},\psi_{v}) \otimes M_{v}) \\
 f_* \left(M \otimes_A A^{\frac{1}{p^{me}}} \right) \otimes_{\sO_{V^{me}}} k(v)^\frac{1}{p^{me}}
%\cong f_* (M) \otimes_{\sO_{S}} k(s)
\ar[r] &  H^0(X_{v}, M_{v}). \\
}
\end{equation}
Furthermore, the right vertical arrow is surjective for $v = s_0$. It follows that so is the bottom horizontal arrow then. Using \cite[Corollary III.12.11]{Hartshorne} concludes \autoref{itm:S_0_maximal_is_open:locally_free}.

To prove point \autoref{itm:S_0_maximal_is_open:maximal}, let us consider \autoref{eq:S_0_maximal_is_open:diagram} for $v =s_0$ again. Now we know that the bottom horizontal arrow is isomorphism. However, then the left vertical arrow is also surjective. By Nakayama's-lemma, there is an equality $(S^0_{\phi^n} f_* (M))_{v} = ( f_* (M))_{v}$ of stalks, which means that by possibly restricting $U$, for every $v \in U$, the left vertical arrow of \autoref{eq:S_0_maximal_is_open:diagram} is surjective. However then, since the bottom horizontal arrow is isomorphism for every $v \in U$, point \autoref{itm:S_0_maximal_is_open:maximal} follows.
\end{proof}

%\todo{{\bf Zsolt:} is $\dim_k(S^0(X_v,\sigma(X_v,\phi_v) \otimes M_v)$ lower semicontinuous (it is not upper semicontinuous be the example below)? I would be very surprised. So, I need an example where $S^0$ jumps at the closed point.}

The surjection in  \autoref{prop:S_f_*_base_change} is not an isomorphism in general by the following two examples.
In the first example the global geometry of the smooth fibers causes the anomaly, while in the second, the degeneration of $F$-pure  singularities to  non $F$-pure singularities  is the main culprit.  However, we first show a lemma.

\begin{lemma}
\label{lem:tango_technique}
  Let $Y$ be a smooth curve over $k$ and $G$ an effective divisor on $Y$. Define $\sB^1_Y$ as the cokernel of $\sO_X \to F_* \sO_X$. Then
\begin{equation*}
S^0(Y, \sigma(X,0) \otimes \omega_Y(G)) = H^0(Y, \omega_Y(G)) \Leftrightarrow H^0(Y, \sB^1_Y(-p^nG))=0 \quad (\forall n >0).
\end{equation*}

\end{lemma}

\begin{proof}
 Consider the following exact sequence.
\begin{equation*}
\xymatrix{
0 \ar[r] & \sO_Y \ar[r] & F_* \sO_Y \ar[r] & \sB^1_Y \ar[r] & 0
}
\end{equation*}
Twist this  sequence by $-G$, apply cohomology and use that since $G$ is effective, either
\begin{enumerate}
\item $G>0$ and hence $H^0(Y, (F_* \sO_Y)(-G)) \cong H^0(\sO_Y(-F^* G))=0$, or
\item  $G=0$ and then $H^0(Y, \sO_Y) \to H^0(Y, F_* \sO_Y)$ is an isomorphism.
\end{enumerate}
 In either case, we have another exact sequence:
\begin{equation*}
\xymatrix{
0 \ar[r] & H^0(Y, \sB^1_Y(-G)) \ar[r] & H^1(Y,\sO_Y(-G)) \ar[r] & H^1(Y, F_* \sO_Y(- F^* G))},
\end{equation*}
where the last map is the Serre dual to $H^0(Y, (F_*\omega_Y)(G)) \to H^0(Y, \omega_Y(G))$. Furthermore, observe that the map $H^0(Y, (F^{n+1}_*\omega_Y)(G)) \to H^0(Y, (F^{n}_* \omega_Y)(G))$ is harmlessly identified with the map \mbox{$H^0(Y, (F_*\omega_Y)(p^nG)) \to H^0(Y, (\omega_Y(p^nG)))$}. Therefore, combining the previous statements, we see that $H^0(Y, (F^{n+1}_*\omega_Y)(G)) \to H^0(Y, (F^{n}_* \omega_Y)(G))$ surjects if and only if $H^0(Y, \sB^1_Y(-p^nG))=0$.
\end{proof}

\begin{example}
\label{ex:S_f_*_base_change_not_isomorphism}
We choose two smooth projective curves $C$ and $D$ of genus $3$ over an algebraically closed field  $k$ of prime characteristic with two ample line bundles $N_C$ and $N_D$(of degree 1), such that $S^0(C,\omega_C \otimes N_C) \neq H^0(C,\omega_C \otimes N_C)$, but $S^0(D,\omega_D \otimes N_D) = H^0(D,\omega_D \otimes N_D)$. Since the relative Picard scheme over the moduli stack of smooth curves of genus $g$ is smooth and irreducible, there is a (possibly reducible) curve connecting $(C, N_C)$ and $(D,N_D)$ in the above moduli space. However, then by possibly replacing $(C,N_C)$ and $(D,N_D)$ we may find also an irreducible curve connecting them. Therefore, by passing to the normalization of this curve, we may assume that there is a family $f : X \to V$, a line bundle $N$ on $X$ and two points $c,d \in C$, such that if we set $(C,N_C):=\left( X_c,N|_{X_c} \right)$ and $(C,N_C):=\left( X_c,N|_{X_c} \right)$, then
\begin{enumerate}
\item $V$ is a smooth curve,
\item $X$ is a family of smooth curves of genus $3$,
\item \label{itm:S_f_*_base_change_not_isomorphism:deg_one} $\deg_{X/V} N = 1$,
\item \label{itm:S_f_*_base_change_not_isomorphism:S_0_C} $S^0(C,\omega_C \otimes N_C) \neq H^0(C,\omega_C \otimes N_C)$ and
\item \label{itm:S_f_*_base_change_not_isomorphism:S_0_D} $S^0(D,\omega_D \otimes N_D) = H^0(D,\omega_D \otimes N_D)$.
\end{enumerate}
Furthermore, fix $e=1$, $L:=\omega_{X/V}^{1-p}$ and $\phi : \omega_{X/V}^{\frac{1}{p}} \to \sO_{X_{V^1}}$ the map with $D_{\phi}=0$ (see  \autoref{def:D_phi_Delta_phi} for the definition of $D_{\phi}$). Note that by \autoref{lem.DeltaPhiRestrictedEqualsDeltaPsi}, $D_{\phi_s} = 0$ as well. Set $M:=\omega_{X/V} \otimes N$.

By assumption \autoref{itm:S_f_*_base_change_not_isomorphism:deg_one}, for every $v \in V$ and $i>0$, $H^i(X_v, N_v) =0$. Therefore, $f_* N$ is a vector bundle of rank 3, and its formation is compatible with arbitrary base-change. In particular, by \autoref{prop:S_0_maximal_is_open}
and assumption \autoref{itm:S_f_*_base_change_not_isomorphism:S_0_D}, $S^n_{\phi} f_* (\omega_{X/V} \otimes M) \subseteq f_* (\omega_{X/V} \otimes M)$ is a subsheaf that is isomorphic generically to $f_*(\omega_{X/V} \otimes M)$. Furthermore, by being a subsheaf of $f_* ( \omega_{X/V} \otimes M)$ it is torsion free, and by $V$ being a curve, it is locally free of rank 3. Therefore,
\begin{equation*}
\dim_{k} S^n_{\phi} f_* (\omega_{X/V} \otimes M) \otimes k(c)   = 3
\end{equation*}
However,
\begin{equation*}
\dim_k S^0(C,\omega_C \otimes M_C) < \dim_k H^0(C,\omega_C \otimes M_C)=3
\end{equation*}
This shows that the base change morphism
\begin{equation*}
 S^0_{\phi^n} f_* (\omega_{X/V} \otimes M) \otimes k(c) \to  S^0(C,\omega_C \otimes M_C)
\end{equation*}
cannot be isomorphism.

We are left to give the polarized curves $(C,M_C)$ and $(D,M_D)$. For that consider the situation of \autoref{lem:tango_technique}.  By \cite[Lemma 12]{TangoOnTheBehaviorOfVBAndFrob},  $H^0(Y, \sB^1(-p^nG))=0$ if $p^n G \geq n(Y)$ where $n(Y)$ is a numerical invariant of the curve, which is at most $1$ for genus $3$ curves \cite[Lemma 10]{TangoOnTheBehaviorOfVBAndFrob}. Hence $H^0(Y, \sB^1(-p^nG))=0$ if $\deg G \geq 1$ and $n >1$. In particular, in our special case (i.e., if $Y$ is of genus $3$ and $\deg G \geq 1$) \autoref{lem:tango_technique} states that
\begin{equation*}
S^0(Y, \sigma(Y,0) \otimes \omega_Y(G)) = H^0(Y, \omega_Y(G)) \Leftrightarrow H^0(Y, \sB^1_Y(-G))=0 .
\end{equation*}
Finding a curve where $H^0(Y, \sB^1_Y(-G))=0$ is quite easy, because again using \cite[Lemma 12]{TangoOnTheBehaviorOfVBAndFrob} yields that if $n(Y)=1$, then there is a degree one divisor $G$ for which  $H^0(Y, \sB^1_Y(-G)) \neq 0$, and then by passing to a linearly equivalent divisor we may also assume that $G$ is effective. Further, \cite[Example 1]{TangoOnTheBehaviorOfVBAndFrob} gives a curve ( $x^3 y + y^3 z + z^3 x=0$) for which $n(Y)=1$. Therefore we have found $C$ and $N_C$.

To find $D$ and $N_D$ assume further that $k= \overline{\bF}_3$. Then by \cite{MillerCurvesWithInvertibleHasseWitt,KoblitzPAdicVariationOfTheZetaFunction} a general genus three curve curve is ordinary, or equivalently $S^0(Y, \sigma(Y,0) \otimes \omega_Y) = H^0(Y, \omega_Y)$. Therefore then $H^0(Y, \sB^1_Y)=0$ for such a $Y$. Take now an arbitrary effective degree $1$ divisor $G$. Then $\sB^1_Y (-G)$ embeds into $\sB^1_Y$ and hence $H^0(Y, \sB^1_Y(-G))=0$ as well. In particular, we can choose $D$ to be a generic curve and $N_D$ to be an arbitrary degree one line bundle on $D$.

% \begin{equation*}
% x^4 + y^2 z^2 = 0
% \end{equation*}
% tangent line at $(\theta,1,1)$.

 %Tango defines an invariant called $n(X)$ \cite{TangoOnTheBehaviorOfVBAndFrob}. The two important facts are the following
\end{example}

\begin{example}
\label{ex:S_f_*_base_change_not_isomorphism_arbitrary_power}
In the following example for an $f$-ample line bundle $Q$  the surjection
\begin{equation}
\label{eq:S_f_*_base_change_not_isomorphism_arbitrary_power:surjection}
S^0_{\varphi^n} f_* (Q^l) \otimes_{\sO_{V^{ne}}} \left( k(s)^{\frac{1}{p^{ne}}} \right)   \twoheadrightarrow S^0(X_s, \sigma(X_s, \phi_s) \otimes Q_s^l).
\end{equation}
is not an isomorphism for any integer $n \geq n_0$ and $l>0$. Therefore, the isomorphism in \autoref{prop:S_f_*_base_change} cannot be obtained by stronger positivity assumptions.

Let $C$ be the projective cone over a supersingular elliptic curve, and let $D$ be a non-singular cubic surface. Then these can be put into a family $f : X \to V$ as above. More precisely, we may find a
family $f : X \to V$, and a point $c \in V$, such that  $C:=X_c$, and hence the following hold
\begin{enumerate}
\item $V$ is a smooth curve,
\item $X$ is a a flat family of normal surfaces,
\item \label{itm:S_f_*_base_change_not_isomorphism_arbitrary_power:C} $X_c$ is not sharply $F$-pure at on point $P \in X_c$ and
\item \label{itm:S_f_*_base_change_not_isomorphism_arbitrary_power:D} $X_s$ is sharply $F$-pure for every $s \in V \setminus \{c \}$.
\end{enumerate}
Let $L:= \omega_{X/V}^{1-p}$ and $\phi:=\phi_0$ as in the previous example. Then by \autoref{thm.UniversalNForAllClosedFibers} and Nakayama lemma, for every $n \gg 0$,  $\sigma_n(X/V, \phi)|_{V \setminus \{c\}} \cong \sO_{f^{-1}(V \setminus \{c\})}$.  Choose now an arbitrary sufficiently $f$-ample line bundle $M$. By \autoref{cor:S_f_if_big_power_of_rel_ample}, for every $n \gg 0$, $S^0_{\phi^n} f_* (M^l)|_{V \setminus \{c\}} \cong (f_*(M))_{V^{ne} \setminus \{c \}}$ for every integer $l>0$.  In particular the $\rk S^0_{\phi^n} f_* (M^l) = \rk f_* (M^l)$ for every $n \gg 0$ and $l>0$ (where $n$ does not depend on $l$). Consequently,
\begin{equation}
\label{eq:S_f_*_base_change_not_isomorphism_arbitrary_power:dimension}
\dim S^0_{\phi^n} f_* (M^l) \otimes k(c) \geq \rk f_* (M^l) =  \underbrace{ \dim_k H^0(X_c, M^l_c )}_{\textrm{$M$ is relatively ample enough}}
\end{equation}
On the other hand for all $n \gg 0$ (independent of $l$)
\begin{equation}
\label{eq:S_f_*_base_change_not_isomorphism_arbitrary_power:surjection_again}
 S^0_{\phi^n} f_* (M^l) \otimes k(c) \twoheadrightarrow
\underbrace{ S^0_{\phi_c^n} (f_c)_* (M_c^l) }_{\textrm{by point \autoref{itm:S_f_*_base_change:high_multiple_of_rel_ample} of \autoref{prop:S_f_*_base_change} } }
= \underbrace{H^0(X_c, \sigma(X_c,0) \otimes M^l_c)}_{\textrm{sine $M_c$ is ample enough (\cite{PatakfalviSemipositivity}) and \autoref{lem.DeltaPhiRestrictedEqualsDeltaPsi}}}
\end{equation}
However, by the relative ampleness assumption,  $M_c$ is globally generated, hence  $\dim_k H^0(X_c, \sigma(X_c,0) \otimes M^l_c)< \dim_k H^0(X_c, M^l_c)$ (because $\sigma(X_c,0) \subsetneq \sO_{X_c}$). Therefore, \autoref{eq:S_f_*_base_change_not_isomorphism_arbitrary_power:dimension} and \autoref{eq:S_f_*_base_change_not_isomorphism_arbitrary_power:surjection_again} implies that indeed the surjection \autoref{eq:S_f_*_base_change_not_isomorphism_arbitrary_power:surjection} is not an isomorphism for any $l>0$ and any $n \gg 0$ (independently bounded of $l$).

\end{example}

\begin{remark}
The fundamental reason for the above example is that $\sigma_n(X/V, \phi)$ is NOT flat in general.
\end{remark}

The above two  examples show that $\dim_k(S^0(X_v,\sigma(X_v,\phi_v) \otimes M_v)$ is not upper semicontinuous. One might guess then it is lower   semicontinuous. The next example shows that that is also not the case (this can also be deduced from \cite[Example 5.5]{HaraRatImpliesFRat} and \cite[Theorem 8.3]{TanakaTraceOfFrobenius} as pointed out to us by Tanaka). So, $\dim_k S^0(X_v,\sigma(X_v,\phi_v) \otimes M_v)$ is not semicontinuous in either direction.

\begin{example}
\label{ex.NotLowerSemiContEither}
Take a flat family $f : X \to V$ of ordinary elliptic curves, and a line bundle $M$ on $X$ of relative degree $0$, such that $M_{X_{v_0}} \cong \sO_{X_{v_0}}$ for a special $v_0 \in V$, and $M_{X_{v}} \not\cong \sO_{X_{v}}$ for generic $v \in V$. Then, by the ordinarity of the fibers, for all $v \in V$,
\begin{equation*}
H^0(X_v, M_v) = S^0(X_v, \sigma(X_v, 0) \otimes M_v).
\end{equation*}
However, $\dim H^0(X_v, M_v)=1$ for the special fiber and $0$ for the generic one. So, $\dim S^0(X_v, \sigma(X_v, 0) \otimes M_v)$ is not lower semicontinuous in this example. One can also easily modify this example by taking an $M$ with higher relative degree to obtain  higher values of dimension for the generic fiber.
\end{example}

% \todo{{\bf Zsolt:} if $S^0(X_v,M_v)=H^0(X_v,M_v)$, then the same holds on an open neighborhood. Maybe use this that for a generic curve $S^0(X_v,M_v)=H^0(X_v,M_v)$, by exhibiting one curve for each genus  such that $S^0(X_v,M_v)=H^0(X_v,M_v)$.}

% \todo{{\bf Zsolt:} interesting question, not sure though that it is in the scope of this article:
% What is the structure of the sets
% \begin{equation*}
% \{ (X,\Delta,L) | \dim_k H^0(X,L) - \dim_k S^0(X, \sigma(X, \Delta) \otimes L) = c \}
% \end{equation*}
% in the adequate moduli space, for some integer $c \geq 0$ fixed?
% }

% \todo{{\bf Zsolt:} corollary  : uniform bound for the stabilization of $S^0$ on bounded moduli spaces (i.e. curves with fixed genus), maybe show by example that there is no uniform bound if the genus is not fixed, the problem with this that it is obvious for curves }

% \todo{{\bf Zsolt:} What is the actual structure of the sets $\{ (X,\Delta,L) | \dim_k H^0(X,L) - \dim_k S^0(X, \sigma(X, \Delta) \otimes L) = c \}$ I think it follows that it is a union of locally closed sets. However, it is not clear if $\{x \neq 0\} \cup {x=y=0}$ can happen or not. }

% \todo{{bf Zsolt:} corollary: the slight generalization of openness of ordinarity for curves }

\subsection{Global generation and semi-positivity}

Suppose now $V$ is a projective variety over a prefect field $k$.  In this section, we explore global generation results if $L \otimes M^{p^e - 1}$ is ample (instead of just relatively ample).  In particular,  $S^0_{\phi^n} f_*(M)$ is globally generated for all large $n$.  This should not be surprising since $S^0_{\phi^n} f_*(M)$ lives on $V^{ne}$ where ampleness is amplified.

\begin{proposition}
\label{prop:global_generation_of_S_f_*}
In the situation of \autoref{notation:S_f_*}, if $V$ is projective over a perfect field $k$ and
$L \otimes M^{p^e - 1}$ is ample, then $S^0_{\phi^n} f_*(M)$ is globally generated for every $n \gg 0$.
\end{proposition}

\begin{proof}
Choose a  globally generated ample divisor $H$ on $V$ and let $d:=\dim V$.
Since  $S^0_{\phi^n} f_*(M)$ is defined as the image of
\begin{equation*}
f_* \left( \left( L^{\frac{p^{ne} -1}{p^e -1}} \right)^{\frac{1}{p^{ne}}} \otimes_R M \right) ,
\end{equation*}
it is enough to show that this sheaf is globally generated as an $\O_{V^{ne}}$-module. By Mumford's criterion \cite[Theorem 1.8.5]{LazarsfeldPositivity1}, it is enough to show that for every $i >0$ and for the divisor $H_{ne}$ on $V^{ne}$ identified with $H$ via the isomorphism $V^{ne} \cong V$,
\begin{equation}
\label{eq:global_generation_of_S_f_*}
H^i \left(V^{ne}, \sO_{V^{ne}} \left(-iH_{ne} \right) \otimes_{\sO_{V^{ne}}} f_* \left( \left( L^{\frac{p^{ne} -1}{p^e -1}} \right)^{\frac{1}{p^{ne}}} \otimes_R M \right) \right)=0.
\end{equation}
By relative Serre vanishing for $i>0$,
\begin{equation*}
 R^i f_*  \left( \left( L^{\frac{p^{ne} -1}{p^e -1}} \right)^{\frac{1}{p^{ne}}} \otimes_R M \right) =0 .
\end{equation*}
In particular, then to prove \autoref{eq:global_generation_of_S_f_*} it is enough to show
\begin{equation*}
H^i \left( X,    L^{\frac{p^{ne} -1}{p^e -1}}  \otimes M^{p^{ne}} ( - i f^* H)  \right)=0.
\end{equation*}
However this is equivalent to showing that,
\begin{equation*}
H^i \left( X,    \left(L  \otimes M^{(p^e -1)}  \right)^\frac{p^{ne} -1}{p^e -1} \otimes M(-i f^*H)   \right)=0,
\end{equation*}
which holds by Serre-vanishing for $n \gg 0$ and $i>0$.
\end{proof}

We recall the following definition.

\begin{definition}[Definition 2.11 of \cite{ViehwegQuasiProjectiveModuli}]
\label{defn:weakly_positive}
Let $\sF$ be a  sheaf on a normal, quasi-projective (over a perfect field $k$) variety $V$ and $U \subseteq V $ a dense open set. Let $U_\mathrm{lf}$ be the open locus of $V$ where $\sF' := \sF/\text{(torsion)}$
is locally free. Then $\sF$ is \emph{weakly positive over $U$}, if for a fixed (or equivalently every: \cite[Lemma 2.14.a]{ViehwegQuasiProjectiveModuli}) ample line bundle $\sH$ for every $a>0$ there is a $b>0$ such that $S^{\langle ab\rangle} (\sF) \otimes  \sH^b$ is globally generated over $U \cap U_\mathrm{lf}$ (here $S^{\langle ab\rangle} (\sF) $ denotes the $ab$th symmetric reflexive power of $\sF$).  We say that $\sF$ is \emph{weakly-positive} if it is weakly-positive over some dense open set.
\end{definition}

\begin{lemma}
\label{lem:weakly_positive_finite}
%\begin{enumerate}
If $g : Y \to Z$  is a finite morphism of normal varieties, quasi-projective over $k$, $U \subseteq Z$  a dense open set and $\sF$ a sheaf on $Z$, then $\sF$ is weakly positive over $U$ if and only if $g^* \sF$ is weakly positive over $g^{-1}(U)$.
%\item If $Y$ is a normal scheme, quasi-projective over $k$, $U \subseteq Z$  a dense open set and $\sF$ a sheaf on $Z$, then $\sF$ is weakly positive over $U$ if and only if for a fixed ample enough (in a way specified in the proof) divisor $H$ on $Y$ and every Frobenius pullback $\tau_e : Y^e \to Y$, $\tau_e^* \sF \otimes \sO_{V^e} \left( \frac{\tau^* H}{p^e} \right)$ is weakly-positive.
%\end{enumerate}

\end{lemma}

%\todo{ {\bf Zsolt} : references to the assumptions in the following proof.}

\begin{proof}
%\begin{enumerate}
% \item
The only if direction is shown in \cite[Lemma 2.15.1]{ViehwegQuasiProjectiveModuli}. For the other direction, according to \autoref{defn:weakly_positive} by throwing out codimension two subset we may assume that $\sF$ is locally free and $Y$ is flat over $Z$. In particular then $g^* \sF$ is also locally free. Choose a very ample divisor $\sH$ on $Z$, such that
\begin{enumerate}
\item \label{itm:weakly_positive_finite:map} $\sHom(g_* \sO_Y,\sH)$ is globally generated and
\item \label{itm:weakly_positive_finite:pushforward} $g_* \sO_Y \otimes \sH$ is globally generated.
\end{enumerate}
Fix then an $a>0$. By the weak positivity of $g^* \sF$, there is a $b>0$, such that there is a homomorphism $\alpha: \sO_Y^{\oplus N} \to S^{\langle ab \rangle} (g^*\sF) \otimes g^*\sH^b$ surjective over $g^* U$.  Consider then for  every choice of $s \in \Hom(g_* \sO_Y,\sH)$ and $s' \in H^0(\sH^{b-1})$  the following composition.
\begin{equation*}
\xymatrix@C=35pt{
(g_* \sO_Y )^{\oplus N} \ar[r]^-{\gamma:=g_* \alpha} \ar[dr]_{\beta_s} & g_* (S^{\langle ab \rangle} (g^* \sF) \otimes g^* \sH^b )  \cong S^{\langle ab\rangle} (\sF) \otimes \sH^b \otimes g_* \sO_Y   \ar[d]^{\delta_{s,s'}:=\id_{S^{\langle ab \rangle} (\sF) \otimes \sH^{b}} \otimes s \otimes s'} \\
&  S^{\langle ab \rangle} (\sF) \otimes \sH^{2b}
}
\end{equation*}
Choose a point $P \in U$ and an element $f$ of the fiber $S^{\langle ab \rangle} (\sF) \otimes k(P)$ over $P$. Then by assumption \autoref{itm:weakly_positive_finite:pushforward} for any preimage $Q$ of $P$ there is a section $t' \in \sO_Y^{\oplus N}$, such that  $\alpha(t')_Q=g^* f \times \textrm{''generator``}$. The section $t'$ descends to a section $t:=g_* (t') \in (g_* \sO_Y )^{\oplus N}$, such that $\gamma(t)_P = f  \times \textrm{''generator``}\times h$ for some $ h \in (g_* \sO_Y)_P$.     However then  by \autoref{itm:weakly_positive_finite:map} and the very ampleness of $\sH$ for a suitable choice of $s$ and $s'$, $\delta_{s,s'} (\gamma(t)))_P= f \times \textrm{``generator''}$ is not zero at $P$. Therefore, for every point $P$ in $U$ and every element $f$ in the fiber of $S^{\langle ab \rangle} (\sF) \otimes \sH^{2b}$ at that point  we find a section of $S^{\langle ab \rangle} (\sF) \otimes \sH^{2b}$ whose image in the fiber is $f$ up to multiplication by unit. This finishes our proof.
% \item Here we may also assume that $\sE$ is locally free. We modify the proof of the previous point. Let $a >0$ be an integer, and choose $e>0$ such that $1 + a \leq p^e$.
% Define $\sH_e:=\sO_{V^e} \left( \frac{\tau^* H}{p^e} \right)$, where $\tau$ is defined as in the statement of the proposition. Then by weak-positivity of $\tau^* \sE \otimes \sH_e$, there is a $b>0$, such that $S^{ab}(\tau^* \sE \otimes \sH_e) \otimes \sH_e^b \cong S^{ab}( \tau^* \sE ) \otimes \sH_e^{b(1+a)}$ is globally generated. By possibly increasing $b$ we may also assume that $\sH_e^{p^e - (1+a)}$ is globally generated, and then so is $S^{ab}(\tau^* \sE) \otimes \sH_e^{b p^e} \cong S^{ab}(\tau^* \sE) \otimes (\tau^* \sH)^b$. From, here we can simply apply the proof of the previous point
% \end{enumerate}
\end{proof}

\begin{corollary}
\label{cor:weak_positivity_of_S_f_*}
In the situation of \autoref{notation:S_f_*},  let $V$ be projective over a prefect field $k$,
$L \otimes M^{p^e -1}$ ample and $n>0$ be an integer such that there is an open set $U \subseteq V$  for which
\begin{equation*}
S^0_{\phi^n} f_* (M) \times_{U^{ne}} U^{me} = S^0_{\phi^m} f_* (M)|_{U^{me}},
\end{equation*}
for every integer $m \geq n$. Then  $S^0_{\phi^n} f_* (M)$ is weakly-positive. Note that an integer as above always exists by \autoref{prop:pushforward_stabilizes} and \autoref{prop.StabilizingSigmaCanonical}.
\end{corollary}

\begin{proof}
By the assumption, there is an embedding
\begin{equation*}
S^0_{\phi^m} f_* (M) \hookrightarrow S^0_{\phi^n} f_* (M) \times_{V^{ne}} V^{me},
\end{equation*}
which is generically isomorphism. Since the left sheaf is globally generated and hence weakly-positive, so is the right one. However then by \autoref{lem:weakly_positive_finite}, so is $S^0_{\phi^n} f_* (M)$.
\end{proof}

\begin{lemma}
\label{lem:weakly_positive_criterion_finite_maps}
If $\sF$ is a coherent sheaf on a normal variety $V$ over a perfect field $k$, then $\sF$ is weakly-positive if and only if for every ample line bundle $\sH$ and every integer $p \nmid r>0$ there is a  finite morphism $\tau: T \to V$, such that  $\tau^* \sH \cong (\sH')^r$ for some line bundle $\sH'$ and $\tau^* \sF \otimes \sH'$ is weakly-positive.
\end{lemma}

\begin{proof}
The proof is identical to the $(b) \Rightarrow (a)$ part of \cite[Lemma 2.15.1]{ViehwegQuasiProjectiveModuli}.
\end{proof}

\begin{theorem}
\label{thm:global_generation_of_S_f_*}
In the situation of \autoref{notation:S_f_*}, if $V$ is projective and
$L \otimes M^{p^e -1}$ is a nef and $f$-ample, then $S^0_{\phi^n} f_*(M)$ is weakly positive for $n \gg 0$.
\end{theorem}

\begin{proof}
Choose an integer $n>0$, as in \autoref{cor:weak_positivity_of_S_f_*}. Fix an ample line bundle $\sH$ on $V$, an integer $p \nmid r>0$ and a  finite morphism $\tau: T \to V$, such that  $\tau^* \sH \cong (\sH')^r$ for some line bundle $\sH'$. Such a morphism exits by \cite[Lemma 2.1]{ViehwegQuasiProjectiveModuli}. By \autoref{lem:weakly_positive_criterion_finite_maps}, we are supposed to prove that $\left(\tau^\frac{1}{p^{ne}} \right)^* S^n_\phi f_* (M) \otimes_{\sO_{V^{ne}}} (\sH')^\frac{1}{p^{ne}}$ is weakly positive.  By disregarding codimension two closed sets, we may assume that $T$ is regular as well. Then by point \autoref{itm:S_f_*_base_change:flat} of \autoref{prop:S_f_*_base_change}, we see that $n$ satisfies also the assumptions of  \autoref{cor:weak_positivity_of_S_f_*} but for $f$ and $M$ replaced by $f_T$ and $M_T \otimes f^*_T \sH_T'$, respectively.  In particular, since $M_T \otimes f_T^* \sH_T'$ is ample, \begin{equation*}
S^0_{\phi_T^n} (f_T)_* (M_T \otimes f_T^* \sH_T') \cong \sH_T' \otimes S^0_{\phi_T^n} (f_T)_* (M_T)
\end{equation*}
 is weakly-positive over $T^{ne}$.
% Choose any ample line bundle $H$ on $V$ and an integer $n>0$, such that $S^m_\phi f_*(M)$ is stabilized for $m \geq n$. Let $\tau: V':=V^{e'} \to V$ be the $e'$-the iterated Frobenius and denote $f':=f_{V^{e'}}$ and $M':=M_{V^{e'}}$.   Then there is an ample line bundle $H'$ on $V'$ such that $(H')^{p^{e'}} \cong H$ and hence $L \otimes ((f')^* H' \otimes M')^{p^e -1}$ is an ample line bundle. In particular $S^-_{\phi'} f'_* ((f')^* H' \otimes M') \cong \sO_V(H') \otimes S^-_{\phi'} f'_* ( M')$ is weakly positive by \autoref{cor:semi_positivity_S_f_*}. In particular then so is  Notice now that
% by addendum \autoref{itm:S_f_*_base_change:flat} of \autoref{prop:S_f_*_base_change}, $S^n_{\phi'} f'_* ( M') \cong \left( S^n_\phi f_* ( M) \right)_{V'}$.
%
% and \autoref{prop:weakly_positive_finite} we may replace $f$ by any of its pullback $f_{V'}$ via an iterated Frobenius $\tau : V^{e'} \to V$.
\end{proof}

\subsection{Relation to global canonical systems}
\label{sec.RelationBetweenRelativeAbsoluteS0F*}There has been another subsheaf $S^0 f_*( \sigma(X, \Delta) \otimes M)$ of $f_* (M)$ introduced in \cite[Definition 2.14]{HaconXuThreeDimensionalMinimalModel} with a definition similar to that of $S^0_{\phi^n} f_*(M)$. In this section we show some of the similarities and differences between the two sheaves. The advantage of $S^0 f_*( \sigma(X, \Delta) \otimes M)$ over $S^0_{\phi^n} f_*(M)$ is that it lives on one $V$, there is no involvement of $V^{ne}$ at all. On the other hand we show that contrary to $S^0_{\phi^n} f_*(M)$ it does not restrict even generically to $S^0(X_s, \sigma (X_s, \Delta_s) \otimes M_s)$.
%Further we give some evidence that $S^0 f_{V^{ne},*}( \sigma(X_{V^{ne}}, \Delta_{V^{ne}}) \otimes M)$ does restrict to the right groups for $n \gg 0$, at least for suitable choices of $M$.

Throughout the section we use the divisorial language since the presentation seems to be more straightforward this way, hence $\Delta$ satisfies the conditions of \autoref{rem:relative_canonical}. Recall that we defined the notion of pair in our setting in \autoref{def:pair}. Whenever we say pair, we mean everything assumed there. In particular, all the assumptions on $f : X \to V$ from \autoref{notation:basic}: flatness, equidimensionality, etc.

First,  we recall the definition of $S^0 f_* (M) $.
\begin{definition} \cite[Definition 2.14]{HaconXuThreeDimensionalMinimalModel}
Let $\Delta$ be an effective $\bQ$-divisor, such that $(X,\Delta)$ is a pair and assume that $f : X \to V$ is projective. Let $e$ be the smallest\footnote{The intersection below is easily seen to be descending, hence if one chooses another $e$, the same object results.} positive integer  such that $(p^e -1)(K_X + \Delta)$ is Cartier. Then given a line bundle $M$ on $X$, we define
\begin{equation*}
S^0 f_*( \sigma(X, \Delta) \otimes M) := \bigcap_{n >0} \im \left( f_* ( F^{ne}_* \sO_X((1-p^{ne})(K_X + \Delta)) \otimes M ) \to f_* (M) \right).
\end{equation*}
We say that $S^0 f_*( \sigma(X, \Delta) \otimes M)$ \emph{stabilizes} if the above intersection stabilizes. Note that in that case $S^0 f_*( \sigma(X, \Delta) \otimes M)$  is a coherent sheaf.
\end{definition}

Recall that $S^0 f_*( \sigma(X, \Delta) \otimes M)$ stabilizes if $M - K_X - \Delta$ is relatively ample, see \cite[Proposition 2.15]{HaconXuThreeDimensionalMinimalModel} for a proof.

\begin{proposition}
\label{prop:global_relative_relation}
Let $\Delta$ be an effective $\bQ$-divisor, such that $(X,\Delta)$ is a pair and assume that $f : X \to V$ is projective. Let $M$ be a line bundle on $X$ such that $S^0 f_*( \sigma(X, \Delta) \otimes M)$ stabilizes. Further assume that $V$ is regular. Then for every $n \gg 0$
\begin{equation*}
S^0_{\Delta,ne} f_* (M) \subseteq (S^0 f_* ( \sigma(X,\Delta) \otimes M))_{V^{ne}}
\end{equation*}
as subsheaves of $(f_*(M))_{V^{ne}}$, where $S^0_{\Delta,ne} f_* (M) $ is defined in \autoref{defn:S_f_*}.
\end{proposition}

\begin{proof}
The statement is local over the base hence we can assume that $V$ is affine, $K_V \sim 0$, and that $\sHom_{\O_V}(F^i_* \O_V, \O_V) \cong F^i_* \O_V$ is a free $\O_V$-module for all $i \geq 0$.  Note then also we can identify $K_{X/V}$ with $K_X$.
We introduce the notation
\begin{equation*}
L:= \sO_X((1-p^{ne})(K_{X} + \Delta)). %%\qquad K:= \sO_V((1-p^{ne})K_V) \qquad N:= L \otimes f^* K
\end{equation*}
We have the evaluation-at-1 map $F^{ne}_* \O_V \cong \sHom_{\O_V}(F^{ne}_* \O_V, \O_V) \to \O_V$ which we identify with the trace of $V$ based on our previous assumption $\omega_V \cong \O_V$.  This map pulls back to $X$ to provide us with a map which we also denote as $\Tr_{F_V}^{ne} : \O_{X_{V^{ne}}} \cong \O_X \otimes f^* F^{ne}_* \O_V \to \O_X \otimes_{\O_X} f^* \O_V \cong \O_X$.
Then there is a diagram as follows, which is commutative up to multiplication by a unit \cf \cite[Lemma 3.9]{SchwedeFAdjunction}.
\begin{equation*}
\xymatrix@R=18pt{
F^{ne}_* L \ar@{^{(}->}[r] \ar@{=}[d]& F^{ne}_* \O_X((1-p^{ne})K_X) \ar[rr]^-{\mathrm{Tr}_{F_X}}  \ar@{=}[d] && \sO_X  \\
F^{ne}_* \O_X( (1-p^{ne})(K_{X/V} + \Delta)) \ar@{^{(}->}[r] & F^{ne}_* \O_X((1-p^{ne})K_{X/V}) \ar[rr]_-{\Tr_{F^{ne}_{X/V}}}  & & \O_X \otimes_{\O_X} f^* F^{ne}_* \O_V \ar[u]_{\Tr_{F_V}^{ne}}
%(\id_X \times F^{ne}_V )_* \left( f_{V^{ne}}^* K \otimes F^{ne}_{X^{ne}/V^{ne},*} L \right)  \ar[rrrr]^(0.6){(\id_X \times F^{ne}_V )_* \left( f_{V^{ne}}^* K \otimes \phi_\Delta  \right)} & & & & \ar[u]^{\mathrm{Tr}_{F_V}} (\id_X \times F^{ne}_V )_* (f_{V^{ne}}^* K  )
}
\end{equation*}
Applying the functor $f_* ( \_ \otimes_{\O_X} M)$ to the above diagram, we obtain the following.
\begin{equation*}
\xymatrix@R=18pt{
f_* ( (F^{ne}_* L) \otimes_{\O_X} M) \ar[r] \ar@{=}[d] & f_* (M)  \\
f_* ((F^{ne}_* \O_X( (1-p^{ne})(K_{X/V} + \Delta)) ) \otimes_{\O_X} M) \ar[r] & f_* (M \otimes_{\O_X} f^* F^{ne}_* \O_V) \ar[u]_{\kappa}
% F^{ne}_{V,*} \left( K \otimes f_* \left( F^{ne}_{X^{ne}/V^{ne},*} L \otimes M_{V^{ne}} \right) \right)  \ar[r] & \ar[u] F^{ne}_{V,*} \left( K  \otimes \left( f_{V^{ne}}\right)_* M_{V^{ne}} \right) \cong F^{ne}_{V,*} K \otimes f_* (M)
}
\end{equation*}
Note that for every $n \gg 0$ the image of the top horizontal row is $S^0 f_* (\sigma(X, \Delta) \otimes M)$ and for every $n>0$ the image of the bottom horizontal row is $S^0_{\Delta, n} f_* ( M )$.
Hence for every $n \gg 0$,
\begin{equation*}
S^0 f_* (\sigma(X, \Delta) \otimes M) = \im \left( S^0_{\Delta, n} f_* ( M ) \to f_*(M) \right),
\end{equation*}
where the above map is induced by $\mathrm{Tr}_{F_V}$.

Since we want containment for subsheaves of $(f_*(M))_{V^{ne}}$, we localize at a point of $V$. By setting $B:=\Gamma(V,\sO_V)$, $P:=f_*(M)$ and $Q:= F^{ne}_{V,*} ( S^0_{\Delta, n} f_* ( M ))$, we are in the following situation:    if $Q$ is a $B^{1/p^{ne}}$ submodule of $B^{1/p^{ne}} \otimes_B P$ for some $B$ module $P$, we would like to show that $Q \subseteq \theta (Q) \otimes_B B^{1/p^{ne}}$ where $\theta \in \Hom_B(B^{1/p^{ne}},B)$ is the generator. This is simply \autoref{lem:pushforward_pullback}.
\end{proof}

\begin{example}
\label{ex:relative_not_equal_global}
We provide an example where the two sheaves of \autoref{prop:global_relative_relation} have different ranks for every $n \gg 0$. Fix an algebraically closed field $k$ of characteristic $p>0$. Let $X:=\bP^1 \times \bA^1$, $U = \bA^1 \times \bA^1 = \Spec k[y, x] \subseteq X$ and $V:= \bA^1 := \Spec k[x]$. Define the divisor $\Delta:= \frac{1}{p-1} V(y^p -x)$ which is a priori a divisor in $U$, but it happens to have support which is a closed subvariety of $X$ as well. Note that $V(y^p - x)$ is a subvariety of $X$ isomorphic to $\bA^1$. In particular, $(X, \Delta)$ is sharply $F$-pure. Choose now a line bundle $N$ on $X$ which is relatively ample over $V$ and let $M:=N^l$ for some $l \gg 0$.  Then by \cite[Lemma 2.19]{HaconXuThreeDimensionalMinimalModel} $S^0 f_* (\sigma(X, \Delta) \otimes M) = f_* (M)$. Therefore,
\begin{equation}
\label{eq:relative_not_equal_global:global}
(S^0 f_* (\sigma(X, \Delta) \otimes M) )_{V^{ne}} = (f_* (M) )_{V^{ne}}
\end{equation}
On the other hand by \autoref{prop:obvious_containment}, $S^0_{\Delta,n} f_* (M) \subseteq f_* (\sigma_n(X/V, \Delta) \otimes M)$. By \autoref{thm.UniversalNForAllClosedFibers} and \autoref{lem.DeltaPhiRestrictedEqualsDeltaPsi}, for every $n \gg0$ and $s \in V$, $\sigma_n(X/V, \Delta)|_{X_s} = \sigma (X_s, \Delta)$. However, $\Delta|_{X_s}$ is one point with multiplicity $\frac{p}{p-1}$, and hence not sharply $F$-pure. In particular, $\sigma (X_s, \Delta) \neq \sO_{X_s}$ for every $s \in S$. Hence for every $n \gg 0$,  $\sigma_n(X/V, \Delta)|_{X_s} \subsetneq \sO_{X_s}$. It follows by the relative ampleness of $M$ that $f_*(\sigma_n(X/V, \Delta) \otimes M)$ has smaller rank than $f_* (M)$. Then $S^0_{\Delta,n} f_* (M)$ has also smaller rank than $(S^0 f_* (\sigma(X, \Delta) \otimes M) )_{V^{ne}}$ by \autoref{eq:relative_not_equal_global:global}.
\end{example}

The above example has an important corollary, which follows immediately from \autoref{itm:S_f_*_base_change:isomorphism_on_an_open} of   \autoref{prop:S_f_*_base_change} and \autoref{prop:S^0_uniform_stabilization}.

\begin{corollary}
\label{cor:global_pushforward_does_not_restrict}
In general  $S^0 f_* (\sigma(X,\Delta) \otimes M) \otimes k(s)$ is not isomorphic to $S^0(X_s, \sigma(X_s, \Delta_s) \otimes M_s)$. In fact there are examples when the former has strictly bigger dimension than the latter for every closed point $s \in V$.
\end{corollary}

Compare the following proposition with \autoref{prop.SigmaUnderBaseChange}.

\begin{proposition}
\label{prop:global_S_0_f_containment}
Let $\Delta$ be an effective $\bQ$-divisor, such that $(X,\Delta)$ is a pair and assume that $f : X \to V$ is projective. Let $M$ be a line bundle on $X$ and assume that $V$ is regular. Then for every $n \geq m \geq 0$,
\begin{equation*}
S^0 f_{V^{ne},*} ( \sigma(X_{V^{ne}}, \Delta_{V^{ne}}) \otimes M_{V^{ne}}) \subseteq (S^0 f_{V^{me},*} ( \sigma(X_{V^{me}},\Delta_{V^{me}}) \otimes M_{V^{me}}))_{V^{ne}}
\end{equation*}
as subsheaves of $(f_*(M))_{V^{ne}}$. Furthermore, if  the above two sheaves stabilize then
\begin{equation*}
\begin{array}{rl}
& \kappa \Big(F_{V,*}^{(n-m)e} S^0 f_{V^{ne},*} ( \sigma(X_{V^{ne}}, \Delta_{V^{ne}}) \otimes f^* \O_{V}( (1 - p^{(n-m)e}) K_{V^{ne}} ) ) \otimes M_{V^{ne}} )\Big) \\
=  & S^0 f_{V^{me},*} ( \sigma(X_{V^{me}},\Delta_{V^{me}}) \otimes M_{V^{me}})
\end{array}
\end{equation*}
where $\kappa$ is induced by $\Tr_V^{(n-m)e}$.
\end{proposition}

\begin{proof}
Note that by replacing $f : X \to V$ by $f_{V^{me}} : X_{V^{me}} \to V^{me}$, we may assume that $m=0$, and then by replacing $e$ by $(n-m)e$ that $n=1$. As in the previous proof, let us assume that
$V$ is affine, $K_V \sim 0$, and that $\sHom_{\O_V}(F^e_* \O_V, \O_V) \cong F^e_* \O_V$ is a free $\O_V$-module. Consider the following commutative diagram.
\begin{equation*}
\xymatrix@R=18pt{
\left(X_{V^e}\right)^{e} \ar[rrrd]_{F_{X_{V^e}}^{e}} \ar[rr]^-{\nu:=(F_V^e)_{X^{e}} } & &  X^{e} \ar[rrrd]^{F_X^{e}}  \\
&& &  X_{V^e} \ar[rr]_{\eta:=(F_V^e)_X } \ar[d]_{f_{V^e}} & & X \ar[d]_f \\
&& & V^e \ar[rr]_{F_V^e}  & & V
}
\end{equation*}
We fix the notations
\begin{equation*}
\sL_{\Delta}':=\sO_{X_{V^e}}((1-p^{e}) (K_{X_{V^e}} + \eta^* \Delta)),
\qquad
\sL_{\Delta}:=\sO_{X}((1-p^{e}) (K_X +  \Delta)).
\end{equation*}
Then we have
\begin{equation*}
 (1-p^{e}) K_{X_{V^e}}  = (1-p^{e}) (\eta^* K_{X/V} + f_{V^e}^* K_{V^e} ),
\textrm{ and }
 (1-p^{e}) K_{X}  = (1-p^{e}) ( K_{X/V} + f^* K_{V} ) .
\end{equation*}
Therefore,
\begin{itemize}
 \item applying pullback of $\Tr_{F_V^e}$ yields a homomorphism $\nu_* \sL_\Delta' \to \sL_\Delta$, which then induces a homomorphism $ \eta_* F_{X_{V^e},*}^{e} \sL_\Delta' \cong F_{X,*}^{e} \nu_* \sL_\Delta' \to F_{X,*}^{e} \sL_\Delta$,
\item applying $\Tr_{F_{X_{V^e}}^{e}}$ yields a homomorphism $F_{X_{V^e},*}^{e} \sL_\Delta' \to \sO_{X_{V^e}}$,
\item applying $\Tr_{F_X^{e}}$ yields a homomorphism $F_{X,*}^{e} \sL_\Delta \to \sO_X$,
\item by the assumption $K_V \sim 0$, $\Tr_{F_V^e}$ corresponds to a homomorphism $F_{V,*}^e \sO_V \to \sO_V$ generating $\Hom_{\sO_V} (F_{V,*}^e \sO_V, \sO_V)$ and
\item the previous homomorphism also induces a pullback homomorphism $\eta_* \sO_{X_{V^e}} \to \sO_X$.
\end{itemize}
Furthermore, the above homomorphisms fit into the following diagram
% \begin{equation*}
% \xymatrix{
% F_{X_{V^e},*} \sL_{\Delta,e}=F_{X^e/V^e,*} \nu_* \sL_{\Delta,e} \ar[r] \ar[rd] & F_{X^e/V^e,*}   \sL_{\Delta} \ar[d] \\
% %
% & \sO_{X_{V^e}}
% }
% \qquad \qquad
% \xymatrix{
% \eta_* F_{X^e/V^e,*} \sL_{\Delta} = F_{X,*} \sL_\Delta  \ar[d] \ar[rd] \\
% \eta_* \sO_{X_{V^e}} \ar[r] & \sO_X
% }
% \end{equation*}
\begin{equation*}
\xymatrix@R=18pt{
\eta_* F_{X_{V^e},*}^{e} \sL_\Delta' \ar[r] \ar[rd] & F_{X,*}^{e}   \sL_{\Delta} \ar[rd] \\
& \eta_*  \sO_{X_{V^e}} \ar[r] & \sO_X
}
\end{equation*}
The diagram is a composition of trace maps (restricted to smaller domains determined by $\Delta$), hence it is commutative.
%Note that by \cite[Lemma 3.9]{SchwedeFAdjunction}, locally both  compositions of arrows in the above diagram yield generators of the $\eta_* F_{X_{V^e},*}^{e} \sO_{X_{V^e}}$ module $\sHom_{\sO_X} ( \eta_* F_{X_{V^e},*}^{e} \sO_{X_{V^e}}((p^e -1) \nu^* \Delta), \sO_X)$. Hence the diagram is commutative up to multiplication by a unit of $\sO_X$ (which by projectivity of $f$ is just multiplication by a unit of $\sO_V$). Since we only consider images of homomorphisms appearing in the above diagram, we may treat the as diagram commutative.
Applying now $f_*(\_ \otimes M)$ to the above diagram and using the projection formula we obtain the following commutative diagram
\begin{equation*}
\xymatrix@R=18pt{
f_*  \eta_* ( F_{X_{V^e},*}^{e} \sL_\Delta' \otimes M_{V^e}) \ar[r] \ar[rd] &  f_*( F_{X,*}^{e}   \sL_{\Delta} \otimes M)  \ar[rd] \\
& f_* \eta_* M_{V^e} \ar[r] & f_*(M).
}
\end{equation*}
Observing that $f_* \eta_*  =  F_{V,*}^ef_{V^e,*}$ yields another commutative diagram
\begin{equation*}
\xymatrix@R=18pt{
F_{V,*}^e f_{V^e,*}  ( F_{X_{V^e},*}^{e} \sL_\Delta' \otimes M_{V^e}) \ar[r] \ar[rd] &  f_*( F_{X,*}^{e}   \sL_{\Delta} \otimes M) \ar[rd] \\
& F_{V,*}^e f_{V^e,*}  M_{V^e} \ar[r] & f_*(M).
}
\end{equation*}
Note now that the top horizontal arrow is split, because it is induced from $\nu_* \sL_\Delta' \to \sL_\Delta$ which is split as well. Therefore, if we define $P:= f_*(M)$ and
\begin{equation*}
S:= \im \left( f_*( F_{X,*}^{e}   \sL_{\Delta} \otimes M) \to f_*(M) \right) \qquad \quad Q:= \im \left(  f_{V^e,*}  ( F_{X_{V^e},*}^{e} \sL_\Delta' \otimes M_{V^e}) \to f_{V^e,*}  M_{V^e}  \right),
\end{equation*}
then we have $Q \subseteq P \otimes_A A^{1/p^e}$, $S \subseteq P$ and  $( \id_P \otimes \Tr_{F_{V}^e}) (Q) = S$. Therefore, by Lemma \ref{lem:pushforward_pullback}, $Q \subseteq S \otimes_A A^{1/p^e}$. Since this holds for any $n$, the statement of the proposition follows.
\end{proof}

\begin{proposition}
\label{prop:global_relative_relation_other_way}
Let $\Delta$ be an effective $\bQ$-divisor, such that $(X,\Delta)$ is a pair and assume that $f : X \to V$ is projective and $V$ is regular. Let $M$ be a line bundle on $X$. Then for every $n \gg 0$,
\begin{equation*}
S^0 f_{V^{ne},*} ( \sigma(X_{V^{ne}}, \Delta_{V^{ne}}) \otimes M_{V^{ne}}) \subseteq S_{\Delta,ne}^0 f_*(M)
\end{equation*}
as subsheaves of $(f_*(M))_{V^{ne}}$.
\end{proposition}

\begin{proof}
As before, let us assume that
$V$ is affine, $K_V \sim 0$, and that $\sHom_{\O_V}(F^i_* \O_V, \O_V) \cong F^i_* \O_V$ is a free $\O_V$-module for all $i \geq 0$. By \autoref{lem.FactorizationOfSigmaImages} and \autoref{lem.ChoiceOfDeltaPhiEqualsDeltaGamma} there is a commutative diagram %as follows.
\begin{equation*}
\xymatrix{
F_{X_{V^{ne}},*}^{ne} \sO_{X_{V^{ne}}} \left( (1-p^{ne}) \left(K_{X_{V^{ne}}} + \Delta_{V^{ne}} \right) \right) \ar[r] \ar[rd] & F_{X^{ne}/V^{ne},*}^{ne} \sO_X \left( (1-p^{ne}) \left(K_{X/V} + \Delta \right) \right) \ar[d] \\
& \sO_{X_{V^{ne}}}.
}
\end{equation*}
Furthermore, the top horizontal arrow in the above diagram is split surjective. Therefore, after applying $f_{V^{ne},*} \left( \_ \otimes  M_{V^{ne}} \right)$ the image of the vertical map still agrees with the image of the diagonal map. The former is exactly $S_{\Delta,n}^0 f_*(M)$, while the latter contains  $S^0 f_{V^{ne},*} ( \sigma(X, (\Delta)_{V^{ne}}) \otimes M)$, since it is one of the terms in the intersection defining $S^0 f_{V^{ne},*} ( \sigma(X, (\Delta)_{V^{ne}}) \otimes M)$.
\end{proof}

\begin{remark}
Assuming that $f : X \to V$ is projective and $M - K_X - \Delta$ is ample, it would be natural to ask whether $\Image\big(S^0 f_{V^{ne}, *} (\sigma(X_{V^{ne}, \Delta_{V^{ne}}} \tensor M_{V^{ne}}) \tensor M_{V^{ne}})\to f_{V^{ne},*} M_s \big)$ equals $S^0(X_s, \sigma(X_s, \Delta_s) \tensor M_s)$ for all perfect points $s \in V$, in analogy with \autoref{cor.SigmaRestrictsToAllPoints}.  We suspect this is true  but will not try to prove it here.  %Based on what we already know so far, we simply need to show that we always have the containment $\supseteq$.
\end{remark}

%\begin{proof}
%We already have that
%\[
%\begin{array}{rl}
%& \Image\big( S^0 f_{V^{ne},*} ( \sigma(X_{V^{ne}}, \Delta_{V^{ne}}) \otimes M_{V^{ne}}) \to f_{V^{ne},*} M_s \big)\\
%\subseteq & \Image\big( S_{\Delta,ne}^0 f_*(M) \to f_{V^{ne},*} M_{s^{ne}}\big)\\
%= & S^0(X_{s^{ne}}, \sigma(X_{s^{ne}}, \Delta_{s^{ne}}) \tensor M_{s^{ne}})\\
%\cong & S^0(X_s, \sigma(X_s, \Delta_s) \tensor M_s)
%\end{array}
%\]
%by \autoref{prop:global_relative_relation_other_way} and \autoref{prop:S^0_uniform_stabilization}.  Therefore, it is sufficient to show that
%\[
%\begin{array}{rl}
%& \Image\big(S^0 f_{V^{ne},*} ( \sigma(X, (\Delta)_{V^{ne}}) \otimes M) \to f_{V^{ne},*} M_s\big)\\
%\supseteq & S^0(X_{s^{ne}}, \sigma(X_{s^{ne}}, \Delta_{s^{ne}}) \tensor M_{s^{ne}}) \\
%\cong & S^0(X_s, \sigma(X_s, \Delta_s) \tensor M_s).
%\end{array}
%\]
%\end{proof}

\begin{corollary}
\label{cor:semi_positivity_of_S_0_f}
Let $f : (X, \Delta) \to V $ be a projective morphism from a pair with $V$ regular and projective over a perfect field $k$. Further, suppose that $M$ is a line bundle on $X$ such that $M - K_{X/V} - \Delta$ is nef and $f$-ample (here $M$ denotes a Cartier divisor corresponding to $M$) and that $\rk S^0f_* (\sigma(X,\Delta) \otimes M)$ equals the general value of $H^0(X_s, \sigma(X_s, \Delta_s) \otimes M_s)$. Then $S^0f_* (\sigma(X,\Delta) \otimes M)$ is weakly positive. In particular, if $V$ is a smooth curve then it is a nef vector bundle.
\end{corollary}

\begin{proof}
From the assumption $\rk S^0f_* (\sigma(X,\Delta) \otimes M)$ equals the general value of $H^0(X_s, \sigma(X_s, \Delta_s) \otimes M_s)$ follows that the inclusion of \autoref{prop:global_relative_relation} is generically an isomorphism for every $n \gg 0$. However, since $S^0_{\phi^n} f_*(M)$ is weakly-positive for every $n \gg 0$ by \autoref{thm:global_generation_of_S_f_*}, we obtain by the above generically isomorphic inclusion that $(S^0f_* (\sigma(X,\Delta) \otimes M))_{V^{ne}}$ is also weakly positive. Then weak-positivity of $S^0f_* (\sigma(X,\Delta) \otimes M)$ follows from \autoref{lem:weakly_positive_finite}.
\end{proof}

%% file: Appendix.tex
\appendix
\section{Relative Serre's condition}

Here we collect the statements of \cite{HassettKovacsReflexivePullbacks} and other sources, that are important for the current paper for ease of reference.  In some cases, we also state them in the greater generality that we need.  All schemes are Noetherian, excellent and possess dualizing complexes and all maps are separable.

\begin{definition}
Let $r > 0$ be an integer. A coherent sheaf $\sE$ on a Noetherian scheme $X$ is $S_r$, if for every $x \in X$,
\begin{equation*}
\depth_{\sO_{X,x}} \sE_x \geq  \min \{ r, \dim_{\sO_{X,x}} \sE_x \}.
\end{equation*}
The sheaf $\sE$ is said to have \emph{full support}, if $\Supp \sE = X$. It is \emph{reflexive} if the natural map $\sE \to \sE^{**} := \sHom_{\O_X}(\sHom_{\O_X}(\sE, \O_X), \O_X)$ is an isomorphism.
\end{definition}

\begin{definition}
If $f : X \to V$ is a morphism of Noetherian schemes and $\sE$ is a coherent sheaf on $X$ flat over $V$, then $\sE$ is $S_r$ over $V$ if $\sE|_{X_v}$ is $S_r$ for every $v \in V$. That is, for every $x \in X$,
\begin{equation*}
\depth_{\sO_{X_{f(x)}}} (\sE|_{X_{f(x)}} )_x \geq \min \left\{ r , \dim_{\sO_{X_{f(x)}}} (\sE|_{X_{f(x)}} )_x \right\}
\end{equation*}

\end{definition}

We now recall two results from EGA about how depth and dimension behave in families.

\begin{proposition} \cite[Proposition 6.3.1]{EGAIV}
\label{prop.EGADepthOverBase}
Let $\phi : A \to B$ be a local homomorphism of Noetherian local rings, $k$ the residue field of $A$ and $M$ and $N$ are finite $A$ and $B$ modules respectively. If $N$ is a flat non-zero $A$-module, then
\begin{equation*}
\depth_B (M \otimes_A N ) = \depth_A (M) + \depth_{B \otimes_A k} (N \otimes_A k).
\end{equation*}
\end{proposition}

\begin{proposition} \cite[Corollaire 6.1.2]{EGAIV}
\label{prop.EGADimOverBase}
Let $\phi : A \to B$ be a local homomorphism of Noetherian local rings, $k$ the residue field of $A$ and $M$ and $N$ are non-zero finite $A$ and $B$ modules respectively. If $N$ is a flat $A$-module, then
\begin{equation*}
\dim_B (M \otimes_A N ) = \dim_A (M) + \dim_{B \otimes_A k} (N \otimes_A k).
\end{equation*}
\end{proposition}

%\todo{{\bf Wenliang:} What's the difference between A.3 and A.4?\\{\bf Karl:} It looks like dimension vs depth.  Perhaps these should be combined into one statement though?}

From these we obtain the following corollary.
\begin{corollary}
\label{cor.RelVanishingLocalCohomology}
Let $f : X \to V$ be a morphism of Noetherian schemes and $\sE$  a coherent sheaf on $X$ flat over $V$, then  $\sE$ is $S_r$ over $V$ if and only if for every $x \in X$,
\begin{multline*}
\min \{i | H^i_x (\sE) \neq 0 \}  \geq \depth_{\sO_{V,v}} \sO_{V,v} + \min \{r,  \dim \sE_x - \dim \sO_{V,v} \}
\\  \left( = \depth_{\sO_{V,v}} \sO_{V,v} + \min \left\{r,  \dim \left(\sE|_{X_{v}} \right)_x  \right\} \right)
\end{multline*}
where $v : = f(x)$.
\end{corollary}

\begin{proof}
To say that $\sE|_{X_v}$ has depth equal to $t$ at $x$ is equivalent to asserting that $\depth_x \sE = (\depth_{\sO_{V,v}}\sO_{V,v}) + t$ by  \autoref{prop.EGADepthOverBase}.  The result follows.
\end{proof}

The following vanishing allows us to extend sections over sets of relative codimension 2.

\begin{proposition}\cite[Proposition 3.3]{HassettKovacsReflexivePullbacks}
\label{prop:vanishing_local_cohomology}
Let $f : X \to V$ be a morphism of Noetherian schemes  and $\sE$ a coherent sheaf on $X$ flat and $S_r$ over $V$. Let $Z \subseteq X$ be a closed subscheme of $X$ such that we have $\codim_{\Supp \sE_v} (\Supp \sE_v \cap Z_v) \geq r$ for every $v \in V$. Then for each $0 \leq i < r$, $\sH_Z^i(\sE)=0$.
\end{proposition}

\begin{proof}
We follow the proof in \cite[Proposition 3.3]{HassettKovacsReflexivePullbacks}.  First, we certainly have a spectral sequence $E_2^{p,q} := H^{p}_x \circ \sH^q_Z \Rightarrow H^{p+q}_x$.  We then induct on $r$, the base case of $r = 0$ being trivial.  Now assume that $\sH^0_Z(\sE) = \ldots = \sH^{r-2}_Z(\sE) = 0$ but $\sH^{r-1}_Z(\sE) \neq 0$.  Then for any $x \in Z$ which is a generic point\footnote{a minimal prime in the language of commutative algebra} of the support of $\sH^{r-1}_Z(\sE)$, we have that $H^0_x(\sH^{r-1}_Z(\sE)) \neq 0$ (note here that $x$ is probably not a closed point).  Additionally, since $x$ is a point of $Z \cap \Supp(\sE)$, we have that $\dim (\sE|_{X_v})_x \geq r$.

On the other hand, for all $i > 0$ we have $H^i_x(\sH^{r-1-i}_Z(\sE)) = 0$, and so by the spectral sequence $H^{r-1}_x(\sE) \cong H^0_x(\sH^{r-1}_Z(\sE)) \neq 0$.  This contradicts \autoref{cor.RelVanishingLocalCohomology}.
%, with adding the extra condition that when  $x \in \Supp \sH_Z^k(\sE)$ is chosen, then it has to be a general point of $\supp \sH_Z^k(\sE)$.
\end{proof}

We now obtain a relative version of Hartog's phenomena, just as in \cite{HassettKovacsReflexivePullbacks}.

\begin{proposition}\cite[Proposition 3.5, 3.6.1]{HassettKovacsReflexivePullbacks}
\label{prop:pushforward_relative_S_2_or_reflexive}
Let $f : X \to V$ be a morphism of Noetherian schemes and that $\sE$ a coherent sheaf on $X$ which satisfies one of the following two conditions:
\begin{itemize}
\item[(i)]  $\sE$ is reflexive and $f$ is flat and relatively $S_2$, or
\item[(ii)]  $\sE$ is of full support, flat and $S_2$ over $V$.
\end{itemize}
Let $j : U \hookrightarrow X$ be an open set such that for $Z := X \setminus U$, $\codim_{X_v} Z_v \geq 2$ for every $v \in V$. Then the following natural map is an  isomorphism:
\begin{equation*}
\sE \to j_* \sE|_U.
\end{equation*}
\end{proposition}

\begin{proof}
First, assume that $\sE$ is reflexive. Consider a presentation of $\sE^*$ by locally free sheaves
\begin{equation*}
\xymatrix{
\sE_2 \ar[r] & \sE_1 \ar[r] & \sE^* \ar[r] & 0
} \text{ and the dual }\xymatrix{
0 \ar[r] & \sE \ar[r] & \sE_1^* \ar[r] & \sE_2^*
}
\end{equation*}
%By dualizing we obtain
%\begin{equation*}
%\xymatrix{
%0 \ar[r] & \sE \ar[r] & \sE_1^* \ar[r] & \sE_2^*
%}
%\end{equation*}
Then by applying \autoref{prop:vanishing_local_cohomology} to $\sO_X$ and the applying local cohomology to the above exact sequence yields that
\begin{equation}
\label{eq:pushforward_relative_S_2_or_reflexive:vanishing}
\sH^1_Z(\sE) = \sH^0_Z(\sE)=0.
\end{equation}
If instead, $\sE$ is flat, S2 and has full support, then we still have the vanishing \autoref{eq:pushforward_relative_S_2_or_reflexive:vanishing} by \autoref{prop:vanishing_local_cohomology}.
%Hence, since either $\sE$ is flat and $S_2$ over $V$ or reflexive, the same vanishing of local cohomology (i.e., \autoref{prop:vanishing_local_cohomology} or \autoref{eq:pushforward_relative_S_2_or_reflexive:vanishing}) holds.
Consider then the exact sequence
\begin{equation*}
\xymatrix{
 0 \ar[r] & \sH_Z^0(\sE) \ar[r] & \sE \ar[r] & j_* \sE \ar[r] & \sH_Z^1(\sE).
}
\end{equation*}
Applying \autoref{eq:pushforward_relative_S_2_or_reflexive:vanishing} concludes our proof.
\end{proof}

\begin{corollary}
\label{cor:relative_S_2_morphism_reflexive_has_extension_property}
\cite[Proposition 3.6.2]{HassettKovacsReflexivePullbacks}
Let $f : X \to V$ be a flat, $S_2$ morphism of Noetherian schemes, $\sE$ a coherent $X$ satisfying either \autoref{prop:pushforward_relative_S_2_or_reflexive}(i) or (ii).  Let $\sF$ be another coherent sheaf satisfying \autoref{prop:pushforward_relative_S_2_or_reflexive}(i) or (ii).  Let $j : U \hookrightarrow X$ be an open set, such that for $Z := X \setminus U$, $\codim_{X_v} Z_v \geq 2$ for every $v \in V$. Assume also that $\sE|_U \cong \sF|_U$. Then $\sE \cong \sF$.
\end{corollary}

\begin{proof}
In any case, $j_* \sE|_U = \sE$ and $j_* \sF|_U = \sF$ and so the result follows.
\end{proof}

\begin{proposition} \cite[Corollary 3.7]{HassettKovacsReflexivePullbacks}
\label{prop:extension_of_reflexive_reflexive}
Let $f : X \to V$ be a flat, $S_2$ morphism of Noetherian schemes and $j : U \hookrightarrow X$  an open set, such that for $Z := X \setminus U$, $\codim_{X_v} Z_v \geq 2$ for every $v \in V$. Assume also that $\sE$ is a reflexive, coherent sheaf on $U$. Then $j_* \sE$ is reflexive.
\end{proposition}

\begin{proof}
Choose a coherent subsheaf $\sF \subseteq j_* \sE$ such that $\sF|_U = \sE$ \cite[Ex II.5.15]{Hartshorne}. Then $j_* \sE$ is reflexive by the following computation.
\begin{equation*}
\sF^{**} \cong \underbrace{j_*(\sF^{**}|_U)}_{\textrm{by \autoref{prop:pushforward_relative_S_2_or_reflexive}}} \cong \underbrace{j_* \sE^{**}}_{\sE \cong \sF|_U} \cong \underbrace{j_* \sE}_{\textrm{$\sE$ is reflexive}}
\end{equation*}
This completes the proof.
\end{proof}

\begin{proposition}
\label{prop:relative_canonical_reflexive}
If
$f : X \to Y$ is a projective, flat, relatively S2 and G1, equidimensional morphism, then $\omega_{X/Y}$
is reflexive.
\end{proposition}

\begin{proof}
The proof given in \cite[Lemma 4.8]{PatakfalviSchwedeDepthOfFSings} works.
\end{proof}